\documentclass[aos]{imsart}
\usepackage{amsmath,amssymb,amsthm,mathtools,bm}
\usepackage{enumitem}
\usepackage{booktabs,array}
\usepackage{graphicx}
\usepackage{natbib}
\usepackage{comment}
\usepackage{url}

\usepackage{hyperref}
\hypersetup{hidelinks}

\usepackage{lineno}
\usepackage{tikz}
\usetikzlibrary{arrows.meta,calc,patterns,positioning}

\allowdisplaybreaks
\numberwithin{equation}{section}

\newtheorem{theorem}{Theorem}[section]
\newtheorem{proposition}[theorem]{Proposition}
\newtheorem{lemma}[theorem]{Lemma}
\newtheorem{corollary}[theorem]{Corollary}
\newtheorem{assumption}[theorem]{Assumption}
\newtheorem{remark}[theorem]{Remark}
\newtheorem{example}[theorem]{Example}

\newcommand{\Q}{\mathbb{Q}}
\newcommand{\R}{\mathbb{R}}
\newcommand{\N}{\mathbb{N}}
\newcommand{\E}{\mathbb{E}}
\newcommand{\Prob}{\mathbb{P}}
\newcommand{\1}{\mathbf{1}}
\newcommand{\calF}{\mathcal{F}}
\newcommand{\calN}{\mathcal{N}}
\newcommand{\calA}{\mathcal{A}}

\newcommand{\calC}{\mathcal{C}}

\newcommand{\calH}{\mathcal{H}}

\newcommand{\calL}{\mathcal{L}}

\newcommand{\calX}{\mathcal{X}}
\newcommand{\Thetao}{\Theta^{\circ}}
\newcommand{\norm}[1]{\left\lVert #1 \right\rVert}
\newcommand{\abs}[1]{\left\lvert #1 \right\rvert}

\newcommand{\op}{\mathrm{op}}
\newcommand{\TV}{\mathrm{TV}}
\newcommand{\KL}{D_{\mathrm{KL}}}

\newcommand{\diff}{\,\mathrm{d}}
\newcommand{\dd}{\mathrm{d}}

\DeclareMathOperator{\Var}{Var}

\DeclareMathOperator{\dist}{dist}

\DeclareMathOperator{\diag}{diag}

\begin{document}
\title{Optimal Estimating Equations for Compact-Memory Hawkes Processes}
\runtitle{Estimating Equations for Hawkes Processes}
\begin{aug}
\author{\fnms{Louis}~\snm{Davis}\thanksref{t1}\ead[label=e1]{ldavis2@stanford.edu}}
\and
\author{\fnms{Conor}~\snm{Kresin}\thanksref{t1}\ead[label=e2]{conor.kresin@otago.ac.nz}}
\thankstext{t1}{Both authors contributed equally to this work.}
\address{Department of Statistics, Stanford University\printead[presep={,\ }]{e1}}
\address{Department of Mathematics and Statistics, University of Otago\printead[presep={,\ }]{e2}}
\end{aug}

\begin{abstract}
Likelihood is standard for Hawkes-process inference, while less computationally demanding methods have largely developed separately. We show that least squares, Tak\'acs--Fiksel, and related moment-based estimators form a single class of compensator-based estimating equations, with the likelihood score as the efficient benchmark. For fixed-dimensional multivariate Hawkes processes with compact memory, nonlinear positive links, and signed kernels allowing inhibition, every suitably regular predictable functional of a fixed lag window yields an unbiased estimating equation when integrated against \(\dd N-\lambda\,\dd t\). Under common regularity, identification, and rank conditions, estimators based on every admissible finite library achieve uniform high-probability and pointwise almost-sure \(\mathcal O(\sqrt{\log(T)/T})\) rates, asymptotic normality with Godambe covariance, and admit feasible two-step optimal weighting. A projection identity quantifies each library's exact efficiency loss as the score information outside its predictable span; a two-point bound shows the root-\(T\) scale cannot be improved uniformly. Although compact memory localizes the intensity rather than the stationary law, exponential forgetting yields Bernstein-type concentration and transfers the theory to nonstationary starts after a logarithmic burn-in. Within this scope, the compensator class is exhaustive for finite-library comparisons: it contains the score, gives admissible libraries common guarantees, and quantifies their efficiency gaps exactly.
\end{abstract}

\maketitle


\section{Introduction}\label{sec:introduction}
The martingale compensator identity 
\begin{equation}\label{eqn:compensator-identity}
    \E_{\theta^\star}\{\dd N(t)-\lambda(t;\theta^\star)\,\dd t\mid\calF_{t-}\}=0 
\end{equation}
is the natural moment identity for a point process $N$ characterised by conditional intensity $\lambda(t;\theta^\star)$ at the true parameter $\theta^\star$ \citep{daley2008}. When integrated against a predictable weight $H(t;\theta)$, Equation~\eqref{eqn:compensator-identity} defines a broad family of unbiased martingale estimating equations of the form 
\begin{equation}\label{eqn:martingale-EE}
    \Psi_T^H(\theta)=\int_0^T H(t;\theta)\{\dd N(t)-\lambda(t;\theta)\,\dd t\}.
\end{equation}
Members of this family of estimators include least squares, maximum likelihood (MLE), Tak\'acs--Fiksel and generalized method of moments (GMM). The family \eqref{eqn:martingale-EE} is the point process analogue to Godambe's estimating-function theory \citep{godambe1991}.

Such estimating equations have long been employed for spatial point processes: the spatial analogue to Equation~\eqref{eqn:compensator-identity} is referred to as the Campbell--Mecke  identity in the Poisson case \citep{daley2008}, and the Georgii--Nguyen--Zessin (GNZ) formula more generally \citep{georgii1976canonical,xanh1979integral}. In the case of spatial Gibbs processes, the likelihood has an intractable normalising constant, and Tak\'acs--Fiksel estimation \citep{fiksel1984estimation} (of which the Stoyan--Grabarnik statistic is most well known \citep{stoyangrabarnik1991,kresinschoenberg2023}) is often applied alongside Godambe-optimal quasi-likelihood and various GMM approaches \citep{coeurjolly2016optimal,guan2015quasi}. For Hawkes processes \citep{hawkes1971,bremau1996}, the situation is inverted: the likelihood is available in closed form, and has dominated both theory and practical application \citep{ogata1978,clinet2017,kwan2025}. Perhaps for this reason, a theory of predictable compensator moments, containing the score as one member and pricing every other member against it, has remained undeveloped.

There are good reasons to develop one. Likelihood is asymptotically efficient but computationally expensive: each evaluation is quadratic in the event count for non-Markovian Hawkes processes, and a $D$-variate model has $\mathcal{O}(D^2)$ parameters corresponding to the full interaction kernel. Because of this, dense multivariate settings are often assumed away, or sparsity-inducing estimation strategies are used to make estimation tractable \citep{hansen2015}. Recent work has accordingly turned to penalized least squares \citep{reynaud2010adaptive,hansen2015,bacry2020sparse}, spectral estimation \citep{cheysson2022spectral, tang2026spectral}, and nonparametric Bayesian modelling \citep{sulem2024bayesian, donnetRousseau}. Likelihood-based approaches have made recent headway for nonlinear Hawkes processes \citep{bonnet2023inference}. But each of these approaches provides a remedy for a relatively narrow gap. 

At the cost of assuming compact memory (\textit{\`a la} \cite{hansen2015} and \cite{donnetRousseau}), we develop a unifying fixed-dimensional theory for a large class of Hawkes processes and the complete family of compensator moment-based estimators.  Our theory has another dividend: its standardized moments are precisely the summary statistics that make simulation-based inference \citep{mcfadden1989,cranmer2020} work (see Remark~\ref{rem:sbi-gmm}). In the compact-memory setting the conditional intensity is a predictable functional of a realisation on a fixed lag window $[t-A,t)$. We assume the parameter dimension $p=\dim(\Theta)$ and $D$ the number of components are fixed, i.e. do not grow with $T$. We further assume that $q$, the number of scalar moment conditions generated by the weight $H$ (equivalently the number of rows of $H$) is fixed (although $q>p$ is permitted without necessitating regularisation).

A row of the estimating weight $H$ can be any smooth predictable feature of the realisation viewed on $[t-A,t)$, controlled by a polynomial local-count envelope. This means that binned counts, lagged summaries, inter-arrival times, filtered convolutions, and generally any smooth transformation are all admissible members of the same finite feature library. In particular, as noted above, this includes the likelihood score and least squares (see Example~\ref{ex:weight-families-body}). The model class our results apply to allows for nonlinear positive link functions \citep{bremau1996} as well as signed interaction kernels (allowing for inhibition, cf. \citep{bonnet2023inference}). Further, our theory accommodates nonstationary initial conditions, allowing inference without assuming equilibrium at the observation origin.

Under one single set of mild, checkable assumptions, we prove uniform high probability and almost sure $\mathcal O(\sqrt{\log(T)/T})$ rates for compact window estimators of the form of \eqref{eqn:martingale-EE}. Further, we prove asymptotic normality with sandwich covariance, a feasible two-step optimal weighting, and a Godambe projection theorem that says that no admissible finite moment library carries more information than the likelihood score, and that the cost for not using it is exactly the variance of the score component orthogonal to the library's predictable span. As an immediate corollary, the same results apply to the stationary MLE, and a two-point Le Cam bound shows that the root-$T$ scale cannot be uniformly improved.

The conclusion is clean: even for dependent data, with appropriate concentration, i.i.d. behavior reappears. Rates, sandwich covariance and feasible optimal weighting, and most importantly optimal likelihood efficiency present just as they do in the i.i.d. case, where the Cram\'er Rao lower bound is only attained by the optimally weighted GMM estimator when the moment library spans the score \citep{chamberlain1987asymptotic,carrasco2014asymptotic}. The technical inputs to achieve this i.i.d.-like behavior rely on the Poisson embedding and exponential forgetting arguments of \cite{bremau1996} (which provide a way around the core challenge of our setting, namely that compact support localizes the intensity but not the stationary law) and the compact window functional coupling argument of \cite{hansen2015}. We note the behavior of the MLE for Hawkes processes is well understood in the existing literature (for instance in the univariate case \citep{ogata1978}, the multivariate Markovian case \citep{clinet2017}). What we provide is a comprehensive and unified treatment, in the compact-memory setting, of a larger estimator class containing the MLE.

Section~\ref{sec:main} details the model, estimator class, standing assumptions and main results. Section~\ref{sec:scope} then contextualizes the scope of our results, and the costs of the assumptions stated in the previous section, verifying them on two model families. Section~\ref{sec:concentration} discusses the probabilistic tools needed to obtain concentration, and a simulation study and discussion are provided in Sections~\ref{sec:simulation} and \ref{sec:discussion}. All proofs are contained in the supplement.

\section{Model, estimators, and main results}\label{sec:main}
In this section we describe the model class, estimator family, and then standing assumptions and main statistical results. The implicit scope, limitations and verification that such assumptions are reasonable is deferred to Section~\ref{sec:scope}, and the necessary technical probabilistic tools are detailed in Section~\ref{sec:concentration}.

\subsection{Compact-memory Hawkes processes}\label{sec:model}
Let $N=(N_1,\ldots,N_D)$ be a simple stationary point process on $\R$ with natural filtration $\mathcal{F}_t$.  For $\theta\in\Theta\subset\R^p$, the $i^\text{th}$ component intensity is
\begin{equation}\label{eq:model}
\lambda_i(t;\theta) = \phi_{\theta,i}\left( \nu_{\theta,i} + \sum_{j=1}^D\int_{[t-A,t)}h_{\theta,ij}(t-s)\,\dd N_j(s) \right),
\qquad i=1,\ldots,D ,
\end{equation}
where $\nu_{\theta,i}$ is the baseline rate, and the interaction kernels $h_{\theta,ij}(u)=\tilde h_{\theta,ij}(u)\1\{0\le u\le A\}$ are compactly supported and can take either sign (positive values represent excitation, negative values inhibition). Note that $A$ is fixed and known. The smooth positive link function $\phi_{\theta,i}$ ensures that the intensity is positive; in the canonical linear case \citep{hawkes1971}, $\phi_{\theta,i}(x)=x$. Under the spectral radius condition Assumption~\ref{ass:stability}, a multivariate Hawkes process characterized by \eqref{eq:model} exists and is stationary \citep{bremau1996}. Compact support means only that \(\lambda_i(t;\theta)\) is a finite-memory functional of \(N|_{[t-A,t)}\).  It does not imply stationarity of the law, nor Poisson-type independence of disjoint time windows.

Next, we introduce some notation. The true parameter $\theta^\star$ lies in a compact set \(\Theta_0\Subset\Theta^\circ\).  The conditional log-likelihood on \([0,T]\) is
\begin{equation}\label{eq:likelihood}
\ell_T(\theta) = \sum_{i=1}^D \left( \int_0^T\log\lambda_i(t;\theta) \dd N_i(t) - \int_0^T\lambda_i(t;\theta)\dd t \right).
\end{equation}
For a configuration \(N\), let \(G_tN\) denote \(N|_{[t-A,t)}\) recentred on \([-A,0)\), and write \(C(x)=x([-A,0))\) for a local configuration on \(x\); the map \(G_t\) standardizes the time-\(t\) window to the reference window \([-A,0)\), so that intensities, weights, and their stationary expectations can all be written as fixed functionals applied to \(G_tN\). We denote the local count 
\begin{equation*}
    \mathfrak C_t=N([t-A,t))=\sum_{j=1}^D N_j([t-A,t))
\end{equation*}
noting that the half-open convention is used for predictability.

\subsection{Compensator moments}\label{sec:estimators}
Let \(H(t;\theta)\in\R^{q\times D}\) be predictable with fixed \(q\ge p\).  The estimating map (Equation \eqref{eqn:martingale-EE}) is restated here: 
\begin{equation*}
\Psi_T^H(\theta)
=
\sum_{i=1}^D\int_0^T H_{\cdot i}(t;\theta)
\{\dd N_i(t)-\lambda_i(t;\theta)\,\dd t\}\in\R^q .
\end{equation*}
Write \(M_i^{\theta^\star}(t)=N_i(t)-\int_0^t\lambda_i(s;\theta^\star)\,\dd s\) for the compensated components under \(\Prob_{\theta^\star}\).  The estimating map splits into a martingale term and a drift,
\begin{equation}\label{eq:generic-decomposition}
\Psi_T^H(\theta)
=
\sum_{i=1}^D\int_0^T H_{\cdot i}(t;\theta)\,\dd M_i^{\theta^\star}(t)
+
\int_0^T H(t;\theta)
\{\lambda(t;\theta^\star)-\lambda(t;\theta)\}\,\dd t.
\end{equation}
Under stationarity the population drift is
\begin{equation}\label{eq:gH-def}
g_H(\theta,\theta^\star)
=
\E_{\theta^\star}
\left[H(0;\theta)\{\lambda(0;\theta^\star)-\lambda(0;\theta)\}\right].
\end{equation}
At the truth, \(\Psi_T^H(\theta^\star)\) is a martingale.  For a fixed deterministic positive-definite matrix \(W\), the GMM criterion and estimator are
\begin{equation}\label{eq:gmm-estimator}
Q_{T,W}^H(\theta)=
\left\|T^{-1}\Psi_T^H(\theta)\right\|_W^2,
\qquad
\|x\|_W^2=x^\top Wx ,
\end{equation}
with \(\hat\theta_H=\hat\theta_{H,W}\) any measurable minimizer over \(\Theta\) (Supplement Proposition~\ref{prop:gmm-measurable-selection}).  The likelihood score is the specific weight
\begin{equation}\label{eq:likelihoodweights}
    H_{\rm score}(t;\theta)
    =D_\theta(t;\theta)^\top\Lambda(t;\theta)^{-1},
\end{equation}
where here and throughout
\begin{equation}\label{eq:D-Lambda-def}
D_\theta(t;\theta)=\partial_\theta\lambda(t;\theta),
\qquad
\Lambda(t;\theta)=\diag\{\lambda_1(t;\theta),\ldots,\lambda_D(t;\theta)\},
\end{equation}
so that \(\Psi_T^{H_{\rm score}}(\theta)=\nabla\ell_T(\theta)\): the score is not outside the moment theory but rather it is the efficient member (Proposition~\ref{prop:godambe-optimality}).

\begin{example}[Four canonical weight families]\label{ex:weight-families-body}
The class contains the standard constructions. \emph{Score weights} \(H=D_\theta^\top\Lambda^{-1}\) give maximum likelihood \citep{ogata1978}.  \emph{Derivative weights} \(H=D_\theta^\top\) give the first-order equation for least squares: the point-process least-squares contrast
\begin{equation}\label{eq:ls-contrast}
\gamma_T(\theta)
=\sum_{i=1}^D\left\{\frac1T\int_0^T\lambda_i(t;\theta)^2\,\dd t-\frac2T\int_0^T\lambda_i(t;\theta)\,\dd N_i(t)\right\}
\end{equation}
of \cite{reynaud2010adaptive} (see also \citealp{hansen2015,bacry2020sparse}) satisfies \(\nabla_\theta\gamma_T(\theta)=-(2/T)\,\Psi_T^{D_\theta^\top}(\theta)\), so its first-order equation is the derivative-weight member of the class. The estimator defined by global minimization of \(\gamma_T\) is treated in Supplement Proposition~\ref{prop:supp-ls-contrast-estimator}; this avoids imposing a global no-spurious-critical-point condition on the derivative moment. \emph{Inverse-intensity weights} \(H=Z\Lambda^{-1}\), for compact-window features \(Z\), give temporal Stoyan--Grabarnik identities \citep{stoyangrabarnik1991,cronievanlieshout2018,kresinschoenberg2023}, matching \(\sum_i\int Z_{\cdot i}\lambda_i^{-1}\dd N_i\) to \(\int\sum_iZ_{\cdot i}\,\dd t\).  \emph{Direct weights} \(H=Z\) give Campbell--Mecke/Tak\'acs--Fiksel equations that match event sums to compensator integrals \citep{takacs1986,coeurjolly2016optimal,lastpenrose2017}.  Each moment family is admissible under the standing assumptions when its weight satisfies Assumption~\ref{ass:weights} and the corresponding identification, rank, and covariance conditions hold; the envelopes are automatic for the score and derivative weights, where they follow from \ref{ass:parameter}--\ref{ass:kernels}, and otherwise reduce to a polynomial-envelope condition on the feature \(Z\).
\end{example}

The four families above are illustrative rather than exhaustive. The admissible-weight formulation also accommodates smooth contrast criteria through their first-order equations.

\begin{remark}[Separable Bregman contrasts]\label{rem:bregman-contrasts}
The score and least-squares weights are also special cases of a separable
Bregman-contrast construction \citep{bregman1967relaxation,gneiting2007strictly}.  Let
\(b:(0,\infty)\to\R\) be \(C^3\) and strictly convex, set
\(g_b(x)=xb'(x)-b(x)\), and define
\[
\Gamma_{b,T}(\theta)=\sum_{i=1}^D\left\{T^{-1}\int_0^T g_b(\lambda_i(t;\theta))\,\dd t-T^{-1}\int_0^T b'(\lambda_i(t;\theta))\,\dd N_i(t)\right\}.
\]
Writing \(\Phi_b(t;\theta)=\diag\{b''(\lambda_1(t;\theta)),\ldots,b''(\lambda_D(t;\theta))\}\) and \(H_b(t;\theta)=D_\theta(t;\theta)^\top\Phi_b(t;\theta)\), suppose that \(H_b\) is admissible in the sense of Assumption~\ref{ass:weights}. Then
\[
\nabla_\theta\Gamma_{b,T}(\theta)=-T^{-1}\Psi_T^{H_b}(\theta).
\]
Thus the first-order condition for \(\Gamma_{b,T}\) is exactly the compensator estimating equation with weight \(H_b\). The choices \(b(x)=x^2\) and \(b(x)=x\log x-x\) make \(\Gamma_{b,T}\), respectively, the least-squares contrast \eqref{eq:ls-contrast} and exactly \(-T^{-1}\ell_T(\theta)\). In the latter case, \(H_b=H_{\mathrm{score}}\), so \(\Psi_T^{H_b}(\theta)=\nabla_\theta \ell_T(\theta)\).
\end{remark}

We now proceed to define various functionals necessary for the analysis of the asymptotic geometry of our setting.
\begin{equation}\label{eq:AH-def}
A_H(\theta)=\E_\theta\{H(0;\theta)D_\theta(0;\theta)\},
\end{equation}
\begin{equation}\label{eq:OmegaH-def}
\Omega_H(\theta)=
\E_\theta\{H(0;\theta)\Lambda(0;\theta)H(0;\theta)^\top\},
\end{equation}
their plug-in version
\begin{equation}\label{eq:OmegaHhat-def}
\hat\Omega_{H,T}(\theta)
=
\frac1T\int_0^T
H(t;\theta)\Lambda(t;\theta)H(t;\theta)^\top\,\dd t ,
\end{equation}
and, for local expansions, the two-parameter drift Jacobian
\begin{equation}\label{eq:BH-two-parameter}
B_H(\theta,\theta^\star)
=-\partial_\theta g_H(\theta,\theta^\star),
\qquad
B_H(\theta,\theta)=A_H(\theta).
\end{equation}
Here \(A_H\) records how the moments respond, on average, to local parameter changes, and \(\Omega_H\) records their intrinsic martingale variability per unit time.  Both \(A_H\) and \(\Omega_H\) are time-free by stationarity, \(\hat\Omega_{H,T}\) is the natural plug-in, and \(B_H\) controls expansions of the drift away from the diagonal.

\subsection{Standing assumptions}\label{sec:assumptions}
All results below are stated under the following block of standing assumptions. Parts (a)-(d) concern the model; parts (e)-(f) are relative to a given weight \(H\).  Section~\ref{sec:scope} discusses what each part means and how it is verified for specific examples. The GMM results of Section~\ref{sec:results} run on (a)-(c) and (e)-(f); part (d) enters through the likelihood benchmark, the least-squares contrast result in the supplement, and the efficiency comparison.

\begin{assumption}[Standing assumptions]\label{ass:standing}
Throughout the remainder of the paper, the dimensions \(p\), \(D\), and \(q\) are fixed and do not depend on \(T\), and we assume
\begin{enumerate}[label={\rm(\alph*)},ref=\theassumption(\alph*)]
    \item \label{ass:parameter} \emph{(Smooth positive links.)} The parameter set \(\Theta\subset\R^p\) is compact, with \(\Theta_0\Subset\Theta^\circ\), and is contained in a compact convex set \(K_\Theta\Subset\Theta^+\), where \(\Theta^+\) is the natural parameter domain.  For each \(i\), let \(\calX_i(K_\Theta)\) denote the set of linear-predictor values in \eqref{eq:model} attainable with \(\theta\in K_\Theta\).  The maps \(\theta\mapsto\nu_{\theta,i}\) and \((\theta,x)\mapsto\phi_{\theta,i}(x)\) are \(C^3\) on neighborhoods of \(K_\Theta\) and \(K_\Theta\times\calX_i(K_\Theta)\), respectively. Uniformly in \(i\), \(\phi_{\theta,i}\) is bounded away from zero on this set, \(\partial_x\phi_{\theta,i}\) is bounded, and all mixed \((\theta,x)\)-derivatives of \(\phi_{\theta,i}\) up to order three have at most polynomial growth in \(x\).  The derivatives of \(\nu_{\theta,i}\) up to order three are uniformly bounded on \(K_\Theta\).  Structural-zero coordinates are removed before applying the theory.
    
    \item \label{ass:kernels} \emph{(Smooth compact kernels.)} Each $\tilde h_{\theta,ij}$ and its active parameter derivatives up to order three are uniformly bounded and continuous on $[0,A]\times K_\Theta$, with second derivatives uniformly Lipschitz in $\theta$. Explicitly, $A$ is assumed to be fixed and known.
    
    \item \label{ass:stability} \emph{(Uniform absolute-kernel stability.)} With $L_{\theta,i}$ a Lipschitz constant for the $i^\text{th}$ link and $G^{\mathrm{abs}}(\theta)_{ij}=L_{\theta,i}\int_0^A|h_{\theta,ij}(u)|\,\dd u$, there is $\varepsilon>0$ such that $\sup_{\theta\in\Theta}\rho\{G^{\mathrm{abs}}(\theta)\}\le1-\varepsilon$, where $\rho(\cdot)$ denotes the spectral radius.
    
    \item \label{ass:identifiability} \emph{(Identifiability and Fisher information.)} For all \(\theta,\theta'\in\Theta\), if $\lambda_i(t;\theta)=\lambda_i(t;\theta')$ for every \(i\) and Lebesgue-a.e.\ \(t\), \(\Prob_\theta\)-a.s., then \(\theta=\theta'\)---equivalently, under \ref{ass:stability}, \(\theta\mapsto\Prob_\theta\) is injective. The per-time Fisher information
    \begin{equation}\label{eq:fisher}
        I(\theta)= \E_\theta\left[ \sum_{i=1}^D \frac{\nabla_\theta\lambda_i(0;\theta) \nabla_\theta\lambda_i(0;\theta)^\top}{\lambda_i(0;\theta)} \right]
    \end{equation}
    is well defined and uniformly nonsingular on $\Theta_0$.
    
    \item \label{ass:weights} \emph{(Admissible compact-window weights.)} \(H\) is shift-covariant and compact-window: for \(r=0,1,2\) there are measurable maps \(\mathsf H_\theta^{(r)}\) on local configurations on \([-A,0)\), continuous in \(\theta\), with
    \begin{equation}\label{eq:H-shift-covariant}
        \nabla_\theta^rH(t;\theta)=\mathsf H_\theta^{(r)}(G_tN),
        \qquad r=0,1,2 ,
    \end{equation}
    satisfying the polynomial local-count and Lipschitz envelopes
    \begin{equation}\label{eq:H-envelope}
        \|\nabla_\theta^rH(t;\theta)\|_{\op}
        \leq B_r\{1+\mathfrak C_t^{m_r}\},
        \qquad r=0,1,2,
    \end{equation}
    \begin{equation}\label{eq:H-lipschitz-envelope}
        \|\nabla_\theta^rH(t;\theta)-\nabla_\theta^rH(t;\theta')\|_{\op}
        \le B'_r\{1+\mathfrak C_t^{m'_r}\}\|\theta-\theta'\|_2,
        \qquad r=0,1,2 .
    \end{equation}

    \item \label{ass:moment-identification} \emph{(Moment identification and rank.)} For every \(\theta^\star\in\Theta_0\), \(g_H(\theta,\theta^\star)=0\) if and only if \(\theta=\theta^\star\); moreover, for every \(\epsilon>0\), 
    \begin{equation}\label{eq:H-global-separation}
        \inf_{\theta^\star\in\Theta_0}
        \inf_{\|\theta-\theta^\star\|_2\ge\epsilon}
        \|g_H(\theta,\theta^\star)\|_2>0,
    \end{equation}
    \begin{equation}\label{eq:H-rank}
        \inf_{\theta\in\Theta_0}
        \lambda_{\min}\{A_H(\theta)^\top A_H(\theta)\}>0 ,
    \end{equation}
    and \(\inf_{\theta\in\Theta_0}\lambda_{\min}\{\Omega_H(\theta)\}=\omega_0>0\).
\end{enumerate}
\end{assumption}

Throughout Section~\ref{sec:results}, the constants \(C,K\) may be taken to depend only on the constants named in (a)-(f), the dimensions \(D,p,q\), the memory length \(A\), and the exponent \(c\). Of primary importance, they carry no hidden dependence on \(\theta^\star\in\Theta_0\).

\subsection{Rates, limits, and efficiency}\label{sec:results}

We now state the main results of the paper. These results are the compact-window martingale counterparts of finite-dimensional GMM consistency, sandwich normality, and two-step optimal weighting \citep{hansen1982large,pakespollard1989,vandervaart1998}.

\begin{theorem}[GMM high-probability rate]\label{thm:generic-rate}
Let \(\hat\theta_H\) be a measurable global minimizer of \eqref{eq:gmm-estimator} with fixed deterministic positive-definite \(W\).  Under Assumptions~\ref{ass:parameter}--\ref{ass:stability} and~\ref{ass:weights}--\ref{ass:moment-identification}, for every \(c>0\) there are \(C,K<\infty\) such that, for all sufficiently large \(T\),
\begin{equation}\label{eq:generic-rate}
\sup_{\theta^\star\in\Theta_0}\Prob_{\theta^\star}\left(\norm{\hat\theta_H-\theta^\star}_2>K\sqrt{\frac{\log T}{T}}\right)\le CT^{-c} .
\end{equation}
\end{theorem}

\begin{theorem}[GMM asymptotic normality]\label{thm:generic-an}
Fix \(\theta^\star\in\Theta_0\).  Under the assumptions of Theorem~\ref{thm:generic-rate},
\begin{equation}\label{eq:generic-sandwich}
\sqrt T(\hat\theta_H-\theta^\star)
\Rightarrow
N_p\{0,V_H(W;\theta^\star)\},
\end{equation}
with
\begin{equation*}
V_H(W;\theta^\star)
=
(A_H^\top W A_H)^{-1}A_H^\top W\Omega_HW A_H(A_H^\top W A_H)^{-1},
\end{equation*}
which for \(W=\Omega_H(\theta^\star)^{-1}\) reduces to
\begin{equation*}
V_H^{\rm opt}(\theta^\star)
=
\{A_H(\theta^\star)^\top\Omega_H(\theta^\star)^{-1}A_H(\theta^\star)\}^{-1}.
\end{equation*}
\end{theorem}

Efficient weighting balances sensitivity against variability, exactly as in classical two-step GMM \citep{hansen1982large}; the Gaussian limit is a martingale central limit theorem for stochastic integrals \citep{rebolledo1980,jacodshiryaev2003}.

\begin{corollary}[Feasible optimal GMM]\label{cor:two-step-gmm}
Let \(\tilde\theta_H\) satisfy \eqref{eq:generic-rate}, set \(\hat W_T=\hat\Omega_{H,T}(\tilde\theta_H)^{-1}\), as per equation \eqref{eq:OmegaHhat-def}, on the event that this inverse exists and \(\hat W_T=I_q\) otherwise, and let \(\hat\theta_H^{(2)}\) be a measurable minimizer of \(\vartheta\mapsto m_T(\vartheta)^\top\hat W_Tm_T(\vartheta)\) with \(m_T(\vartheta)=T^{-1}\Psi_T^H(\vartheta)\).  Under the assumptions of Theorem~\ref{thm:generic-rate}, for every \(c>0\) there are \(C,K<\infty\) such that, for all sufficiently large \(T\),
\begin{equation}\label{eq:pluginHPrate}
   \sup_{\theta^\star\in\Theta_0}
\Prob_{\theta^\star}\left(
\norm{\hat\theta_H^{(2)}-\theta^\star}_2
>
K\sqrt{\frac{\log(T)}{T}}
\right)
\le CT^{-c} ,
\end{equation}
and, for each fixed \(\theta^\star\in\Theta_0\),
\begin{equation}\label{eq:plugin-CLT}
\sqrt T\left(\hat\theta_H^{(2)}-\theta^\star\right)
\Rightarrow
N_p\!\left(0,
\{A_H(\theta^\star)^\top\Omega_H(\theta^\star)^{-1}A_H(\theta^\star)\}^{-1}
\right).
\end{equation}
\end{corollary}

\begin{corollary}[Almost-sure logarithmic rates]\label{cor:almost-sure-rate}
Fix \(\theta^\star\in\Theta_0\) and write \(\hat\theta_{H,T}\), \(\hat\theta_{H,T}^{(2)}\) when the horizon must be explicit.  Under the assumptions of Theorem~\ref{thm:generic-rate},
\begin{equation}\label{eq:as-rate-gmm}
\Prob_{\theta^\star}\left(
\limsup_{T\to\infty}
\sqrt{\frac{T}{\log(T)}}
\norm{\hat\theta_{H,T}-\theta^\star}_2<\infty
\right)=1 ,
\end{equation}
and, when the first step is taken to be \(\tilde\theta_H=\hat\theta_{H,T}\), under the assumptions of Corollary~\ref{cor:two-step-gmm} likewise
\begin{equation}\label{eq:as-rate-two-step}
\Prob_{\theta^\star}\left(
\limsup_{T\to\infty}
\sqrt{\frac{T}{\log(T)}}
\norm{\hat\theta_{H,T}^{(2)}-\theta^\star}_2<\infty
\right)=1 .
\end{equation}
\end{corollary}

The efficiency geometry of the class is a projection statement.  It is the classical GMM and optimal-estimating-function inequality \citep{godambe1960,godambeheyde1987,hansen1982large,chamberlain1987asymptotic,newey1990efficient,carrasco2014asymptotic}, written in the predictable-martingale inner product
\[
    \langle f,g\rangle_\theta
    =\E_\theta\{f(0)\Lambda(0;\theta)g(0)^\top\}
\]
that replaces the ordinary \(L^2(P_\theta)\) pairing.  For the score weight, \(A_{H_{\rm score}}(\theta)=\Omega_{H_{\rm score}}(\theta)=I(\theta)\): the score is the compensator moment whose sensitivity and covariance coincide with Fisher information.

\begin{proposition}[Godambe information bound]\label{prop:godambe-optimality}
Fix \(\theta\in\Theta_0\).  Let \(H\) be admissible in the sense of Assumption~\ref{ass:weights} with \(\Omega_H(\theta)\) nonsingular, and set \(J_H(\theta)=A_H(\theta)^\top\Omega_H(\theta)^{-1}A_H(\theta)\).  Then
\begin{equation}\label{eq:godambe-cov-bound}
    J_H(\theta) \preceq I(\theta),
    \qquad\text{hence}\qquad
    V_H^{\rm opt}(\theta)=J_H(\theta)^{-1}\succeq I(\theta)^{-1}
\end{equation}
whenever both sides are nonsingular.  Equality holds if and only if each row of \(H_{\rm score}\) lies in the closed linear span, under \(\langle\cdot,\cdot\rangle_\theta\), of the rows generated by \(H\).
\end{proposition}

\begin{proof}
Let \(S=H_{\rm score}\).  The Gram matrix of the row families of \(H\) and \(S\) under \(\langle\cdot,\cdot\rangle_\theta\) is positive semidefinite:
\[
    \begin{pmatrix}
    \Omega_H(\theta) & A_H(\theta)\\
    A_H(\theta)^\top & I(\theta)
    \end{pmatrix}
    \succeq0 .
\]
Taking the Schur complement of \(\Omega_H(\theta)\) gives \eqref{eq:godambe-cov-bound}.  The Schur complement is the Gram matrix of the score residual after projection onto the closed span of the rows of \(H\); it vanishes exactly when that residual is zero.  The covariance comparison follows by inversion.
\end{proof}

If \(\Omega_H(\theta)\) is singular, the same statement holds on the quotient by zero-variance moment directions, or with the Moore--Penrose inverse after removing redundant rows.  The bound holds for every admissible library \(H\), hence pointwise across the entire admissible class; what is not claimed is an attainability or local-asymptotic-minimax statement over all inference procedures.  Writing \(\mathcal V_H(\theta)\) for the closed row span generated by \(H\) and \(P_H\) for the orthogonal projection onto it, the Schur-complement residual is exactly the score energy the library cannot see:
\begin{equation*}
    I(\theta)-J_H(\theta)
    =
    \left(
    \left\langle S_a-P_HS_a,\,
    S_b-P_HS_b\right\rangle_\theta
    \right)_{a,b=1}^p
    \succeq0 .
\end{equation*}
A nested sequence of finite admissible libraries can improve efficiency only by shrinking this residual, and likelihood efficiency is reached precisely when it vanishes; Figure~\ref{fig:godambe-projection} illustrates, for a single score row, how the projection climbs and the residual contracts as the library grows.  Section~\ref{sec:simulation} evaluates exactly this gap numerically, for a just-identified and an overidentified derivative library, against the Fisher benchmark.  The theory here is pointwise in the chosen finite library; letting \(q=q_T\to\infty\) would require many-moment concentration and regularized weighting, as for efficient instruments in i.i.d. GMM \citep{newey1990efficient,donald2009choosing,newey2009generalized}.

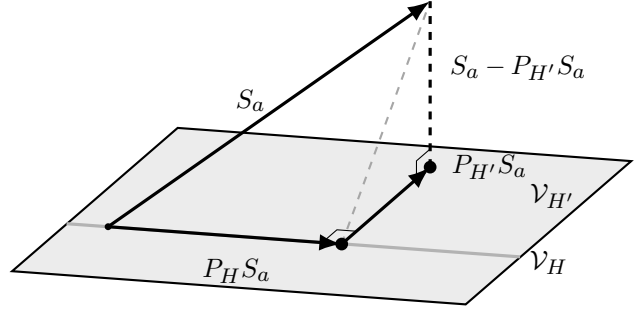
\begin{figure}[!ht]
\centering
\begin{tikzpicture}[
    x=1cm,y=1cm,
    >=Latex,
    every node/.style={font=\small},
    plane/.style={draw, thick, fill=gray!15},
    span/.style={gray!60, very thick},
    score/.style={->, very thick},
    component/.style={->, very thick},
    residualold/.style={thick, dashed, gray!70},
    residualnew/.style={very thick, dashed},
    marker/.style={thin}
]
\coordinate (A) at (0.70,0.85);
\coordinate (B) at (6.70,0.41);
\coordinate (C) at (8.89,2.31);
\coordinate (D) at (2.89,2.75);
\filldraw[plane] (A)--(B)--(C)--(D)--cycle;
\coordinate (L1) at (1.42,1.48);
\coordinate (L2) at (7.42,1.04);
\draw[span] (L1) -- (L2);
\coordinate (O)  at (1.97,1.44);   
\coordinate (P)  at (5.06,1.21);   
\coordinate (Pp) at (6.23,2.23);   
\coordinate (S)  at (6.23,4.43);   
\draw[score]
    (O) -- node[pos=.56, left=5pt] {$S_a$} (S);
\draw[component]
    (O) -- node[pos=.55, below=5pt] {$P_HS_a$} (P);
\draw[component]
    (P) -- (Pp);
\draw[residualold] (P) -- (S);
\draw[residualnew]
    (Pp) -- node[pos=.60, right=4pt] {$S_a-P_{H'}S_a$} (S);
\coordinate (M1) at (4.801,1.229);
\coordinate (M2) at (4.997,1.399);
\coordinate (M3) at (5.256,1.380);
\draw[marker] (M1) -- (M2) -- (M3);
\coordinate (N1) at (6.049,2.073);
\coordinate (N2) at (6.049,2.313);
\coordinate (N3) at (6.230,2.470);
\draw[marker] (N1) -- (N2) -- (N3);
\node at (7.80,0.97) {$\mathcal V_H$};
\node at (7.85,1.85) {$\mathcal V_{H'}$};
\fill (O) circle (1.3pt);
\fill (P) circle (2.4pt);
\fill (Pp) circle (2.4pt);
\node[right=5pt of Pp] {$P_{H'}S_a$};
\end{tikzpicture}
\caption{Projection geometry for the Godambe bound, for a single score row \(S_a\).  A library \(H\) generates the predictable row span \(\mathcal V_H\) (line); enlarging it to \(H'\supseteq H\) enlarges the span to \(\mathcal V_{H'}\) (plane).  The projection climbs from \(P_HS_a\) to \(P_{H'}S_a\), recovering the in-span component \(P_{H'}S_a-P_HS_a\), and the residual shrinks from the grey to the black dashed segment: in the norm induced by \(\langle\cdot,\cdot\rangle_\theta\), \(\|S_a-P_HS_a\|^2=\|S_a-P_{H'}S_a\|^2+\|P_{H'}S_a-P_HS_a\|^2\).  Likelihood efficiency is reached when the span contains the full score span.  Section~\ref{sec:simulation} measures exactly this effect, with the just-identified library \(H_J\) enlarged to \(H_O\supseteq H_J\).}
\label{fig:godambe-projection}
\end{figure}

\begin{remark}[Simulation-based reading]\label{rem:sbi-gmm}
The residual summary \[Z_T^H(\vartheta;N)=T^{-1/2}\Psi_T^H(\vartheta;N)\] is the natural compact summary statistic for simulation-based inference on this model class \citep{mcfadden1989,pakespollard1989,gourierouxmonfortrenault1993,gallanttauchen1996}.  If paths are simulated from \(\Prob_\vartheta\) and the law of \(Z_T^H(\vartheta;N)\) is approximated by its martingale Gaussian limit, the resulting synthetic-likelihood quadratic \citep{price2018bayesian} is, up to Monte Carlo error and log-determinant terms, the oracle optimally weighted GMM criterion \citep{cranmer2020}.  Compact-window GMM is in this sense the analytic version of simulation-based inference built from compensated residual summaries.
\end{remark}

Specializing to \(H_{\rm score}\) makes the estimating equation the likelihood score equation, so the generic theory covers local score roots; for the score, the envelopes of Assumption~\ref{ass:weights} follow from \ref{ass:parameter}--\ref{ass:kernels}, and identification follows from \ref{ass:identifiability} through Kullback--Leibler separation, so the likelihood results below assume only (a)--(d).  The one likelihood-specific step is global localization: a global maximizer of \(\ell_T\) must be shown to lie, with high probability, in the local ball where the score expansion
\[
\sqrt T(\hat\theta_T-\theta^\star)
=
I(\theta^\star)^{-1}T^{-1/2}\nabla\ell_T(\theta^\star)
+o_{\Prob_{\theta^\star}}(1)
\]
applies.  This is done through a uniform likelihood law and population Kullback--Leibler separation in Supplement Section~\ref{app:mle-upper}.

\begin{theorem}[MLE high-probability rate]\label{thm:main-rate}
Let \(\hat\theta_T\) be any measurable global maximizer of \eqref{eq:likelihood}.  Under Assumptions~\ref{ass:parameter}--\ref{ass:identifiability}, for every \(c>0\) there are \(C,K<\infty\) such that, for all sufficiently large \(T\),
\begin{equation}\label{eq:main-rate}
\sup_{\theta^\star\in\Theta_0}
\Prob_{\theta^\star}\left(
\norm{\hat\theta_T-\theta^\star}_2
>
K\sqrt{\frac{\log T}{T}}
\right)
\le CT^{-c} .
\end{equation}
\end{theorem}

\begin{corollary}[MLE asymptotic normality]\label{cor:mle-asymptotic-normality}
Fix \(\theta^\star\in\Theta_0\).  Under Assumptions~\ref{ass:parameter}--\ref{ass:identifiability},
\begin{equation*}
\sqrt T\left(\hat\theta_T-\theta^\star\right)
\Rightarrow
N_p\left\{0,I(\theta^\star)^{-1}\right\} .
\end{equation*}
\end{corollary}

Finally, using Le Cam's two point testing reduction \citep[e.g.][Chapter 15]{wainwright2019} we are able to demonstrate that the $T^{-1/2}$ scale cannot be uniformly improved; as one may expect from classical i.i.d.\ theory.

\begin{theorem}[Root-\(T\) scale lower bound]\label{thm:lower-bound}
Suppose Assumptions~\ref{ass:parameter}--\ref{ass:identifiability} hold and \(\Theta_0\) contains a Euclidean ball \(B(\theta_0,r)\subset\Theta^\circ\).  Then there exist \(a_0>0\), \(\epsilon_0>0\), and \(T_0<\infty\) such that, for every \(0<a\le a_0\) and every \(T\ge T_0\),
\begin{equation}\label{eq:lower-bound}
\inf_{\tilde\theta_T}
\sup_{\theta\in\Theta_0}
\Prob_\theta\left(
\norm{\tilde\theta_T-\theta}_2\ge\frac{a}{\sqrt T}
\right)
\ge\epsilon_0 ,
\end{equation}
where the infimum is over all estimators measurable with respect to \(\sigma\{N|_{[-A,T]}\}\).
\end{theorem}

Together, Theorems~\ref{thm:generic-rate}, \ref{thm:main-rate}, and~\ref{thm:lower-bound} and Corollaries~\ref{cor:almost-sure-rate} and~\ref{cor:mle-asymptotic-normality} identify the statistical scale: GMM estimators and the MLE attain a uniform high-probability \(\mathcal O(\sqrt{\log(T)/T})\) upper rate, while the \(T^{-1/2}\) scale cannot be improved uniformly.  

\section{Assumptions, examples, and scope}\label{sec:scope}
The assumptions are intentionally stronger than the weakest possible: they isolate the dependence and efficiency questions without boundary, sparsity, or high-dimensional pathologies, which require different tools.  This section explains what each part asks of a model and of a moment library, verifies the conditions on two kernel families and on a genuinely non-score library, and records that the theory requires neither sign restrictions on the kernels nor stationary initialization.

Assumption~\ref{ass:standing} serves three purposes. Parts \ref{ass:parameter} and \ref{ass:kernels} are regularity conditions. The link \(\phi_{\theta,i}\) acts on the inner linear predictor, and~\ref{ass:parameter} asks for smoothness and positivity only on the predictor values the model can actually reach (its input domain).  With excitation alone these values never fall below the baseline, so a positive set suffices and the identity link \(\phi(x)=x\) is admissible.  Inhibition removes that floor: negative kernels can drive the predictor below zero and, as recent events accumulate, across an unbounded part of \(\R\).  A signed model whose predictor ranges this widely therefore needs a link that stays positive and smooth on all of \(\R\), while a positive linear subfamily is already covered on its positive domain.

The stability assumption \ref{ass:stability} is the same subcriticality assumption used by Br\'emaud--Massouli\'e for their Poisson-embedding theory \citep{bremau1996}, and its absolute values are not cosmetic: a disagreement coupling must dominate propagation without relying on cancellation, so it is the absolute kernels, not the signed ones, that must be uniformly subcritical.  Under \ref{ass:parameter}--\ref{ass:stability} the embedding admits a stationary solution for every \(\theta\in\Theta\), stochastically dominated, under a common Poisson embedding, by the linear Hawkes process with baselines \(\phi_{\theta,i}(\nu_{\theta,i})\) and nonnegative kernels \(L_{\theta,i}|h_{\theta,ij}|\).  That dominating process is subcritical uniformly in \(\theta\), so its fixed-window counts have uniformly bounded exponential moments \citep{leblanc2024}; the domination passes the bound to \(N\), giving, for every fixed \(L<\infty\) and some \(\xi_L,C_L>0\),
\[
    \sup_{\theta\in\Theta}\sup_{s\in\R}\E_\theta\bigl[\exp\{\xi_L N([s,s+L])\}\bigr]\le C_L .
\]
This uniform control of fixed-window counts (Supplement Lemma~\ref{lem:fixed-window-count-mgf}) is the moment input the concentration and envelope arguments rely on throughout.

The remaining conditions concern identification, at two levels: the model, and the finite moment library an estimator actually uses.  Part~\ref{ass:identifiability} is the model-level condition. It ensures that distinct parameters produce distinct intensity paths, and the Fisher information \(I(\theta)\) must be nonsingular. When this condition is met, \(\theta\) is recoverable.  The other two parts govern an estimator that replaces the whole path with a handful of summaries.  Part~\ref{ass:weights} keeps that library well-behaved (finitely many compact-window, shift-covariant features with polynomial count envelopes) so the limit functionals \(A_H\) and \(\Omega_H\) exist and are time-free.  Part~\ref{ass:moment-identification} then asks that the summaries still determine \(\theta\): the population moment \(g_H(\theta,\theta^\star)\) vanishes only at the truth, with rank and covariance margins.  This does not follow from~\ref{ass:identifiability}. A model identified from its full path can be underdetermined by a small library, exactly as a well-identified GMM model can still carry irrelevant or rank-deficient instruments \citep{hansen1982large,newey1990efficient}.  

The derivative library \(H=D_\theta^\top\) of Example~\ref{ex:weight-families-body} is instructive. Its drift is a gradient,
\[
    g_{D_\theta^\top}(\theta,\theta^\star)
    =-\tfrac12\nabla_\theta\,
    \E_{\theta^\star}\bigl\|\lambda(0;\theta)-\lambda(0;\theta^\star)\bigr\|_2^2 ,
\]
so zeros of the drift are critical points of the population least-squares contrast. Model identifiability implies that this contrast has the unique global minimizer \(\theta^\star\), but it does not by itself exclude other critical points of the derivative drift. Thus the derivative-moment GMM theorem applies to \(H=D_\theta^\top\) once \eqref{eq:H-global-separation} is verified, while the globally minimized least-squares contrast is covered separately in Supplement Proposition~\ref{prop:supp-ls-contrast-estimator}. The local rank and covariance margins for the derivative expansion are automatic, since pointwise \(D_\theta^\top D_\theta\succeq\underline\lambda\,D_\theta^\top\Lambda^{-1}D_\theta\) and \(D_\theta^\top\Lambda D_\theta\succeq\underline\lambda\,D_\theta^\top D_\theta\), so that \(A_{D_\theta^\top}(\theta)\succeq\underline\lambda\,I(\theta)\) and \(\Omega_{D_\theta^\top}(\theta)\succeq\underline\lambda\,A_{D_\theta^\top}(\theta)\) are uniformly nonsingular on \(\Theta_0\).

Generally, the pattern for verifying a GMM library is to compute the drift \(g_H(\theta,\theta^\star)\), show its zero set is \(\{\theta^\star\}\) with a separation margin, check that \(A_H(\theta)\) has full column rank on the active set, and remove redundant rows so that \(\Omega_H(\theta)\) is nonsingular. For least squares, the global step is instead contrast separation, supplied by model identifiability; the derivative moment is used locally through the first-order condition. Worked verifications for the truncated-exponential, smooth-rectifier, and one-point-age libraries are collected in Supplement Section~\ref{app:trunc-exp-info}.

\subsection{Examples}\label{sec:examples}

Examples~\ref{ex:trunc-exp} and~\ref{ex:nl-smooth-relu-trunc-exp} verify the model parts \ref{ass:parameter}--\ref{ass:identifiability} on a linear and a genuinely nonlinear kernel family. Example~\ref{ex:one-point-age-body} verifies the weight parts \ref{ass:weights}--\ref{ass:moment-identification} for a library containing no score ingredient.  We note that while the assumptions permit signed kernels, verification of global identifiability is model-specific.

\begin{example}[Truncated exponential kernels]\label{ex:trunc-exp}
Consider a linear Hawkes submodel with positive baselines and positive active kernels, bounded away from zero on their active coordinate intervals,
\[
    h_{\theta,ij}(u)
    =
    \alpha_{ij}c_{ij}(\beta_{ij})e^{-\beta_{ij}u}\1\{0\le u\le A\},
    \qquad
    \beta_{ij}\in[\underline\beta_{ij},\overline\beta_{ij}]\Subset(0,\infty),
\]
with \(c_{ij}\in C^3\) strictly positive on the compact \(\beta\)-interval, covering both unnormalized and normalized truncated exponentials.  On compact positive active-amplitude and baseline regions satisfying the spectral-radius restriction, parts \ref{ass:parameter}--\ref{ass:stability} follow from smoothness and the usual linear Hawkes subcriticality \citep{bremau1996}.  Identification is concrete: a source event entering the memory window produces the endpoint signature \(\alpha_{ij}c_{ij}(\beta_{ij})\), and its exit at age \(A\) removes \(\alpha_{ij}c_{ij}(\beta_{ij})e^{-\beta_{ij}A}\); for positive active amplitudes the ratio of the two signatures determines \(\beta_{ij}\) and the entry signature then determines \(\alpha_{ij}\).  Detailed identifiability and Fisher-information checks are in Supplement Section~\ref{app:trunc-exp-info}.
\end{example}

\begin{example}[Smooth-ReLU nonlinear Hawkes with truncated-exponential filters]\label{ex:nl-smooth-relu-trunc-exp}
For a genuinely nonlinear example, take known smooth rectifier links
\[
    \Phi_i(x)
    =
    \epsilon_i+\frac{a_i}{b_i}\log\{1+\exp(b_i(x-c_i))\},
    \qquad
    \epsilon_i,a_i,b_i>0,\quad c_i\in\R,
\]
and \(\lambda_i(t;\theta)=\Phi_i\{X_i(t;\theta)\}\) with normalized truncated-exponential filters
\[
    X_i(t;\theta)=\nu_i+
    \sum_{j=1}^D\int_{[t-A,t)}
    \gamma_{ij}\frac{\beta_{ij}e^{-\beta_{ij}(t-s)}}{1-e^{-\beta_{ij}A}}\,\dd N_j(s),
\]
where \(\nu_i\) and the active \(\gamma_{ij},\beta_{ij}\) lie in compact interior intervals.  The smooth rectifier keeps the link positive and \(C^4\) while preserving the threshold/saturation shape; since the filters are normalized and \(\sup_x\Phi_i'(x)=a_i\), the stability matrix is simply \(G^{\mathrm{abs}}(\theta)_{ij}=a_i|\gamma_{ij}|\), and \ref{ass:stability} is the spectral-radius bound on this matrix.  For positive active amplitudes, \(X_i\ge\underline\nu_i\) and \(\Phi_i'\) is bounded away from zero there; because the links are known and strictly increasing, equality of intensities is equality of the latent inputs, and the endpoint argument above identifies amplitudes, decays, and baselines on the active set.  Signed amplitudes leave the dependence theory untouched, since it uses \(|\gamma_{ij}|\), but the identification check must still rule out active-set degeneracies; Supplement Section~\ref{app:trunc-exp-info} gives the formal verification in the positive-active case.
\end{example}

\begin{example}[A non-score identifying library]\label{ex:one-point-age-body}
The estimator class for which \ref{ass:moment-identification} holds is not vacuous. In the univariate truncated-exponential submodel with \(\theta=(\mu,\alpha,\beta)\), the three compact-window features
\[
    H(t)=\bigl(\1\{\mathfrak C_t=0\},\ \1\{\mathfrak C_t=1\},\ U_t\1\{\mathfrak C_t=1\}\bigr)^\top,
    \qquad
    U_t=\int_{[t-A,t)}(t-s)\,\dd N(s),
\]
form a genuine non-score Campbell--Mecke/Tak\'acs--Fiksel library; there are no inverse-intensity factors or intensity derivatives present.  The empty-window moment identifies the baseline, and the one-point age moments identify amplitude and decay.  The global-identification, rank, and covariance checks are in Supplement Section~\ref{app:one-point-age-gmm}.
\end{example}

\subsection{Inhibition and nonstationary starts}\label{sec:robustness}

Two modelling idealisations are ubiquitous in the current literature: the first is that interactions are excitatory, and the second is that the process is observed in equilibrium (i.e. that it is observed with a stationary start). The results in Section~\ref{sec:results} assume neither of these (often artificial) modelling conveniences. 

\begin{remark}[Signed and inhibitory compact-memory interactions]\label{rem:signed-inhibitory}
Under Assumptions~\ref{ass:parameter}--\ref{ass:stability}, with stability measured by the absolute-kernel matrix \(G^{\mathrm{abs}}\), the replacement and concentration theory of Section~\ref{sec:concentration} (Theorem~\ref{thm:local-concentration} together with the block-replacement and truncation lemmas of Supplement Section~\ref{app:poisson-embedding}) remains valid when the kernels \(h_{\theta,ij}\) have either sign, positivity of the intensities being maintained by the links.  This is a dependence statement: identifiability and Fisher nonsingularity remain model-specific and must be checked for the active subfamily.
\end{remark}

\begin{proposition}[Logarithmic burn-in from nonstationary initial histories]\label{prop:log-burn-in}
Suppose Assumptions~\ref{ass:parameter}--\ref{ass:stability} hold.  Let \(N^\calH\) be started from an initial history \(\calH\) at time \(0\) with
\[
    \sup_{\theta^\star\in\Theta_0}
    \E_{\theta^\star}^{\calH}\left[\sum_{j=1}^D\calH_j([-A,0])\right]<\infty .
\]
Let \(b_T=A+C_0\log T\).  For every \(M<\infty\), \(C_0\) can be chosen so that
\begin{equation*}
\sup_{\theta^\star\in\Theta_0}
 d_{\TV}\!\left(
\calL_{\theta^\star}^\calH(N^\calH|_{[b_T-A,T]}),
\calL_{\theta^\star}(N^{\rm stat}|_{[b_T-A,T]})
\right)
\le C T^{-M} .
\end{equation*}
Consequently every stationary high-probability bound for a compact-window statistic computed over an interval of length \(T-b_T\) transfers to the post-burn-in statistic with an additional \(O(T^{-M})\) error, and stationary fixed-parameter central limit theorems transfer as well: the conclusions of Theorems~\ref{thm:generic-rate}--\ref{thm:generic-an}, Corollaries~\ref{cor:two-step-gmm}--\ref{cor:almost-sure-rate}, and Theorem~\ref{thm:main-rate} are unchanged when the estimating equations and likelihood are computed over \([b_T,T]\), with normalization by \(T-b_T\).  For a fixed deterministic locally finite history the bounds are conditional on that history, with constants depending on \(1+\sum_j\calH_j([-A,0])\); uniformity over a class of deterministic starts requires a uniform bound on these local counts.
\end{proposition}

Stationary initialization is the standard idealization of applied Hawkes inference: consistency and asymptotic normality are established for the equilibrium process, as though the data had been generated by a system running since the infinite past, whereas a real record begins at a finite and generally out-of-equilibrium time.  The proposition shows the idealization is inessential: any start with finite expected local count is forgotten after a logarithmic burn-in, and the stationary rates and limit laws then govern the estimator computed from the finite, nonstationary record actually observed.  The transfer device itself is elementary: the variational characterization of total variation, applied after mapping the path to the statistic (Supplement Lemma~\ref{lem:tv-transfer}); the proof is in Supplement Section~\ref{app:poisson-embedding}.  Initial histories generated by running the dynamics from any finite-expected-count configuration at an arbitrary earlier time satisfy the displayed window-count condition automatically as demonstrated in Lemma \ref{lem:finite-expected-local-history-nonstationary} in the supplement.

\section{The probabilistic mechanism}\label{sec:concentration}

Compact support in \eqref{eq:model} makes \(\lambda(t;\theta)\) depend only on the window \([t-A,t)\). But crucially, this locality is pathwise, and does not extend to the entire stationary law because two faraway windows can contain points that share an ancestor.  For nonnegative linear models, this dependence is managed via the Poisson-cluster characterisation \citep{hawkesoakes1974,reynaudbouretroy2007,hansen2015}, with convergence-to-equilibrium rates going back to \citet{bremau1996} and \citet{bremaud2002rate}. In the complementary bounded-intensity, bounded-memory setting, \citet{bremau1996} proved exponentially fast convergence in total variation from arbitrary initial conditions.  Our contribution in this section is the extraction of the finite-window consequences that facilitate inference: a total-variation replacement bound for local windows with an explicit initial-window count, and concentration for sliding compact-window functionals with polynomial envelopes.  Compared with quoting a generic mixing coefficient, this formulation keeps track of exactly the window an estimating equation sees and feeds directly into the stochastic-integral brackets used in the proofs; it is close in spirit to the nonasymptotic arguments of \citet{hansen2015}, with the Poisson embedding replacing the cluster genealogy.

In the remainder of this section, we outline the proof strategy and probabilistic mechanism underpinning the main results presented in Section~\ref{sec:results}. The full proof splits into two halves. The first demonstrates a uniform analogue of Hansen et al.'s (2015) Proposition 3, which gives a Bernstein type inequality for functionals of the shifted counting process. Our theorem extends Hansen et al.'s so that the same result holds uniformly over a compact (finite-dimensional) parameter set, for linear \emph{and} non-linear Hawkes processes, with the latter (more general) case not being covered by their proposition.

\begin{theorem}[Local-window concentration]\label{thm:local-concentration}
Suppose Assumptions~\ref{ass:parameter}--\ref{ass:stability} hold.  Let \(\Xi\subset\R^d\) be compact with fixed \(d\), and let \(\{Z_\xi:\xi\in\Xi\}\) be real-valued functionals of local configurations on \([-A,0)\).  Assume that, for constants \(B<\infty\) and \(m\ge0\),
\begin{equation}\label{eq:local-envelope}
    \sup_{\xi\in\Xi}|Z_\xi(x)|\le B\{1+C(x)^m\},
\end{equation}
and
\begin{equation}\label{eq:local-lipschitz-envelope}
    |Z_\xi(x)-Z_{\xi'}(x)|\le B\{1+C(x)^m\}\norm{\xi-\xi'}_2 .
\end{equation}
Then, for every \(c>0\), there exist \(C,K<\infty\) and an integer \(q_0\ge1\) such that, for all sufficiently large \(T\),
\begin{equation*}
\sup_{\theta^\star\in\Theta_0}\Prob_{\theta^\star}\left(\sup_{\xi\in\Xi}\left|\frac1T\int_0^T\{Z_\xi\circ G_t-\E_{\theta^\star}Z_\xi\}\,\dd t\right|>K\frac{\log^{q_0}(T)}{\sqrt T}\right)\le CT^{-c} .
\end{equation*}
\end{theorem}

We note that Theorem \ref{thm:local-concentration} may be of independent interest since it is a general Bernstein-type inequality for a large class of functionals of the underlying non-linear multivariate Hawkes process. The proof of Theorem \ref{thm:local-concentration} mostly follows that of \cite{hansen2015}, which uses a standard blocking, coupling, and truncation argument to apply Bernstein's inequality for i.i.d.\ random variables. The key insight is to use the Poisson embedding techniques of \cite{bremau1996} to construct the coupling law for non-linear Hawkes processes when no cluster representation exists. In doing so, we prove that conditional on a finite pre-history in $[s-A,s)$, the law of the non-stationary process converges exponentially fast in total variation distance, uniformly over $\Theta$, to its stationary distribution; see the Supplement section \ref{app:poisson-embedding} for the formal statement.

The remainder of the proof for the main results follows a more standard $M$-estimation route. We prove that the empirical risk, and its Jacobian, concentrate uniformly around their expectation and then apply standard Taylor expansion arguments. In more detail, the  concentration arguments we obtain for our class of estimators, follows from repeated applications of Theorem \ref{thm:local-concentration} and a continuous time Freedman type inequality from \cite{dzhaparidze2001}. We are able to exploit such martingale concentration inequalities via the Doob-Meyer decomposition \citep[e.g.][Corollary 14.I.V]{daley2008} which ensures that the estimation equations admit a martingale and drift decomposition. Pointwise in $\theta$, the drift terms, by assumption, satisfy the necessary conditions of Theorem \ref{thm:local-concentration}. Moreover, pointwise in $\theta$, the martingale terms concentrate around $0$ as their predictable variations concentrate around their expectation (which is linear in $T$); also by Theorem \ref{thm:local-concentration}. Thus, we may apply the Freedman type inequality to obtain fluctuations around the mean that are $\mathcal O\left(\log^{q_0}(T)/\sqrt{T}\right)$ with high-probability. A standard stochastic equicontinuity, and $\epsilon$-net, argument ensures that the empirical risk and its Hessian uniformly concentrate around their expectations at these aforementioned radii. We then conclude the proof of Theorem~\ref{thm:generic-rate} via a standard pathwise Taylor expansion argument. Using the preceding techniques, the two step estimation results and central limit theorems easily follow from Taylor expansions and the martingale central limit theorem.

As one may expect, the almost sure rate results directly extend from the previous nonasymptotic results. Such rates are obtained from the same high probability bounds by applying the first Borel--Cantelli lemma on integer horizons and then extending from integer times to all positive real \(T\) by a pathwise interpolation lemma.  On the event \(\sup_{0\le t\le n+1}\mathfrak C_t\le L\log n\), the normalized estimating equations, their local Jacobians, and the plug-in covariance change by at most \(\mathcal O\left(\log^{q_0}(n)/n\right)\) as \(T\) ranges over \([n,n+1]\). Since this event holds eventually almost surely, the integer-horizon pathwise Taylor expansion arguments extend to all real horizons. A similar argument, after a suitable burn-in period and an application of Proposition \ref{prop:log-burn-in}, yields the same almost sure convergence rates under a non-stationary start.

\section{Numerical experiment}\label{sec:simulation}

We use a bivariate linear Hawkes experiment \citep{hawkes1971,hawkesoakes1974} as a transparent setting for the Godambe projection comparison with our aim being to demonstrate the efficiency geometry. This submodel lets maximum likelihood, least squares, and an optimally weighted two-step overidentified enlargement be compared against explicit population Godambe targets.  The data-generating intensity is
\[
    \lambda_i(t;\theta)=\mu_i+
    \sum_{j=1}^2\alpha_{ij}X_j(t;\beta),
    \qquad
    X_j(t;\beta)=\int_{[t-A,t)}k_\beta(t-s)\,\dd N_j(s),
\]
with
\[
    k_\beta(u)=
    \frac{\beta e^{-\beta u}}{1-e^{-\beta A}}\mathbf 1\{0\le u\le A\},
\]
parameter \(\theta=(\mu_1,\mu_2,\alpha_{11},\alpha_{12},\alpha_{21},\alpha_{22},\beta)\), and design
\[
    A=3,
    \qquad
    \theta_0=(0.22,0.18,0.34,0.10,0.24,0.30,1.25).
\]
The spectral radius of \(\alpha_0\) is approximately \(0.476\) and the stationary mean intensity approximately \((0.393,0.392)\), so \(T=1000\) corresponds to about \(784\) expected events and \(T=16000\) to about \(12552\).

We compare three estimators: the conditional MLE; the global continuous-time least-squares contrast minimizer \(\hat\theta_{\rm LS}\), reported as \(J\) because its local first-order equation uses the derivative weight \(H_J=D_\theta^\top\); and a feasible two-step overidentified GMM estimator with
\[
    H_O(t;\theta)=
    \begin{pmatrix}
        D_\theta(t;\theta)^\top\\
        D_\theta(t;\theta)^\top R_\tau(t;\theta)
    \end{pmatrix},
    \qquad
    R_\tau(t;\theta)=
    \diag\left\{
    \frac{\tau_i}{\tau_i+\lambda_i(t;\theta)}:i=1,2
    \right\},
\]
and \(\tau_1=\tau_2=0.4\), so that \(H_J\) has \(q=p=7\) moments and \(H_O\) has \(q=14\). The factor \(R_\tau\) is a bounded inverse-intensity-like rescaling, not the likelihood inverse \(\Lambda^{-1}\); we use it to enlarge the derivative row span toward the score direction suggested by Proposition~\ref{prop:godambe-optimality}, without claiming an algebraic non-containment theorem for the score span. This is overidentification, not model overspecification: the Hawkes model is unchanged and only the moment library grows. The reported \(J\) estimator is the global least-squares contrast minimizer covered by Supplement Proposition~\ref{prop:supp-ls-contrast-estimator}. Its local sandwich target is
\[
    V_J
    =
    A_{H_J}(\theta_0)^{-1}
    \Omega_{H_J}(\theta_0)
    A_{H_J}(\theta_0)^{-1}.
\]
For \(H_O\), the local rank condition \eqref{eq:H-rank} is automatic because \(A_{H_O}(\theta)\) stacks additional rows on \(A_{H_J}(\theta)\); we do not separately verify the global separation \eqref{eq:H-global-separation} or the covariance margin \(\Omega_{H_O}(\theta)\succ0\), so \(H_O\) is reported as a numerical efficiency benchmark for the population Godambe projections rather than as a theorem-verified estimator.

The Monte Carlo design uses \(T\in\{1000,2000,4000,8000,16000\}\) with \(2000\) replications per horizon; implementation details and additional diagnostics are in Supplement Sections~\ref{app:simulation-implementation} and~\ref{app:simulation-diagnostics}. Let \(V_M=I(\theta_0)^{-1}\), let
\[
    V_J=A_{H_J}(\theta_0)^{-1}\Omega_{H_J}(\theta_0)A_{H_J}(\theta_0)^{-1}
\]
be the local least-squares sandwich covariance, and let \(V_O\) be the optimally weighted overidentified Godambe covariance. Figure~\ref{fig:bivar-overall-efficiency}(a) reports, for \(m\in\{M,J,O\}\),
\[
    R_m(T)=
    \left[
    \frac1p\sum_{k=1}^p
    \left\{
    \frac{\sqrt T\operatorname{RMSE}(\hat\theta_{m,k})}
    {\sqrt{(V_M)_{kk}}}
    \right\}^2
    \right]^{1/2}.
\]
The MLE sits close to one, the Fisher benchmark; the least-squares contrast estimator settles near its larger derivative-weight sandwich target, while the overidentified estimator lies much closer to the MLE. The comparison reads as maximum likelihood versus least squares versus an optimally weighted enlargement, with the population Godambe targets pricing the least-squares efficiency loss exactly. This is the projection mechanism of Proposition~\ref{prop:godambe-optimality} at work (Figure~\ref{fig:godambe-projection}): enlarging \(\mathcal V_H(\theta_0)\) from the derivative span to the overidentified span enlarges the span toward the score and shrinks the residual \(S_a-P_HS_a\) coordinate by coordinate.

\begin{figure}[!ht]
\centering
\includegraphics[width=0.88\textwidth]{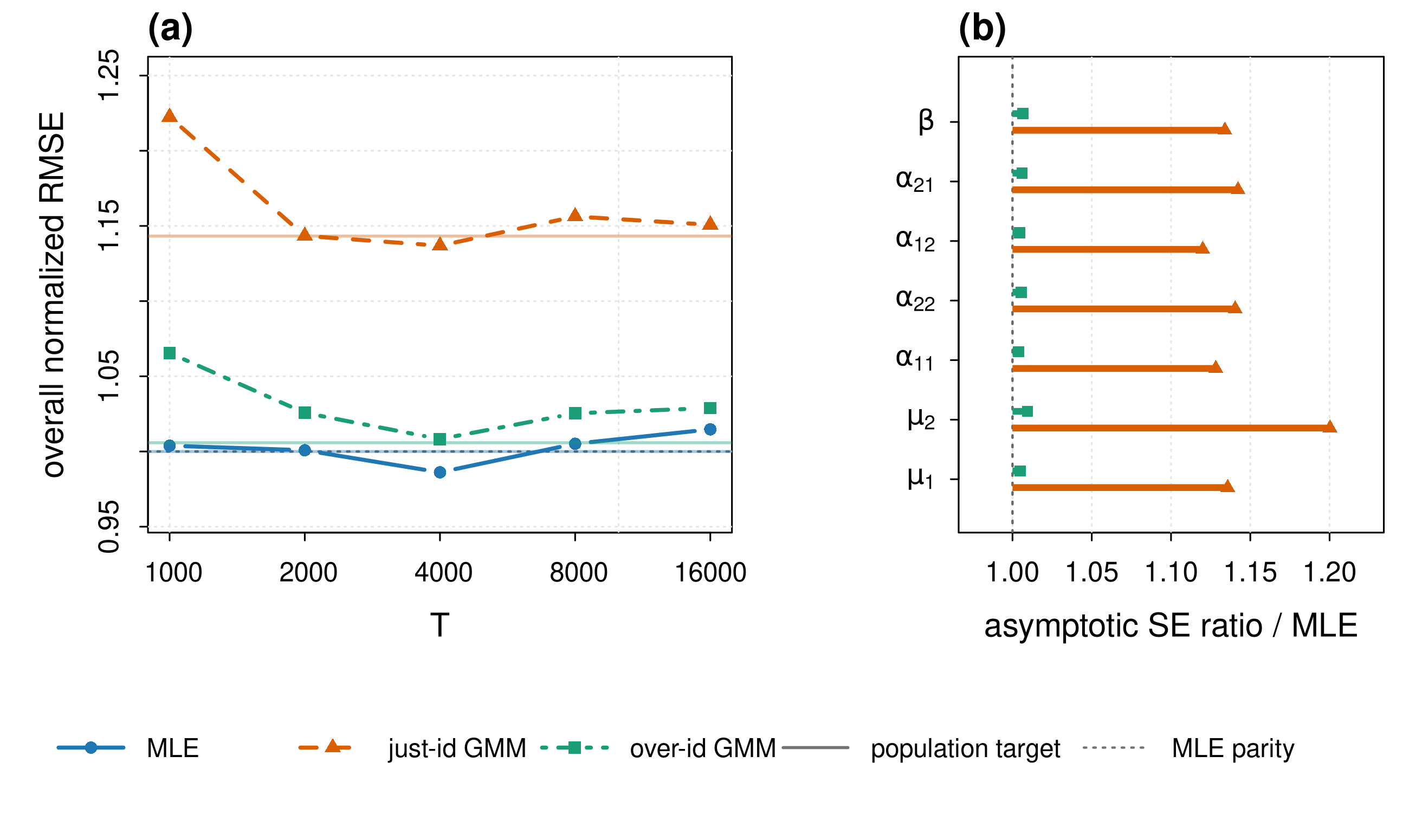}
\caption{(a): Overall normalized RMSE in the bivariate linear Hawkes experiment. The dotted line is the MLE parity benchmark; method-specific horizontal lines are population Godambe targets. (b): Population asymptotic standard-error inflation by parameter, relative to the MLE benchmark. The dotted vertical line is the MLE value.}
\label{fig:bivar-overall-efficiency}
\end{figure}

Figure~\ref{fig:bivar-overall-efficiency}(b) gives the population standard-error inflation by coordinate, \(\{(V_m)_{kk}/(V_M)_{kk}\}^{1/2}\) for \(m\in\{J,O\}\). The least-squares estimator pays a visible penalty across coordinates, while the overidentified estimator is nearly indistinguishable from the MLE benchmark in this design. The message for practitioners is that a stable library whose span better approximates the score recovers substantial efficiency.

\section{Discussion}\label{sec:discussion}
This paper unifies a broad family of estimators used for Hawkes models inside a single inferential geometry and equips admissible finite moment libraries with common guarantees. An admissible moment is any smooth predictable feature of the recent configuration controlled by a polynomial local-count envelope, so the likelihood score, the derivative equation behind the least-squares contrast, and non-score Tak\'acs--Fiksel and Stoyan--Grabarnik-type moments all belong to one family rather than to separate methods. For each admissible GMM member we establish uniform high-probability and almost-sure \(\mathcal O(\sqrt{\log(T)/T})\) rates, asymptotic normality with explicit sandwich covariance, and feasible two-step optimal weighting; the least-squares contrast minimizer has the corresponding high-probability rate and local derivative-weight sandwich limit by Supplement Proposition~\ref{prop:supp-ls-contrast-estimator}. A projection geometry then prices efficiency: the score is the efficient member, and any other library loses exactly the part of the score lying outside its predictable span. The scope is wide on purpose. Signed and inhibitory interactions are handled by the same dependence theory through absolute-kernel stability; nonstationary starts are absorbed by a logarithmic burn-in; and a two-point bound shows the root-\(T\) scale cannot be improved uniformly. For the modeller, then, the choice between likelihood, least squares, and richer libraries is not a jump between unrelated methods but a single, quantified trade-off of computational cost against efficiency.

The limitations made implicit by our assumptions indicate where new tools are required. Allowing the library to grow, \(q=q_T\to\infty\), affects two fixed-dimensional steps: the fixed-dimensional nets behind Theorem~\ref{thm:local-concentration} and the unregularized inversion of \(\hat\Omega_{H,T}\). A growing library would therefore require dimension-explicit martingale concentration, entropy control for the library, and regularized weighting in the spirit of efficient-instrument GMM \citep{newey1990efficient,donald2009choosing,newey2009generalized}. This is also the regime in which the penalized least-squares literature for Hawkes already operates \citep{reynaud2010adaptive,hansen2015,bacry2020sparse}. 

Noncompact memory removes the deterministic window on which the coupling bookkeeping of Section~\ref{sec:concentration} rests, so the local memory count would have to be replaced by a tail-weighted residual-influence functional.   Growing dimension, \(D=D_T\to\infty\), creates a different obstruction.  The stability condition  \(\rho\{G^{\mathrm{abs}}(\theta)\}\le 1-\varepsilon\) is formally dimension-free, but the present constants in the envelopes, count moments, nets, union bounds, and matrix inversions are not tracked as functions of \(D\).  A high-dimensional theory would have to make these dependences explicit, and would likely require sparsity or other structural restrictions, as in \citet{hansen2015}.  Boundary and sparse regimes degenerate the rank condition \eqref{eq:H-rank} rather than the dependence theory. Within the fixed-dimensional compact-window setting studied here, however, the 
compensator class is exhaustive for finite-library comparisons: it contains the 
efficient score-based choice, accommodates any other admissible finite library, and quantifies the resulting efficiency gap.

\acks During the preparation of this work the authors used ChatGPT 5.5 Pro, an AI large language model, for code generation and mathematical proof checking. The authors take full responsibility for all content. This work was supported by the Royal Society of New Zealand Marsden Fund under grant MFP-UOO2518. There were no competing interests to declare which arose during the preparation or publication process of this article. The code used to replicate the results in Section~\ref{sec:simulation} is freely available at \url{https://github.com/ckres213/optimalEstEQ}.

\renewcommand{\doi}{doi: \begingroup \urlstyle{rm}\Url}
\bibliographystyle{abbrvnat}
\bibliography{martingalebib}

\clearpage
\setcounter{page}{1}                            
\renewcommand{\thepage}{S\arabic{page}}         

\begin{center}
  {\Large\bfseries Supplement to ``Optimal Estimating Equations for Compact-Memory Hawkes Processes''}\\[1ex]
  Louis Davis \quad Conor Kresin
\end{center}
\medskip

\setcounter{section}{0}\setcounter{equation}{0}\setcounter{theorem}{0}
\renewcommand{\thesection}{S\arabic{section}}
\renewcommand{\thesubsection}{S\arabic{section}.\arabic{subsection}}
\renewcommand{\theequation}{S\arabic{section}.\arabic{equation}}
\renewcommand{\thetheorem}{S\arabic{section}.\arabic{theorem}}

\section{Concentration via Poisson embedding}\label{app:poisson-embedding}

This section details the proof of Theorem~\ref{thm:local-concentration} and related lemmata. Theorem \ref{thm:local-concentration} is obtained by a coupling argument followed by logarithmic blocking, Bernstein's inequality, and a fixed-dimensional net. However, in order to effectively couple and truncate, we must first prove a few technical results. In particular, we demonstrate that the moment generating function of the counting process over any window of finite length exists in a neighbourhood of 0 (Lemma \ref{lem:fixed-window-count-mgf}), and two coupling Lemmata \ref{lem:nonstationary-return-to-stationarity} and \ref{lem:nonlinearcoupling} driven by the Poisson embedding theory of \cite{bremau1996}. 

For notational brevity, set
\[
    K_{\theta,ij}(u)=L_{\theta,i}|h_{\theta,ij}(u)|\mathbf 1\{0\le u\le A\},
    \qquad
    G_{\theta,ij}=\int_0^A K_{\theta,ij}(u)\,\dd u .
\]
By Assumption~\ref{ass:stability}, $\sup_{\theta\in\Theta}\rho(G_\theta)<1$.  Compactness and continuity then give
\begin{equation}\label{eq:tilted-resolvent-bound}
    \sup_{\theta\in\Theta}
    \left\| (I-G_{\theta})^{-1}\right\|_{\op}<\infty .
\end{equation}

Throughout this section alone do we explicitly state the dependence of $\theta$ on the counting process $N_\theta$, this is to make clear which law we are working under. Elsewhere, we suppress this argument for clarity.

We first prove that the counting process law uniformly returns to the stationary distribution, in total variation distance, under any finite start via an embedding argument.

\begin{lemma}[Exponential return to stationarity from a nonstationary history]
\label{lem:nonstationary-return-to-stationarity}
Suppose Assumptions~\ref{ass:parameter}--\ref{ass:stability} hold and fix
$s\in\R$. Let \(\calH\) be a \(D\)-variate locally finite history on
\((-\infty,s]\), measurable with respect to the filtration at time $s$ but independent of the future Poisson embedding noise. Define the local memory count
\[Z_\calH(s)\vcentcolon= \sum_{j=1}^D \calH_j([s-A,s]).\]
Let $N_\theta^\calH$ be a $D$ dimensional-nonlinear Hawkes process with parameter
\(\theta\), started from history \(\calH\) at time \(s\). Let
\(N_\theta^{\rm stat}\) be a stationary $D$-dimensional nonlinear Hawkes process with the
same parameter, in particular its infinite history is independent of $\calH$. Then there exist constants \(C,c>0\), independent of
\(\theta\), \(s\), \(r\), and \(\calH\), such that, for every \(r\ge0\),
\begin{align}\label{eq:condTV}
d_{\TV}\left(
    \mathcal L_\theta\{N_\theta^\calH|_{(s+r,\infty)}\mid \calH\},
    \mathcal L_\theta\{N_\theta^{\rm stat}|_{(s+r,\infty)}\}
\right)
\le
C e^{-cr}\{1+Z_\calH(s)\}.
\end{align}

Consequently, if
\[
    \sup_{\theta\in\Theta}\E_\theta\left[Z_\calH(s)\right]<\infty,
\]
then
\begin{align}\label{eq:totTv}
\sup_{\theta\in\Theta}
d_{\TV}\left(
    \mathcal L_\theta\{N_\theta^\calH|_{(s+r,\infty)}\},
    \mathcal L_\theta\{N_\theta^{\rm stat}|_{(s+r,\infty)}\}
\right)
\le
C e^{-cr}
\left\{
1+\sup_{\theta\in\Theta}\E_\theta\left[Z_\calH(s)\right]
\right\}.
\end{align}
In particular, for the empty, or stationary, start $\sup_{\theta \in \Theta}\E_\theta\left[Z_\calH(s)\right]<\infty$.
\end{lemma}

\begin{proof}

Following previous notation
\[ K_{\theta,ij}(u)=L_{\theta,i}|h_{\theta,ij}(u)|\mathbf 1\{0\le u\le A\},\qquad G_{\theta,ij}=\int_0^A K_{\theta,ij}(u)\,\dd u .\]

Then Lipschitzness of the link function gives, e.g. using \eqref{eq:LipschitzNonLintoLin}, for any locally finite configuration $N$,
\begin{align*}
m_{\theta,i}&\leq \phi_{\theta,i}(\nu_{\theta,i})+\sum_{j=1}^DG_{\theta,ij}m_{\theta,j},
\end{align*}
where $m_{\theta,i}=\E_\theta\left[\lambda_i(0;\theta)\right]$, under the stationary law. By Theorem 7 of \cite{bremau1996} pointwise in $\theta$, $\max_{i=1}^Dm_{\theta,i}<\infty$. Therefore, by entry-wise non-negativity and that $\rho(G_\theta)<1$ in vector form
\begin{equation}\label{eq:m_starBound}
m_\theta\leq (I-G_\theta)^{-1}\phi_\theta(\nu_\theta)\implies m_*=\sup_{\theta \in \Theta}\sum_{i=1}^D m_{\theta,i}<\infty
\end{equation}
where the finiteness follows from the continuity in $\theta$, compactness of $\Theta$ and the uniform spectral margin, e.g.\ equation \eqref{eq:tilted-resolvent-bound}.

Next we consider an exponential tilting, to do so choose $\alpha>0$ so small that the exponentially tilted matrix
\[ H_{\theta,\alpha}= (H_{\theta,\alpha,ij})_{i,j=1}^D,\qquad H_{\theta,\alpha,ij} =\int_0^\infty e^{\alpha u}K_{\theta,ij}(u)\,\dd u\]
satisfies
\[\sup_{\theta\in\Theta}\rho(H_{\theta,\alpha})<1.\]
This is possible since \(K_{\theta,ij}\) is supported on \([0,A]\), so that $H_{\theta,\alpha,ij}\leq e^{\alpha A}G_{\theta,ij}.$

Next, define the constant
\[R_\alpha \vcentcolon= \sup_{\theta\in\Theta}\mathbf 1^\top (I-H_{\theta,\alpha})^{-1}\mathbf 1<\infty \]
where $\mathbf 1$ is the $D$-1s vector.

Now fix \(\theta\in\Theta\) and condition on the initial history \(\calH\).  Let
\(\xi_\theta\) be an independent stationary history on \((-\infty,s]\) with
law $\mathcal L_\theta\left\{N_\theta^{\rm stat}|_{(-\infty,s]}\right\}.$

Let \(\Pi=(\Pi_1,\ldots,\Pi_D)\) be independent Poisson random measures on
\((s,\infty)\times\mathbb R_+\) with common intensity \(\dd t\,\dd z\). We now aim to construct a
coupling law \(\mathbb Q_\theta^{s,\calH}\), for two non-linear Hawkes processes with history $\calH$ and $\xi_\theta$ respectively. We construct \(N_\theta^\calH\) from the
history \(\calH\), \(N_\theta'\) from  the history \(\xi_\theta\) by
driving both processes after time \(s\) with the same Poisson random measures as done by \cite{bremau1996} (see for example \citep[][Theorem 3]{bremau1996}).
\[N_{\theta,i}^\calH(\dd t)=\int_0^\infty\1\left\{z\le \lambda_{\theta,i}^\calH(t;\theta)\right\}\Pi_i(\dd t,\dd z),\qquad t>s,\]
and
\[N_{\theta,i}'(\dd t)=\int_0^\infty\1\left\{z\le \lambda_{\theta,i}'(t;\theta)\right\}\Pi_i(\dd t,\dd z),\qquad t>s.\]
Conditional on \(\calH\), the first marginal is the nonstationary law started
from \(\calH\), and the second marginal is the stationary law
\(\Prob_\theta^{\rm nl}\).

For \(t>s\), define the disagreement measure
\[
    D_{\theta,i}(\dd t)
    =
    |N_{\theta,i}^\calH-N_{\theta,i}'|(\dd t).
\]
Equivalently,
\[
D_{\theta,i}(\dd t)
=
\int_0^\infty
\left|
\mathbf 1\{z\le \lambda_{\theta,i}^\calH(t;\theta)\}
-
\mathbf 1\{z\le \lambda_{\theta,i}'(t;\theta)\}
\right|
\Pi_i(\dd t,\dd z).
\]
The two indicators differ exactly when the shared mark \(z\) lies between the
two intensities.  Therefore,
\begin{align}\label{eq:compensatorCouple}
    \E_{\mathbb Q_\theta^{s,\calH}}
    \left[
    D_{\theta,i}(\dd t)\mid \mathcal F_{t-}
    \right]
    =
    |\lambda_{\theta,i}^\calH(t;\theta)
      -\lambda_{\theta,i}'(t;\theta)|\,\dd t .
\end{align}

By the Lipschitz property of \(\phi_{\theta,i}\),
\begin{align}\label{eq:lipschitzinequality}
|\lambda_{\theta,i}^\calH(t;\theta)
-\lambda_{\theta,i}'(t;\theta)|
&\le
\sum_{j=1}^D
\int_{(s,t)}
K_{\theta,ij}(t-u)D_{\theta,j}(\dd u)  \\
&\quad+
\sum_{j=1}^D
\int_{(-\infty,s]}
K_{\theta,ij}(t-u)
\,d|\calH_j-\xi_{\theta,j}|(u).
\end{align}

 Conditional on \(\calH\), the triangle inequality and the stationarity of \(\xi_\theta\) implies
\begin{align}\label{eq:condexp}
& \E_{\mathbb Q_\theta^{s,\calH}}
\left[
\sum_{j=1}^D
\int_{(-\infty,s]}
K_{\theta,ij}(t-u)
\,d|\calH_j-\xi_{\theta,j}|(u)
\,\middle|\,\calH
\right]  \\
&\qquad\le
\sum_{j=1}^D
\int_{[s-A,s]}
K_{\theta,ij}(t-u)\,d\calH_j(u)
+
\sum_{j=1}^D
m_{\theta,j}
\int_{-\infty}^{s}
K_{\theta,ij}(t-u)\,\dd u .
\end{align}

For \(v\ge0\), define
\[ b_{\theta,i}^{s,\calH}(v):=\sum_{j=1}^D \int_{[s-A,s]} K_{\theta,ij}(s+v-u)\,d\calH_j(u) + \sum_{j=1}^D m_{\theta,j}\int_{v}^{\infty}K_{\theta,ij}(w)\,dw .\]
Then the conditional expectation in equation \eqref{eq:condexp} is bounded by
\(b_{\theta,i}^{s,\calH}(t-s)\).  Since \(K_{\theta,ij}\) is supported on
\([0,A]\), we also have $ b_{\theta,i}^{s,\calH}(v)=0$ when $v>A$.

For any $T>s$, define
\[x_{\theta,i}^{s,\calH}(T) \vcentcolon=\E_{\mathbb Q_\theta^{s,\calH}}\left[\int_s^T e^{\alpha(t-s)}D_{\theta,i}(\dd t)\,\middle|\,\calH\right],\]
Using the compensator identity \eqref{eq:compensatorCouple}, the Lipschitz bound \eqref{eq:lipschitzinequality}, and the definition of
\(b_{\theta,i}^{s,\calH}\),
\begin{align}\nonumber
x_{\theta,i}^{s,\calH}(T)&\le\int_s^T e^{\alpha(t-s)}b_{\theta,i}^{s,\calH}(t-s)\,\dd t \\\label{eq:xidef2}
&\quad+\E_{\mathbb Q_\theta^{s,\calH}}\left[\int_s^T e^{\alpha(t-s)}\left(\sum_{j=1}^D\int_s^t K_{\theta,ij}(t-u)D_{\theta,j}(\dd u)\right)\dd t\,\middle|\,\calH\right].
\end{align}

If we define
\begin{multline}
    \beta_{\theta,i}^{s,\calH}\vcentcolon=
    \int_0^\infty e^{\alpha v}b_{\theta,i}^{s,\calH}(v)\,dv \\
    =  \int_0^\infty e^{\alpha v}\sum_{j=1}^D\int_{[s-A,s]}K_{\theta,ij}(s+v-u)d\calH_j(u)\,dv + \int_0^\infty e^{\alpha v}\sum_{j=1}^D m_{\theta,j}\int_v^\infty K_{\theta,ij}(w)\,dw\,dv\\
    =\vcentcolon I_i+II_i
\end{multline}
Then it is clear that
\[
    \int_s^T e^{\alpha(t-s)}
    b_{\theta,i}^{s,\calH}(t-s)\,\dd t
    \le
    \beta_{\theta,i}^{s,\calH}.
\]

We now bound \(\sum_{i=1}^D\beta_{\theta,i}^{s,\calH}\).  Its first term, $I_i$, satisfies,
\[
\begin{aligned}
\sum_{i=1}^DI_i&=\sum_{i=1}^D\sum_{j=1}^D\int_0^\infty e^{\alpha v}\int_{[s-A,s]}K_{\theta,ij}(s+v-u)\,d\calH_j(u)\,dv  \\
&=\sum_{i=1}^D\sum_{j=1}^D\int_{[s-A,s]}\left[\int_0^\infty e^{\alpha v}K_{\theta,ij}(s+v-u)\,dv\right]d\calH_j(u)  \\
&\leq e^{\alpha A}\sum_{i=1}^D\sum_{j=1}^DG_{\theta,ij}\calH_j([s-A,s])  \\
&\le C Z_\calH(s).
\end{aligned}
\]
The second term, $II_i$, satisfies
\[
\begin{aligned}
\sum_{i=1}^D II_i&=\sum_{i=1}^D\sum_{j=1}^D m_{\theta,j}\int_0^\infty e^{\alpha v}\int_v^\infty K_{\theta,ij}(w)\,dw\,dv  \\
&=\sum_{i=1}^D\sum_{j=1}^D m_{\theta,j}\int_0^\infty K_{\theta,ij}(w)\int_0^w e^{\alpha v}\,dv\,dw  \\
&\le A e^{\alpha A}\sum_{i=1}^D\sum_{j=1}^D m_{\theta,j}G_{\theta,ij}\le C.
\end{aligned}
\]
Here, \(C<\infty\) is uniform over \(\theta\in\Theta\), using the uniform bound
on \(\sum_{j=1}^D m_{\theta,j}\) and the compactness of the parameter space $\Theta$.  Therefore,
\[\sum_{i=1}^D \beta_{\theta,i}^{s,\calH} \le C\{1+Z_\calH(s)\}.\]

It remains to treat the term in equation \eqref{eq:xidef2}.
\[
\begin{aligned}
& \E_{\mathbb Q_\theta^{s,\calH}} \left[\int_s^T e^{\alpha(t-s)}\left(\sum_{j=1}^D\int_s^t K_{\theta,ij}(t-u)D_{\theta,j}(\dd u)\right)\dd t\,\middle|\,\calH\right] \\
&\qquad=\sum_{j=1}^D \E_{\mathbb Q_\theta^{s,\calH}}\left[\int_{(s,T]}\left(\int_u^T e^{\alpha(t-s)}K_{\theta,ij}(t-u)\,\dd t\right)D_{\theta,j}(\dd u)\,\middle|\,\calH\right] \\
&\qquad\le\sum_{j=1}^DH_{\theta,\alpha,ij}\E_{\mathbb Q_\theta^{s,\calH}}\left[\int_{(s,T]} e^{\alpha(u-s)}D_{\theta,j}(\dd u)\,\middle|\,\calH\right].
\end{aligned}
\]
Thus, in vector notation,
\[
    x_\theta^{s,\calH}(T)
    \le
    \beta_\theta^{s,\calH}
    +
    H_{\theta,\alpha}x_\theta^{s,\calH}(T).
\]
Since \(H_{\theta,\alpha}\) is nonnegative and has spectral radius strictly
less than one,
\[
    x_\theta^{s,\calH}(T)
    \le
    (I-H_{\theta,\alpha})^{-1}\beta_\theta^{s,\calH}.
\]
Hence,
\[\sum_{i=1}^D x_{\theta,i}^{s,\calH}(T)\le\mathbf 1^\top(I-H_{\theta,\alpha})^{-1}\beta_\theta^{s,\calH} \le R_\alpha\sum_{i=1}^D \beta_{\theta,i}^{s,\calH} \le
C\{1+Z_\calH(s)\}.\]
Letting \(T\to\infty\) and applying the monotone convergence theorem yields
\[\E_{\mathbb Q_\theta^{s,\calH}}\left[\sum_{i=1}^D\int_{(s,\infty)}e^{\alpha(t-s)}D_{\theta,i}(\dd t)\,\middle|\,\calH\right]\le C\{1+Z_\calH(s)\}.\]

Next, fix \(r\ge0\) so that
\[
\begin{aligned}
\E_{\mathbb Q_\theta^{s,\calH}}
\left[
\sum_{i=1}^D
D_{\theta,i}((s+r,\infty))
\,\middle|\,\calH
\right]
&\le
e^{-\alpha r} \E_{\mathbb Q_\theta^{s,\calH}}
\left[
\sum_{i=1}^D
\int_{(s,\infty)}e^{\alpha(t-s)}D_{\theta,i}(\dd t)
\,\middle|\,\calH
\right]  \\
&\le
C e^{-\alpha r}\{1+Z_\calH(s)\}.
\end{aligned}
\]
Therefore, by Markov's inequality,
\[
\mathbb Q_\theta^{s,\calH}
\left(
N_\theta^\calH|_{(s+r,\infty)}
\ne
N_\theta'|_{(s+r,\infty)}
\,\middle|\,\calH
\right)
\le
C e^{-\alpha r}\{1+Z_\calH(s)\}.
\]
The coupling characterization of the total variation distance yields
\[
d_{\TV}\left(
    \mathcal L_\theta\{N_\theta^\calH|_{(s+r,\infty)}\mid \calH\},
    \mathcal L_\theta\{N_\theta^{\rm stat}|_{(s+r,\infty)}\}
\right)
\le
C e^{-\alpha r}\{1+Z_\calH(s)\}.
\]
Renaming \(\alpha\) as \(c\) proves the equation \eqref{eq:condTV}.

Equation \eqref{eq:totTv} follows from the convexity of the total variation distance, equation \eqref{eq:condTV} and then finally taking the supremum over \(\theta\in\Theta\).

Finally, if the history is empty then $Z_\calH(s)=0$ almost surely, so that the conditional expectation bound is obvious, and under a stationary equation \ref{eq:m_starBound}, and the compensator identity \[\E_\theta\left[N_\theta\left([-A,0)\right)\right]=A\sum_{i=1}^D\E_\theta \left[\lambda_i(0;\theta)\right]\]
give the required uniform finite first moment.

\end{proof}

We now use the previous coupling lemma to demonstrate that the moment generating function of the window count of a stationary non-linear multivariate Hawkes process exists. This follows from a Poisson embedding and domination argument, concluded by Theorem 1 of \cite{leblanc2024}.

\begin{lemma}[Uniform fixed-window exponential moments]\label{lem:fixed-window-count-mgf}
Suppose Assumptions~\ref{ass:parameter}--\ref{ass:stability} hold. Then, for every fixed $0\le L<\infty$ there exist constants $\xi_L,C_L>0$ such that
\begin{equation}\label{eq:count-mgf}
    \sup_{\theta\in\Theta}\sup_{s\in\mathbb R}
    \E_\theta\left[\exp\{\xi_L N_\theta([s,s+L])\}\right]
    \le C_L.
\end{equation}
\end{lemma}

\begin{proof}
Write
\[ K_{\theta,ij}(u)=L_{\theta,i}|h_{\theta,ij}(u)|\1\{0\le u\le A\}.\]
Let $\bar N_\theta$ be the linear Hawkes process with immigrant arrival rate $\phi_{\theta,i}(\nu_{\theta,i})$ and nonnegative kernels $K_{\theta,ij}$, realised with a stationary history on $(-\infty,-m]$ and independent driving Poisson random measure $\Pi$ for time $t>-m$. Next, let $N_\theta^{[-m]}|_{(s+r,\infty)}$ be the nonlinear multivariate Hawkes process started at time $-m$ with an empty history, with componentwise conditional intensity $\lambda_i^{\rm nl}(t;\theta)$ and for $t>-m$ drive it with the same Poisson random measure $\Pi$.  Therefore, the common Poisson embedding with thinning \citep[e.g.][Lemma 4.6]{Patrichi2018} says that both processes may be realized from the same family of driving Poisson random measures. Therefore, pathwise
\begin{align}\label{eq:LipschitzNonLintoLin}
\lambda_i^{\rm nl}(t;\theta)
&\vcentcolon=
\phi_{\theta,i}\!\left(\nu_{\theta,i}+\sum_{j=1}^D\int_{[t-A,t)}h_{\theta,ij}(t-u)\,\dd N_{j,\theta}(u)\right) \\
&\le
\phi_{\theta,i}(\nu_{\theta,i})+
\sum_{j=1}^D\int_{[t-A,t)}K_{\theta,ij}(t-u)\,\dd \bar N_{j,\theta}(u)=\vcentcolon \bar \lambda_i(t;\theta).
\end{align}
Thus, the histories of $\bar N_\theta$ and $N^{[-m]}_\theta$ are ordered, prior to time $-m$, so that the acceptance region $ (0,\lambda_i^{\rm nl}(t;\theta)] \subset (0,\bar \lambda_i(t;\theta)]$ for every $t>-m$. Hence, the common-thinning construction preserves order over all future times, \citet[Lemma 3]{bremau1996} and \citet[Proposition 4.16]{Patrichi2018}. Therefore, $N_\theta^{[-m]}([-m,\infty))\le\bar N_\theta([-m,\infty))$ almost surely.

Next, let $N_\theta^{\rm stat}|_{(s+r,\infty)}$ be a stationary nonlinear multivariate Hawkes process, with the same parametric conditional intensity function as $N_\theta^{[-m]}$. Pick $s=-m$ and $r=m-1$ so that for any fixed window $[0,L]$ equation \eqref{eq:totTv} gives
\begin{equation}\label{eq:mstartHistTV}
\sup_{\theta\in\Theta}
d_{\TV}\left(
    \mathcal L_\theta\{N_\theta^{[-m]}|_{[0,L]}\},
    \mathcal L_\theta\{N_\theta^{\rm stat}|_{[0,L]}\}
\right)
\le
C e^{-c(m-1)}.
\end{equation}
Such an inequality follows since the restriction from $(-1,\infty)$ to $[0,L]$ is measurable so that equation \eqref{eq:mstartHistTV} holds via a data-processing inequality.

The stochastic ordering under the embedding implies that
\begin{equation}\label{eq:mstartMGFbound}
\sup_{\theta\in\Theta}\E_\theta\left[\exp\{\xi_L N_\theta^{[-m]}([0,L])\}\right]\leq     \sup_{\theta\in\Theta}\E_\theta\left[\exp\{\xi_L \bar N_\theta([0,L])\}\right]
\end{equation}

To ensure the right hand side of equation  \eqref{eq:mstartMGFbound} is finite, consider the following argument. The productivity matrix of $\bar N_\theta$ is $G_\theta=\left(\int_0^A K_{\theta,ij}(u)\,\dd u\right)_{i,j=1}^D$, and by Assumption~\ref{ass:stability}, there exists some fixed $\delta>0$ such that $\sup_{\theta\in\Theta}\rho(G_\theta)=1-\delta$. Set $r=1-\delta/2$. Since $G_\theta$ is nonnegative with $\rho(r^{-1}G_\theta)\le(1-\delta)/(1-\delta/2)<1$, every power is dominated entrywise by the Neumann series, 
\[\left(\frac{G_\theta}{r}\right)^n\le\sum_{m\ge0}\left(\frac{G_\theta}{r}\right)^m=(I-r^{-1}G_\theta)^{-1}\implies \norm{G_\theta^n}_\infty\le K r^n\]
for all $n\ge1$ with $K:=\sup_{\theta\in\Theta}\norm{(I-r^{-1}G_\theta)^{-1}}_\infty$, which is finite by the continuity of $\theta\mapsto G_\theta$. Such a result should be clear by the stated assumptions, and compactness of $\Theta$, exactly as in \eqref{eq:tilted-resolvent-bound}. Therefore, we may apply Theorem 1 of \cite{leblanc2024} with this pair $(r,K)$. Together with the uniform upper bound of the baseline intensities the resulting upper bound of the exponential-moment is uniform over $\Theta$. Thus, the right hand side of equation  \eqref{eq:mstartMGFbound} is finite for sufficiently small $\xi_L$, additionally, call this upper bound $C_L$.

Therefore, for every finite $m$
\[ \sup_{\theta\in\Theta}\E_\theta\left[\exp\{\xi_L N_\theta^{[-m]}([0,L])\}\right]\leq C_L<\infty.\]
We now pass to the stationary limit and to do so, define $g_R(k)=\min\left\{e^{\xi_Lk},R\right\}$. Then by equation \eqref{eq:mstartHistTV} 
\begin{align*}
    \E_\theta\left[g_R\left(N_\theta^{\rm stat}([0,L])\right)\right]=\lim_{m\to \infty} \E_\theta\left[g_R\left(N_\theta^{[-m]}([0,L])\right)\right]\leq C_L.
\end{align*}
The monotone convergence theorem, taking $R\to \infty$, yields
\[  \E_\theta\left[\exp\left\{\xi_LN_\theta^{\rm stat}([0,L])\right\}\right]\leq C_L.\]
Since $C_L$ is independent of $\theta$, stationarity, and taking a supremum over $\theta$ yields equation \eqref{eq:count-mgf}.

Polynomial moments follow by the existence of the exponential moment in a neighbourhood of $0$.
\end{proof}

The previous lemma, demonstrating that the random variable denoting the number of events of the process, in any fixed window, has a moment generating function, allows us to show that the counting process is well-behaved, formalized in the following lemma.

\begin{lemma}[Sliding-window truncation]\label{lem:count-truncation}
Under Assumptions~\ref{ass:parameter}--\ref{ass:stability}, for every $c>0$ there exist constants $C,K<\infty$ such that, for all sufficiently large $T$,
\begin{equation}\label{eq:count-truncation}
\sup_{\theta^\star\in\Theta_0}
\Prob_{\theta^\star}\left(
\sup_{0\le t\le T}\mathfrak C_t>K\log(T)
\right)
\le C T^{-c} .
\end{equation}
On the same event, after enlarging $K$,
\begin{equation}\label{eq:total-count-truncation}
    \frac{1}{T}N([0,T])\le K\log(T) .
\end{equation}
\end{lemma}

\begin{proof}
Place a deterministic grid of mesh length $A/2$ on $[-A,T]$ there are at most $C_A(1+T)$ of such grid intervals. Every window $[t-A,t)$, $0\le t\le T$, intersects at most three grid intervals, so a window count exceeding $u$ implies that one grid interval contains more than $u/3$ points.  Lemma~\ref{lem:fixed-window-count-mgf} and a union bound with stationarity give
\begin{align*}\Prob_{\theta^\star}\left(\sup_{0\le t\le T}\mathfrak C_t>u\right)\leq C_A(1+T)e^{-c_0u},\end{align*}
uniformly over $\Theta_0$.  Taking $u=K\log(T)$ proves \eqref{eq:count-truncation}.  Partitioning $[0,T]$ into intervals of length at most $A$ proves \eqref{eq:total-count-truncation} on the same event.
\end{proof}

Now using Lemma \ref{lem:nonstationary-return-to-stationarity} we are able to construct $n$ independent stationary nonlinear multivariate Hawkes processes on disjoint separated blocks that are close in total variation distance to any indivual stationary nonlinear multivariate Hawkes process restricted to the same blocks.

\begin{lemma}[Block coupling for nonlinear compact-memory Hawkes processes]
\label{lem:nonlinearcoupling}
Suppose Assumptions~\ref{ass:parameter}--\ref{ass:stability} hold.  Let
\(N_\theta\sim \Prob_\theta\) be the stationary nonlinear Hawkes process with intensity
\[
    \lambda_i(t;\theta)
    =
    \phi_{\theta,i}\!\left(
        \nu_{\theta,i}
        +
        \sum_{j=1}^D
        \int_{[t-A,t)}
        h_{\theta,ij}(t-s)\,\dd N_{j,\theta}(s)
    \right),
    \qquad i=1,\ldots,D .
\]
Let \(I_q=[a_q,b_q]\), \(q=1,\ldots,n\), be deterministic intervals ordered
from left to right and separated by gaps at least \(g>0\), i.e.
\(a_{q+1}-b_q\ge g\).  Then there exist constants \(C,c>0\), independent of
\(\theta\), \(n\), \(g\), and the intervals, such that
\begin{equation}\label{eq:nonlinBlockCouple}
\sup_{\theta\in\Theta}
d_{\TV}\left(
    \calL_\theta\{N|_{I_1},\ldots,N|_{I_n}\},
    \bigotimes_{q=1}^n \calL_\theta\{N|_{I_q}\}
\right)
\le Cne^{-cg}.
\end{equation}
\end{lemma}

\begin{proof}
The proof of this lemma exploits the result of Lemma \ref{lem:nonstationary-return-to-stationarity} by considering Hawkes processes with two independent stationary histories. In particular, the coupling constructed in Lemma \ref{lem:nonstationary-return-to-stationarity} implies that if $I=[a,b]$ and $a\geq s+r$
\begin{equation}\label{eq:IntTVbound}
\Q_\theta^s\left(N_\theta|_{I}\neq N_\theta'|_{I}\right)\leq C_0e^{-c r}
\end{equation}
where $C_0$ absorbs the integrable finite local stationary history $Z_\calH(s)$.

We now construct the independent block copies. Consider the deterministic intervals $I_q=[a_q,b_q]$ for $q=1,\dots,n$ with gaps $g$ so that $a_{q+1}-b_q\geq g$, let $r=g/3$ and define $s_q=a_q-r$. Consider the enlarged windows $J_q=(s_q,b_q]$, each of which is disjoint since $s_{q+1}=a_{q+1}-r\geq b_q+g-r>b_q$.

We construct the coupling law $\Q_\theta$ as follows. Sample one stationary nonlinear Hawkes process $N_\theta \sim \Prob_\theta^{\text{nl}}$, and represent it through the Poisson embedding with $\Pi$, as was done in the proof of Lemma \ref{lem:nonstationary-return-to-stationarity}. For each block $q$, construct an auxiliary process, $\tilde N_\theta^{(q)}$ in the following manner. 
\begin{enumerate}
    \item Prior to time $s_q$, give $\tilde N_\theta^{(q)}$ an independent stationary nonlinear history with law $\Prob_\theta^{\text{nl}}$ restricted to $(-\infty,s_q]$.
    \item On the window $J_q=(s_q,b_q]$ drive $\tilde N_\theta^{(q)}$ using the same Poisson embedding $\Pi$ as for $N_\theta$.
\end{enumerate}

Specifically, for $t\in J_q$,
\[\tilde N_{\theta,i}^{(q)}(dt)=\int_0^\infty \1\left\{z\leq \tilde\lambda_{\theta,i}^{\text{nl},(q)}(t;\theta)\right\}\Pi_i(\dd t,\dd z)\]
where 
\[\tilde\lambda_{\theta,i}^{\text{nl},(q)}(t;\theta)=\phi_{\theta,i}\left(\nu_{\theta,i}+\sum_{j=1}^D\int_{[t-A,t)}h_{\theta,ij}(t-u)d\tilde N_{\theta,j}^{(q)}(u)\right)\]
Since the history of each $\tilde N_\theta^{(q)}$ is independent, and the Poisson embedding noises are independent across $q$ as the windows $J_q$ are disjoint (Poisson random measures have independent increments) the family $\{\tilde N_\theta^{(q)}|_{I_q}\}_{q=1}^n$ are independent under $\Q_\theta$ so that 
\[\calL_{\Q_\theta}\left(\{\tilde N_\theta^{(q)}|_{I_q}\}_{q=1}^n\right)=\bigotimes_{q=1}^n\calL_{\Q_\theta}\left(\tilde N_\theta^{(q)}|_{I_q}\right).\]
Now for a fixed pair $(N_\theta,\tilde N_\theta^{(q)})$ on the interval $(s_q,b_q]$ has exactly the same distribution as the two-process coupling under $\Q_\theta^{s_q}$. As $r=a_q-s_q=g/3$, equation \eqref{eq:IntTVbound} gives 
\[\Q_\theta\left(N_\theta|_{I_q}\neq \tilde N_\theta^{(q)}|_{I_q}\right)\leq C_0e^{-\alpha g/3}.\]

Via the coupling characterisation of the total variation distance, and a union bound,
\begin{align*}d_{\TV}\left(\calL_{\Prob_\theta^{\text{nl}}}(N_\theta|_{I_1},\ldots,N_\theta|_{I_n}),
\bigotimes_{q=1}^n\calL_{\Prob_\theta^{\text{nl}}}(\tilde N_\theta^{(q)}|_{I_q})
\right)&\leq \Q_\theta\left(\exists q \ \mid \ N_\theta|_{I_{q}}\neq \tilde N_\theta^{(q)} |_{I_q}\right)\leq C_0ne^{-c_0g}.
\end{align*}
As $C_0,c_0$ are independent of $\theta$ taking the supremum yields equation \eqref{eq:nonlinBlockCouple} so that the proof concludes.
\end{proof}

Just prior to proving Theorem \ref{thm:local-concentration} we consider a practical scenario which allows for the local initial history to be integrable, where the history is non-empty but stochastic and finite. In particular, this is the case in which the history starts off empty at some prior finite time, and then propagates forward via the non-explosive Hawkes process. As a concrete example suppose that the process truly started at time $T_{\rm int}$ with an empty history. Then if data collection starts at some future time, $T_{\rm data}>T_{\rm int}$, even though the process itself is not stationary for $t>T_{\rm data}$ the local history in $[T_{\rm data}-A,T_{\rm data}]$ has a first moment, so that the law generating the data converges exponentially fast in total variation distance to the stationary law.
\begin{lemma}
\label{lem:finite-expected-local-history-nonstationary}
Suppose Assumptions~\ref{ass:parameter}--\ref{ass:stability} hold. Fix a
finite starting time \(\tau\in\mathbb R\). Let \(N^{\calH}_\theta\) be the
Hawkes process with parameter \(\theta\), constructed from an initial history
\(\calH\) on \((-\infty,\tau]\).  Assume that
\[
    Z_\calH(\tau)
    :=
    \sum_{j=1}^D\calH_j([\tau-A,\tau])
\]
has finite expectation uniformly over the parameter set:
\[
    \sup_{\theta\in\Theta}\E_\theta \left[Z_\calH(\tau)\right]<\infty .
\]
Then, for every finite \(B<\infty\),
\[
    \sup_{\theta\in\Theta}
    \sup_{\tau\le s\le \tau+B}
    \E_\theta
    \left[\sum_{j=1}^D
    N_{j,\theta}^{\calH}([s-A,s])\right]
    <\infty .
\]
In particular, if the process is started from the empty history at time
\(\tau\), then the conclusion holds for every finite \(B\).
\end{lemma}

\begin{proof}
Following the notation of Lemma \ref{lem:fixed-window-count-mgf}, and by equation \eqref{eq:LipschitzNonLintoLin}, under the common Poisson embedding \citep[e.g.][]{bremau1996}, \(N^{\calH}_\theta\) is dominated by
the nonnegative linear Hawkes envelope \(\bar N^{\calH}_\theta\) with baseline
\(\phi_{\theta,i}(\nu_{\theta,i})\), kernels \(K_{\theta,ij}\), and the same initial history
\(\calH\).  It is therefore enough to prove the claim for
\(\bar N^{\calH}_\theta\).

Let
\[\bar m_i(t)= \E_\theta\left[ \bar\lambda_i(t;\theta)\right], \qquad t\ge \tau,
\]
where \(\bar\lambda_i\) is the intensity of the linear envelope.  Since the
kernels are compactly supported,
\[\bar m_i(t)\le
    \sup_{\theta\in\Theta} \phi_{\theta,i}(\nu_{\theta,i})+\E_\theta\left[\sum_{j=1}^D\int_{[\tau-A,\tau]} K_{\theta,ij}(t-u)\,d\calH_j(u)\right]+\sum_{j=1}^D\int_\tau^tK_{\theta,ij}(t-u)\bar m_j(u)\,\dd u \]
The first term is uniformly bounded by compactness. For the second summand, compact support and boundedness of the kernels ensures that
\[\E_\theta\left[\sum_{j=1}^D\int_{[\tau-A,\tau]}K_{\theta,ij}(t-u)\,d\calH_j(u)\right]\le C \E_\theta \left[Z_\calH(\tau)\right],\]
uniformly in \(t\ge\tau\), \(i\), and \(\theta\), in addition to the fact that $t\mapsto \sum_{i=1}^D\bar m_i(t)$ is locally finite on compact intervals due to the finite initial history.  Hence, 
\[\sum_{i=1}^D \bar m_i(t)\le C\{1+\E_\theta Z_\calH(\tau)\} + C\int_\tau^t \sum_{i=1}^D \bar m_i(u)\,\dd u,\qquad t\ge \tau,\]
where \(C<\infty\) is deterministic and uniform over \(\theta\in\Theta\). By
Gronwall's inequality for locally finite measures \citep[e.g.][]{HORVATH1996183}, for every finite \(B<\infty\),
\[\sup_{\theta\in\Theta}\sup_{\tau\le t\le \tau+B}\sum_{i=1}^D \bar m_i(t)\leq C_B\left\{1+\sup_{\theta\in\Theta}\E_\theta \left[Z_\calH(\tau)\right]\right\} <\infty .\]
Therefore,
\[\sup_{\theta\in\Theta}\E_\theta\left[ N^{\calH,\theta}([\tau,\tau+B])\right]\le \sup_{\theta\in\Theta}\E_\theta \left[\bar N^{\calH}_\theta([\tau,\tau+B]) \right]=\sup_{\theta\in\Theta}\int_\tau^{\tau+B} \sum_{i=1}^D \bar m_i(t)\,\dd t<\infty.\]

Finally, if \(s\in[\tau,\tau+B]\), then clearly $[s-A,s]\subset[\tau-A,\tau+B]$ so that
\[\sup_{\theta\in\Theta}\sup_{\tau\le s\le \tau+B} \E_\theta\left[ \sum_{j=1}^D N_{j,\theta}^{\calH}([s-A,s])\right]<\infty\]
and the proof concludes.
\end{proof}

We next prove a lemma that uses the total variation bound \eqref{eq:totTv}, to transfer the statistical guarantees of Section \ref{sec:results} to any start which has an integrable prehistory count. This is facilitated by finite sample total variation bound introduced in Lemma \ref{lem:nonstationary-return-to-stationarity}.

\begin{lemma}[Total-variation transfer]\label{lem:tv-transfer}
Suppose that the same assumptions as Lemma \ref{lem:nonstationary-return-to-stationarity} hold. Let $N_\theta^\calH|_{(b_T,\infty)}$ be a sample path of the non-linear Hawkes with history $\calH$ on $(-\infty,0]$, similarly, let $N_\theta^{\rm stat}|_{(b_T,\infty)}$ be a sample path with a stationary start at $0$. 

(i) Let $b_T=A+L\log(T)$. Then, for every $c_0>0$ there exists a $C>0$ uniform over $\Theta_0$ and $L>0$ so that for every measurable event $E_T$, and any measurable map $\psi_T$
\begin{equation*}
    \sup_{\theta \in \Theta_0}\Prob_\theta\left(\psi_T\left(N_\theta^\calH|_{(b_T,\infty)}\right)\in E_T\right)\leq \sup_{\theta \in \Theta_0}\Prob_\theta\left(\psi_T\left(N_\theta^{\rm stat}|_{(b_T,\infty)}\right)\in E_T\right)+CT^{-c_0}.
\end{equation*}

(ii) Moreover, for any bounded and measurable $f$
\begin{equation}\label{eq:BurninDistributionalLimit}
\left|\E_\theta\left[f\left(\psi_T\left(N_\theta^\calH|_{(b_T,\infty)}\right)\right)\right]-\E_\theta \left[f\left(\psi_T\left(N_\theta^{\rm stat}|_{(b_T,\infty)}\right)\right)\right]\right|\leq 2\norm{f}_\infty CT^{-c_0} 
\end{equation}
\end{lemma}
\begin{proof}
    The proof of (i) is obvious by equation \eqref{eq:totTv}. For any $c_0>0$  pick $s=0$, $r=b_T$, and $L$ so large that $cL>c_0$. Finally, use the standard data processing inequality saying that 
    \[d_{\TV}\left(\calL\psi_T\left(N_\theta^{\rm stat}|_{(b_T,\infty)}\right),\calL \psi_T\left(N_\theta^{\calH}|_{(b_T,\infty)}\right) \right)\leq d_{\TV}\left(\calL\left(N_\theta^{\rm stat}|_{(b_T,\infty)}\right),\calL \left(N_\theta^{\calH}|_{(b_T,\infty)}\right)\right).\]

    (ii) Easily follows from the characterisation of total variation as the expected difference of bounded test functions and the aforementioned data processing inequality.
\end{proof}

\subsection{Proof of Proposition~\ref{prop:log-burn-in} (logarithmic burn-in)}\label{app:proof:prop:log-burn-in}
\begin{proof}
Lemma~\ref{lem:nonstationary-return-to-stationarity} gives exponential total-variation convergence of the nonstationary law to the stationary law, uniformly over \(\Theta_0\); Lemma~\ref{lem:finite-expected-local-history-nonstationary} verifies the finite-expected-local-count hypothesis for any history launched from a finite earlier time (in particular an empty start); and Lemma~\ref{lem:tv-transfer} converts the resulting total-variation gap on the burnt-in window into the stated transfer of every high-probability bound and fixed-parameter limit law. Taking \(b_T=A+L\log T\) with \(L\) large enough yields the displayed conclusions over \([b_T,T]\), with normalization by \(T-b_T\).
\end{proof}

\subsection{Proof of Theorem~\ref{thm:local-concentration} (local-window Bernstein coupling inequality)}\label{app:proof:thm:local-concentration}
We now prove Theorem \ref{thm:local-concentration} by combining the previously proven propositions and lemmata.

\begin{proof}
$\Xi$ is compact and $\xi\mapsto Z_\xi(N)$ is Lipschitz for every locally finite configuration, therefore, the supremum over $\Xi$ agrees with the supremum over a fixed countable dense subset; the displayed suprema are therefore measurable. We prove the upper tail and apply the same argument to $-Z_\xi$ for the lower tail, and thus combine via a union bound. The following blocking argument is similar to that of \citet[Proof of Proposition~3]{hansen2015}.

Fix the target polynomial exponent $c>0$ from the theorem and set $M=c+d+20.$
Let $c_{\rm coup}$ be the exponent in Lemma~\ref{lem:nonlinearcoupling}. Choose $L_0>A$ so large that
\begin{equation}\label{eq:L0-choice}
c_{\rm coup}L_0\ge M+2,
\end{equation}
and set
\[
x_T=L_0\log T,
\qquad
n_T=\left\lfloor \frac{T}{2x_T}\right\rfloor .
\]
For block indices $b=0,\ldots,n_T-1$, define
\[
B_b^{(0)}=[2bx_T,(2b+1)x_T),
\qquad
B_b^{(1)}=[(2b+1)x_T,(2b+2)x_T).
\]
For every $t\in B_b^{(e)}$, $e\in \{0,1\}$ the random variable $Z_\xi\circ G_t$ depends only on the observed point configuration in
\[I_b^{(e)}=[\inf B_b^{(e)}-A,\sup B_b^{(e)}).\]
As should be clear, $I_b^{(e)}$ has length at most $x_T+A$, and consecutive intervals of the same parity are separated by a gap at least $x_T-A$.

By stationarity, 
\begin{align}\label{eq:bernBlock}
&\Prob_{\theta^\star}\left(\sup_{\xi \in \Xi}\int_0^T\{Z_\xi\circ G_t-\E_{\theta^\star}Z_\xi\}\,\dd t>u\right)\\\label{eq:BernBlock2}
&\quad \leq 2\Prob_{\theta^\star}\left(\sup_{\xi \in \Xi}\sum_{b=0}^{n_T-1}\int_{2bx_T}^{(2b+1)x_T}\{Z_\xi\circ G_t-\E_{\theta^\star}Z_\xi\}\,\dd t>\frac u3\right)\\ \label{eq:Bernblock3}
&\qquad+\Prob_{\theta^\star}\left(\sup_{\xi \in \Xi}\int_{2n_Tx_T}^T\{Z_\xi\circ G_t-\E_{\theta^\star}Z_\xi\}\,\dd t>\frac u3\right),
\end{align}
and we now remove the dependence on the parity since, without loss of generality, we only consider even blocks by stationarity.

We first bound the remainder term. Define the event,
\[\Omega_{\rm slide}=\left\{\sup_{0\le t\le T}\mathfrak C_t\le K_1\log T\right\}.\]
By Lemma~\ref{lem:count-truncation}, there exists a $K_1>0$ so large such that
\begin{equation}\label{eq:slide-event-local}
\sup_{\theta^\star\in\Theta_0}
\Prob_{\theta^\star}(\Omega_{\rm slide}^c)
\le CT^{-M-3}.
\end{equation}
Then, pathwise on $\Omega_{\rm slide}$ the remainder satisfies
\begin{align}\label{eq:pathwiseRemainderBound}
    \sup_{\xi \in \Xi}\frac{1}{T}\int_{2n_Tx_T}^T\{ Z_\xi \circ G_t-\E_{\theta^\star}Z_\xi\}\,\dd t\leq \frac{Cx_T\log^{m}(T)}{T}\leq \frac{C\log^{m+1}(T)}{T},
\end{align}
since  $C$ and $m$ are selected uniformly over $\Xi$.

We now bound the main sum term. First, we rename $B_b=B_b^{(0)}$, and $I_b=I_b^{(0)}$. Without loss of generality we can ignore the odd blocks since we have removed their explicit influence via a union bound in the display \eqref{eq:BernBlock2} and \eqref{eq:Bernblock3}.

Then, by Lemma~\ref{lem:nonlinearcoupling} we introduce the sequence of independent Hawkes processes, so that 
\begin{multline*}
2\Prob_{\theta^\star}\left(\sup_{\xi \in \Xi}\sum_{b=0}^{n_T-1}\int_{2bx_T}^{(2b+1)x_T}\{Z_\xi\circ G_t-\E_{\theta^\star}Z_\xi\}\,\dd t>\frac{u}{3}\right)\\
\leq 2\Prob_{\theta^\star}\left(\sup_{\xi \in \Xi}\sum_{b=0}^{n_T-1}\int_{2bx_T}^{(2b+1)x_T}\{\tilde Z_\xi\circ G_t-\E_{\theta^\star}Z_\xi\}\,\dd t>\frac{u}{3}\right)+\delta_T
\end{multline*}
where each $\tilde Z$ is independent, and $\delta_T$ is the coupling error. In more detail, as per Lemma \ref{lem:nonlinearcoupling}, let
\[\mathbb P_{\theta^\star}^{\rm dep}=\calL_{\theta^\star}\left( N_\theta^{(0)}|_{I_0},\ldots,N_\theta^{(n_T-1)}|_{I_{n_T-1}}\right), \qquad \mathbb P_{\theta^\star}^{\rm ind}=\bigotimes_{b=0}^{n_T-1}\calL_{\theta^\star}\left(N_\theta^{(b)}|_{I_b}\right),\]
which is the law of independent stationary restrictions
$N^{(b)}_\theta|_{I_b}$, $b=0,\ldots,n_T-1$. Since
$n_T(1+x_T)\leq CT$ for all large $T$,
\begin{align}
\delta_T&:=\sup_{\theta^\star\in\Theta_0}d_{\rm TV}\big(\mathbb P_{\theta^\star}^{\rm dep},\mathbb P_{\theta^\star}^{\rm ind}\big) \leq C n_Te^{-c_{\rm coup}(x_T-A)}\leq C T e^{c_{\rm coup}A}T^{-c_{\rm coup}L_0}\leq C T^{-M-1},
\label{eq:local-coupling-event}
\end{align}
where the last inequality uses \eqref{eq:L0-choice}.

We now fix $\xi\in \Xi$. Let $N^{(b)}_\theta$ be the independent stationary copies used, and set
\[F_{b,\xi}=\int_{2bx_T}^{(2b+1)x_T}\{\tilde Z_\xi\circ G_t-\E_{\theta^\star}Z_\xi\}\,\dd t.\]
By stationarity, and integrability, $F_{b,\xi}$ are centred so that $\E_{\theta^\star}[F_{b,\xi}]=0.$

All probabilities in the independent-block argument below are taken under
$\mathbb P_{\theta^\star}^{\rm ind}$. Under this product law, for fixed $\xi$, the variables $F_{b,\xi}$, $b=0,\ldots,n_T-1$, are
independent and centred, because each $N^{(b)}$ is a stationary copy under
$\Prob_{\theta^\star}$. Define
\[
R_b=\sup_{t\in B_b}N^{(b)}([t-A,t)),
\qquad
\Omega_{\rm blk}=\bigcap_{b<n_T}\{R_b\le R_0\log(T)\}.
\]
The event \(\Omega_{\rm blk}\) is the analogue of Hansen et al.'s logarithmic local-count truncation event $\Omega_q$.
Moreover, the interval $|I_b|\leq x_T+A\leq C\log T$.

Next, we place a deterministic grid of mesh $A/2$ on $I_b$; it contains at most $C\log T$ grid intervals and every sliding window $[t-A,t)$, with $t\in B_b$, intersects at most three of these grid intervals. Therefore, $R_b>r$ implies that some grid interval has more than $r/3$ points. The exponential moment \eqref{eq:count-mgf}, with $L=A/2$, therefore gives constants $C_R,a_R>0$ such that
\begin{equation}\label{eq:block-count-tail-new}
\sup_{\theta^\star\in\Theta_0}
\Prob_{\theta^\star}\left(R_b>r\right)
\leq C_R\log(T)\, e^{-a_Rr},
\qquad r\ge0 .
\end{equation}
Choose $R_0$ so large that $a_RR_0\geq M+4$. Then, for all large $T$,
\begin{align}
\sup_{\theta^\star\in\Theta_0}
\Prob_{\theta^\star}\left((\Omega_{\rm blk})^c\right)
&\le n_T C_R\frac{\log(T)}{T^{a_RR_0}} \leq C T^{1-a_RR_0}\leq C T^{-M-3}.
\label{eq:block-count-event-new}
\end{align}
On the event $\left\{R_b\leq R_0\log(T)\right\}$, the envelope \eqref{eq:local-envelope} and the fixed-window moment \eqref{eq:count-mgf} imply, uniformly in $\xi$,
\begin{align}\label{eq:block-bound-new}
\abs{F_{b,\xi}}\1\left\{R_b\leq R_0\log(T)\right\}\leq Bx_T\{1+\left(R_0\log(T)\right)^m\}+x_T\sup_{\theta\in\Theta}\sup_{\xi\in\Xi}\E_{\theta^\star}\abs{Z_\xi}\leq B_1\log^{m+1}(T)=:b_T,
\end{align}
for some universal constant $B_1<\infty$. Moreover,
\begin{equation}\label{eq:block-var-new}
\Var_\theta\left(F_{b,\xi}\1\{R_b\leq R_0\log(T)\}\right)\leq B_2\log^{2m+2}(T).
\end{equation}
Next, we consider the tail term. Since $\E\left[F_{b,\xi}\right]=0$, truncation creates only the centering error
\[\abs{\E_{\theta^\star}\left[F_{b,\xi}\1\{R_b\le R_0\log(T)\}\right]}
=\abs{\E_{\theta^\star}\left[F_{b,\xi}\1\{R_b>R_0\log(T)\}\right]}.\]
Then, by the tail-integral identity, for any $m>0$,
\[\E_{\theta^\star}\left[R_b^m\1\{R_b>R_0\log(T)\}\right]
=\left(R_0\log(T)\right)^m\Prob_\theta\left(R_b>R_0\log(T)\right)+m\int_{R_0\log(T)}^{\infty}u^{m-1}\Prob_\theta(R_b>u)\,\dd u.\]
Hence, by \eqref{eq:block-count-tail-new},
\begin{align}
\E_{\theta^\star}\left[R_b^m\1\{R_b>R_0\log(T)\}\right]
&\leq C_R\log(T)
\left[
\left(R_0\log(T)\right)^m e^{-a_RR_0\log(T)}
+
m\int_{R_0\log(T)}^{\infty}u^{m-1}e^{-a_Ru}\,\dd u
\right] \notag\\
&\leq
C\log(T)\,(1+\left(R_0\log(T)\right)^m)e^{-a_RR_0\log(T)}.
\label{eq:block-tail-integral-poly}
\end{align}
The same bound is trivial for $m=0$. Therefore, uniformly over $\xi\in\Xi$, from equation \eqref{eq:local-envelope},
\begin{align}
\E\left[\abs{F_{b,\xi}}\1\{R_b>R_0\log(T)\}\right]
&\leq
C x_T
\E\big[
\{1+R_b^m\}\1\{R_b>R_0\log(T)\}
\big] \notag\\
&\le
C\log^{m+2}(T)\,e^{-a_RR_0\log T}
\le
C T^{-M-3},
\label{eq:centering-error-new}
\end{align}
where the last inequality follows from the choice of $R_0$, after potentially increasing $C$. Hence, 
\[\sup_{\xi \in \Xi}\sum_{b=0}^{n_T-1}\E_{\theta^\star}\left[\abs{F_{b,\xi}}\1\{R_b>R_0\log(T)\}\right]\leq CT^{-M-2}.\]

We now apply this truncation to the sum of independent Hawkes processes, and define 
\[\bar F_{b,\xi}=F_{b,\xi}\1\{R_b\leq R_0\log(T)\}-\E\left[F_{b,\xi}\1\{R_b\leq R_0\log(T)\}\right].\]
Therefore,
\begin{align}\label{eq:truncapplied0}
    \Prob_{\theta^\star}\left(\sum_{b=0}^{n_T-1}F_{b,\xi}>\frac u3\right)&\leq \Prob_{\theta^\star}\left(\left\{\sum_{b=0}^{n_T-1}F_{b,\xi}\1\{R_b\leq R_0\log(T)\}>\frac u3\right\} \cap \Omega_{\rm blk}\right)+CT^{-M-3}\\ \label{eq:truncapplied}
    &\leq \Prob_{\theta^\star}\left(\left\{\sum_{b=0}^{n_T-1}\bar F_{b,\xi}>\frac u6\right\} \cap \Omega_{\rm blk}\right)+CT^{-M-3}+\\
    &+\Prob_{\theta^\star}\left(\left\{\sup_{\xi \in \Xi}\sum_{b=0}^{n_T-1}\E_{\theta^\star}\left[\abs{F_{b,\xi}}\1\{R_b>R_0\log(T)\}\right]>\frac{u}{6}\right\}\cap \Omega_{\rm blk}\right)\\\label{eq:truncappliedend}
    &\leq \Prob_{\theta^\star}\left(\left\{\sum_{b=0}^{n_T-1}\bar F_{b,\xi}>\frac u6\right\} \cap \Omega_{\rm blk}\right)+CT^{-M-3}+\1\left\{CT^{-M-2}>\frac{u}{6}\right\}
\end{align}

The first term of equation \eqref{eq:truncapplied} can be bounded via Bernstein's inequality \citep[e.g.][Proposition 2.14]{wainwright2019} for
independent centred bounded block variables, as in \citet[Proposition~3]{hansen2015}, we also extend it to the supremum via a standard $\epsilon$-net argument.

Let $\calN_T$ be a $T^{-1/2}$-net of $\Xi$ with $\abs{\calN_T}\le C_\Xi T^{d/2}$. For a fixed $\eta\in\calN_T$, we apply Bernstein's inequality to the independent centred variables $\bar F_{\eta,b}$ on the event $\Omega_{\rm blk}$. Doing so yields,
\begin{equation}\label{eq:bernstein-block-explicit}
\Prob_{\theta^\star}\left(
\left\{\left|\sum_{b=0}^{n_T-1}\bar F_{b,\eta}\right|>\frac{u}{6}\right\}
\cap \Omega_{\rm blk}\right)
\leq
2\exp\left\{-\frac{u^2/36}{2\left(n_TB_2\log^{2m+2}(T)+2b_Tu/18\right)}\right\}.
\end{equation}
as $\abs{\bar F_{b,\eta}}\leq 2b_T$. Since $n_T\le T/(2L_0\log(T))$,
\begin{equation}\label{eq:variance-budget}
n_TB_2\log^{2m+2}(T)\leq V_0 T\log^{2m+1}(T)
\end{equation}
for some constant $V_0<\infty$. Next, choose an integer $q_0$ such that
\begin{equation}\label{eq:q0-choice}
2q_0-(2m+1)\ge1;
\end{equation}
for instance, $q_0=\lceil m+1\rceil$ suffices. Pick $u=K\sqrt T\log^{q_0}(T)$. Thus, for all sufficiently large $T$,
\[
2\left(n_TB_2\log^{2m+2}(T)+2b_Tu/18\right)\leq 4V_0T\log^{2m+1}(T).
\]
Combining this bound with \eqref{eq:q0-choice} yields
\[\frac{u^2}{2\left(n_TB_2\log^{2m+2}(T)+2b_Tu/3\right)}\geq\frac{K^2}{4V_0}\log^{2q_0-(2m+1)}T\geq\frac{K^2}{4V_0}\log(T).\]
Choose $K$ so large that
\begin{equation*}
\frac{K^2}{4V_0}\ge M+d/2+4.
\end{equation*}
Then \eqref{eq:bernstein-block-explicit} gives
\[\sup_{\theta^\star\in \Theta_0}\Prob_{\theta^\star}\left(\left\{\left|\sum_{b=0}^{n_T-1}\bar F_{b,\eta}\right|>K\sqrt T\log^{q_0}(T)\right\}\cap \Omega_{\rm blk}\right)\le 2T^{-M-d/2-4}.\]
A union bound over $\calN_T$ yields 
\begin{equation*}
\sup_{\theta^\star\in \Theta_0}\Prob_{\theta^\star}\left(\left\{\max_{\eta \in \calN_T}\left|\sum_{b=0}^{n_T-1}\bar F_{b,\eta}\right|>K\sqrt T\log^{q_0}(T)\right\}\cap \Omega_{\rm blk}\right)\leq CT^{-M-3},
\end{equation*}

On the event $\Omega_{\rm blk}$, if $\pi(\xi)\in\calN_T$ is nearest to $\xi$, then the pathwise Lipschitz envelope, equation \eqref{eq:local-lipschitz-envelope} gives
\begin{multline}\label{eq:discretization-path-new}
\sum_{b=0}^{n_T-1}\frac1T\int_{2bx_T}^{(2b+1)x_T}
\abs{\tilde Z_\xi\circ G_t-\tilde Z_{\pi(\xi)}\circ G_t}\1\{R_b\leq R_0\log(T)\}\diff t\\
\leq
\frac{B\{1+(R_0\log(T))^m\}}{\sqrt T}.
\end{multline}
With the expectation similarly satisfying
\begin{equation}\label{eq:discretization-mean-new}
\abs{\E_{\theta^\star}\left[(Z_\xi-Z_{\pi(\xi)})\1\{R_b\leq R_0\log(T)\}\right]}
\leq \frac{B\E_{\theta^\star}\left[(1+N[-A,0))^m\right]}{\sqrt T}
\leq CT^{-1/2},
\end{equation}
where the last bound follows the fact that $N([-A,0))$ has moments of all order. Therefore,
\begin{align}\label{eq:bernsteinbound1}
&\Prob_{\theta^\star}\left(\left\{\sup_{\xi \in\Xi}
\left|\sum_{b=0}^{n_T-1}\bar F_{b,\xi}\right|>K'\sqrt T\log^{q_0'}(T)\right\}
\cap \Omega_{\rm blk}\right)\\
&\leq \Prob_{\theta^\star}\left(\left\{\sup_{\eta \in \calN_T}
\frac 1T\left|\sum_{b=0}^{n_T-1}\bar F_{b,\eta}\right|>K'\frac{\log^{q_0'}(T)}{\sqrt T}-\Delta_T
\right\}\cap \Omega_{\rm blk}\right)\\ \label{eq:bernsteinboundend}
&\leq \Prob_{\theta^\star}\left(\left\{\max_{\eta \in \calN_T}
\frac 1T\left|\sum_{b=0}^{n_T-1}\bar F_{b,\eta}\right|>K\frac{\log^{q_0}(T)}{\sqrt T}
\right\}\cap \Omega_{\rm blk}\right)\leq CT^{-M-3}
\end{align}
where $K'$ and $q_0'$ are selected so that $\Delta_T$, which is the pathwise Lipschitz envelope from adding the right hand sides of equations \eqref{eq:discretization-path-new} and \eqref{eq:discretization-mean-new}, ensures that 
\[K'\frac{\log^{q_0'}}{\sqrt{T}}-\Delta_T\geq K\frac{\log^{q_0}(T)}{\sqrt{T}}.\]

We now combine all of the estimates to conclude the final bound. In particular,
\begin{align}\nonumber
    &\Prob_{\theta^\star}\left(\sup_{\xi \in \Xi}\frac 1T\int_0^T\{Z_\xi\circ G_t-\E_{\theta^\star}Z_\xi\}\,\dd t>K'\frac{\log^{q_0'}(T)}{\sqrt{T}}\right)\\\label{eq:finalbound1}
    &\leq 2\Prob_{\theta^\star}\left(\sup_{\xi \in \Xi}\frac 1T\sum_{b=0}^{n_T-1}\int_{2bx_T}^{(2b+1)x_T}\{\tilde Z_\xi\circ G_t-\E_{\theta^\star}Z_\xi\}\,\dd t>K'\frac{\log^{q_0'}(T)}{\sqrt{T}}\right)+CT^{-M-3}\\ \nonumber
    &\qquad+\1\left\{\frac{C\log^{m+1}(T)}{T}>K'\frac{\log^{q_0'}(T)}{\sqrt{T}}\right\}\\ \label{eq:finalbound3}
    &\leq 2\Prob_{\theta^\star}\left(\left\{\sup_{\xi \in \Xi}\frac{1}{T}\sum_{b=0}^{n_T-1}\bar F_{b,\xi} >K'\frac{\log^{q_0'}(T)}{\sqrt{T}}\right\}\cap \Omega_{\rm blk}\right)+CT^{-M-3}+\1\left\{CT^{-M-2}>\frac{K}{6}\sqrt{T}\log^{q_0}(T)\right\}\\ \label{eq:finalbound4}
    &\leq CT^{-M-3} \leq CT^{-c}
\end{align}
Inequality \eqref{eq:finalbound1} follows from equations \eqref{eq:bernBlock}, \eqref{eq:slide-event-local} and \eqref{eq:pathwiseRemainderBound}. Equation \eqref{eq:finalbound3} follows from the fact that the indicator is eventually always 0 and the definition of $F_{b,\xi}$, and equations \eqref{eq:truncapplied0} to \eqref{eq:truncappliedend}. Finally, equation \eqref{eq:finalbound4} follows from equations \eqref{eq:bernsteinbound1} to \eqref{eq:bernsteinboundend}.

The constants $C,M$ are uniform over $\Theta$ so that a supremum can be taken. Furthermore, the lower tail follows by applying the same argument to
$-Z_\xi$, and the proof concludes by union bounding over the two events.
\end{proof}

We conclude this section by proving a brief proposition about the population finite-window continuity input used in the generic GMM and likelihood-localisation arguments. The argument is standard and follows easily from the previously demonstrated probabilistic tools.

\begin{proposition}[Finite-window population continuity]\label{prop:finite-window-continuity}
Let $K\subset\Theta$ be compact, let $L<\infty$, and let $\Xi\subset\R^d$ be compact.  Suppose $Z_\xi$ depends only on the multivariate restriction $N_\theta|_{[-L,0]}$, satisfies a polynomial finite-window count envelope
\[
    |Z_\xi|\le C\{1+N_\theta([-L,0])^m\},
    \qquad \xi\in\Xi,
\]
and is Lipschitz in $\xi$ with a polynomial count envelope: for some $C',m'<\infty$,
\[
    |Z_\xi-Z_\eta|
    \le C'\{1+N([-L,0])^{m'}\}\norm{\xi-\eta}_2,
    \qquad \xi,\eta\in\Xi .
\]
Then $(\xi,\theta)\mapsto \E_\theta\left[Z_\xi\right]$ is continuous on $\Xi\times K$.
\end{proposition}

\begin{proof}
Write $\mathfrak C_0=N_\theta([-L,0])$. Fix $(\xi_n,\theta_n)\to(\xi,\theta)$ in $\Xi\times K$. By the triangle and Jensen's inequalities,
\begin{align*}
    \left|\E_{\theta_n}[Z_{\xi_n}]-\E_{\theta}[Z_{\xi}]\right|
    &\leq
    \E_{\theta_n}\!\left[\left|Z_{\xi_n}-Z_{\xi_n,M}\right|\right]
    +\left|\E_{\theta_n}[Z_{\xi_n,M}]-\E_{\theta}[Z_{\xi_n,M}]\right| \\
    &\quad
    +\E_\theta\!\left[|Z_{\xi_n,M}-Z_{\xi,M}|\right]
    +\E_{\theta}\!\left[|Z_{\xi,M}-Z_\xi|\right],
\end{align*}
where $Z_{\xi,M}=Z_\xi\1\{\mathfrak C_0\le M\}$.

We first bound $\left|\E_{\theta_n}[Z_{\xi_n,M}]-\E_{\theta}[Z_{\xi_n,M}]\right|$. To do so fix $B>L+A$ and let $\Prob_{\zeta,B}^{0,[-B,0]}$ be the law on $[-B,0]$ of the process obtained by starting at time $-B$ with empty pre-$-B$ history. Let $\Prob_{\zeta,B}^{0,[-L,0]}$ denote its restriction to $[-L,0]$, and write $\E_{\zeta,B}$ for expectation with respect to this finite-start law whenever the integrand depends only on $N_\zeta|_{[-L,0]}$.

By Lemmata \ref{lem:nonstationary-return-to-stationarity} and \ref{lem:finite-expected-local-history-nonstationary}, uniformly over $\zeta\in K$,
\[
    \Delta_B
    :=
    \sup_{\zeta\in K}
    d_{\TV}\!\left(
        \Prob_{\zeta}^{[-L,0]},
        \Prob_{\zeta,B}^{0,[-L,0]}
    \right)
    \longrightarrow 0
    \qquad\text{as }B\to\infty .
\]
Hence, since $Z_{\xi_n,M}$ is bounded by $C(1+M^m)$,
\begin{align*}
    &\left|\E_{\theta_n}[Z_{\xi_n,M}]-\E_{\theta}[Z_{\xi_n,M}]\right| \\
    &\leq
    \left|\E_{\theta_n}[Z_{\xi_n,M}]-\E_{\theta_n,B}[Z_{\xi_n,M}]\right|
    +\left|\E_{\theta_n,B}[Z_{\xi_n,M}]-\E_{\theta,B}[Z_{\xi_n,M}]\right| \\
    &\quad
    +\left|\E_{\theta,B}[Z_{\xi_n,M}]-\E_{\theta}[Z_{\xi_n,M}]\right| \\
    &\leq
    4C(1+M^m)\Delta_B
    +2C(1+M^m)
    d_{\TV}\!\left(
        \Prob_{\theta_n,B}^{0,[-L,0]},
        \Prob_{\theta,B}^{0,[-L,0]}
    \right).
\end{align*}

For fixed $B$, the finite-start laws are continuous in total variation as a function of the parameter. In short, this is due to Scheff\'e's lemma. In more detail, take $\Prob_0^B$, the law of $D$ independent unit-rate Poisson processes on $[-B,0]$, as the reference measure. For each $\zeta\in K$, the finite-start law on $[-B,0]$ has likelihood density
\[
L_B(\zeta;N)
=
\prod_{i=1}^D
\exp\left\{
\int_{-B}^0 \log\lambda_i^{B,0}(t;\zeta)\,\dd N_i(t)
-
\int_{-B}^0 \{\lambda_i^{B,0}(t;\zeta)-1\}\,\dd t
\right\}
\]
with respect to $\Prob_0^B$, where $\lambda_i^{B,0}(t;\zeta)$ is the predictable intensity functional computed from the empty pre-$(-B)$ history and the realized path $N|_{[-B,t)}$. Such a formula is standard \citep[e.g.][Section 7]{daley2003} or \cite{Leskel2024InformationDA} and references cited therein.

For $\Prob_0^B$-almost every path there are finitely many points on $[-B,0]$; by Assumptions~\ref{ass:parameter} and~\ref{ass:kernels}, the displayed likelihood is then continuous in $\zeta$. Since $\int L_B(\zeta;N)\,d\Prob_0^B(N)=1$ for every $\zeta$, Scheff\'e's lemma yields convergence in total variation on $[-B,0]$, and therefore also after restriction to $[-L,0]$:
\begin{equation}\label{eq:finite-start-parameter-tv}
    d_{\TV}\left(
        \Prob_{\theta_n,B}^{0,[-L,0]},
        \Prob_{\theta,B}^{0,[-L,0]}
    \right) \to 0 .
\end{equation}

We next bound the remaining terms. By the fact that $Z_\xi$ is Lipschitz in $\xi$ with a polynomial count envelope, so too is $Z_{\xi,M}$. Thus,
\[
\E_\theta[|Z_{\xi_n,M}-Z_{\xi,M}|]
\leq
C'\norm{\xi_n-\xi}_2
\E_\theta[1+\mathfrak C_0^{m'}]
=
O(\norm{\xi_n-\xi}_2)
\]
by the exponential fixed-window count moment \eqref{eq:count-mgf}.

For the two truncation terms, by H\"older's inequality,
\begin{align*}
    &\E_{\theta_n}\!\left[\left|Z_{\xi_n}-Z_{\xi_n,M}\right|\right]
    +\E_{\theta}\!\left[\left|Z_{\xi}-Z_{\xi,M}\right|\right] \\
    &\leq
    C\sqrt{
        \E_{\theta_n}[(1+\mathfrak C_0^m)^2]
        \Prob_{\theta_n}\left(\mathfrak C_0>M\right)
    }
    +
    C\sqrt{
        \E_{\theta}[(1+\mathfrak C_0^m)^2]
        \Prob_{\theta}\left(\mathfrak C_0>M\right)
    } .
\end{align*}
By \eqref{eq:count-mgf}, the moments of $\mathfrak C_0$ are uniformly bounded over $K$, and the tails $\Prob_\zeta(\mathfrak C_0>M)$ converge to zero uniformly over $\zeta\in K$. Hence, for every $\epsilon>0$, we can choose $M$ large enough such that
\[
\sup_{n\in\N}
\left\{
\E_{\theta_n}\!\left[\left|Z_{\xi_n}-Z_{\xi_n,M}\right|\right]
+
\E_{\theta}\!\left[\left|Z_{\xi}-Z_{\xi,M}\right|\right]
\right\}
<\epsilon .
\]
For this fixed $M$, choose $B$ sufficiently large that
\[
    4C(1+M^m)\Delta_B<\epsilon .
\]
Finally, with $M$ and $B$ fixed, \eqref{eq:finite-start-parameter-tv} and $\xi_n\to\xi$ imply that, for all sufficiently large $n$,
\[
    2C(1+M^m)
    d_{\TV}\!\left(
        \Prob_{\theta_n,B}^{0,[-L,0]},
        \Prob_{\theta,B}^{0,[-L,0]}
    \right)
    +
    \E_\theta[|Z_{\xi_n,M}-Z_{\xi,M}|]
    <\epsilon .
\]
Combining the preceding displays gives
\[
    \limsup_{n\to\infty}
    \left|\E_{\theta_n}[Z_{\xi_n}]-\E_\theta[Z_\xi]\right|
    \leq 3\epsilon .
\]
Since $\epsilon>0$ was arbitrary, the claim follows.
\end{proof}

\section{Common compact-window calculus and martingale tools}\label{app:common-tools}

The estimates in the generic estimating-map arguments and in the likelihood appendix use the same two elementary inputs: polynomial envelopes for local intensity derivatives, and a predictable-truncation version of the Dzhaparidze--van Zanten Bernstein inequality for local martingales \citep{dzhaparidze2001}. Under the assumptions given in Subsection \ref{sec:assumptions} we demonstrate further polynomial count envelopes of the intensity, and its log derivatives. Throughout Section \ref{app:gmm} we can refer back to the generic compact-window machinery rather than reproving the same martingale-net arguments.

For a multi-index $a$ in the coordinates of $\theta$ with $|a|\le3$, define the generic local kernel derivative convolution
\begin{equation}\label{eq:Yell-def}
Y_{ij}^{(a)}(t;\theta)
=
\int_{[t-A,t)}
\partial_\theta^a h_{\theta,ij}(t-s)\,\dd N_j(s),
\end{equation}
with the convention that inactive coordinates have derivative zero.  By compact support and Assumption~\ref{ass:kernels},
\begin{equation}\label{eq:Yell-bound}
\sup_{\theta\in K_\Theta}
\abs{Y_{ij}^{(a)}(t;\theta)}
\le C \mathfrak C_t,
\qquad |a|\le3 .
\end{equation}

\begin{lemma}[Polynomial derivative envelopes]\label{lem:derivative-envelope}
Under Assumptions~\ref{ass:parameter}--\ref{ass:kernels}, for each $r=0,1,2,3$ there are constants $C_r<\infty$ and $m_r\in\N$ such that, for all $i$, all $t$, and all $\theta\in K_\Theta$,
\begin{equation}\label{eq:deriv-envelope-log}
\norm{\nabla_\theta^r\log\lambda_i(t;\theta)}_2
\le
C_r\{1+\mathfrak C_t^{m_r}\},
\end{equation}
and
\begin{equation}\label{eq:deriv-envelope-lambda}
\norm{\nabla_\theta^r\lambda_i(t;\theta)}_2
\le
C_r\{1+\mathfrak C_t^{m_r}\}.
\end{equation}
\end{lemma}

\begin{proof}
Write
\[
X_i(t;\theta)=\nu_{\theta,i}+
\sum_{j=1}^D\int_{[t-A,t)}h_{\theta,ij}(t-s)\,\dd N_j(s).
\]
For each multi-index $a$ with $|a|\le3$, Assumptions~\ref{ass:parameter} and~\ref{ass:kernels} give
\[
\abs{\partial_\theta^a X_i(t;\theta)}
\le C(1+\mathfrak C_t),
\qquad \theta\in K_\Theta,
\]
by the boundedness of the baseline derivatives and \eqref{eq:Yell-bound}.  The multivariate Fa\`a di Bruno formula \citep[e.g.][]{Hardy_2006} applied to
$\lambda_i(t;\theta)=\phi_{\theta,i}\{X_i(t;\theta)\}$ expresses each $\partial_\theta^a\lambda_i(t;\theta)$, $|a|\le3$, as a finite sum of products of link derivatives evaluated at $X_i(t;\theta)$, baseline derivatives, and the local convolutions $Y_{ij}^{(b)}(t;\theta)$.  The link derivatives have at most polynomial growth in the input coordinate and $|X_i(t;\theta)|\le C(1+\mathfrak C_t)$, so \eqref{eq:deriv-envelope-lambda} follows.

The positive-link assumption gives a uniform lower bound $\lambda_i(t;\theta)\ge\underline\lambda>0$ on $K_\Theta$, and the case $r=0$ above gives $\lambda_i(t;\theta)\le C(1+\mathfrak C_t^{m_0})$.  Derivatives of $\log\lambda_i$ up to order three are finite sums of products of derivatives of $\lambda_i$ divided by powers of $\lambda_i$; the denominator is controlled by $\underline\lambda$.  This proves \eqref{eq:deriv-envelope-log}.

\end{proof}

\begin{lemma}[Predictability, endpoints, and bracket bookkeeping]\label{lem:predictable-jumps}
Under Assumptions~\ref{ass:parameter}--\ref{ass:stability}, the following standard point-process facts hold.
\begin{enumerate}[label=(\roman*)]
\item The local count process $\mathfrak C_t=N([t-A,t))$ is predictable. More generally, for every deterministic continuous $f$ on $[0,A]$ and every component $j$,
\[Y_f(t)=\int_{[t-A,t)} f(t-s)\,\dd N_j(s)\]
is predictable, as a left-continuous adapted local-window functional under the half-open predictable convention \citep[e.g.][Def. 2.4.1, Prop. 2.4.1 and Lem. 2.4.1]{bjork_pointprocess_martingale}. 

\item Endpoint and no-coincidence convention. Assume that each component counting process \(N_k\) has a continuous predictable compensator \(\Lambda_k\), and that the marked point process has a simple ground process (equivalently, is completely simple). Then each \(N_k\) is quasi-left continuous, so \(\Delta N_k(\sigma)=0\) almost surely for every finite predictable stopping time \(\sigma\) \citep[Cor.~II.1.19]{jacodshiryaev2003}; see also \citet[Sec.~2]{Bandini2024}. In particular, deterministic times carry no points. If \(T\) is an event time and \(A>0\) is fixed, then \(T+A\) is predictable, being announced by \(T+A(1-1/n)\); hence no component jumps at \(T+A\). By local finiteness, this excludes exact \(A\)-lag pairs on compact intervals. Complete simplicity of the ground process excludes simultaneous jumps of distinct components \citep[e.g.]{schoenberg2006nonsimple}. Consequently, in stochastic-integral estimating equations, the history window \([t-A,t)\) may be replaced by \((t-A,t)\) almost surely.

\item By the standard jump formula for stochastic integrals with predictable integrands \citep[e.g.][Lem. 5.93]{vandervaart_stochint}, if $H_i$ is a predictable scalar integrand and $M_i^{\theta^\star}=N_i-\int_0^\cdot\lambda_i(s;\theta^\star)\diff s$, then
\begin{equation*}
\Delta\left(\int_0^\cdot H_i(t)\diff M_i^{\theta^\star}(t)\right)_s=H_i(s)\Delta M_i^{\theta^\star}(s)=H_i(s)\Delta N_i(s).
\end{equation*}
Consequently, for a finite sum $\sum_i\int H_i\diff M_i^{\theta^\star}$, each jump is $\sum_i H_i(s)\Delta N_i(s)$.

\item  For predictable scalar integrands $H_i$, the predictable variation is of the form
\begin{equation*}
\left\langle \sum_{i=1}^D\int_0^\cdot H_i(t)\diff M_i^{\theta^\star}(t) \right\rangle_T
=
\sum_{i=1}^D\int_0^T H_i(t)^2\lambda_i(t;\theta^\star)\diff t .
\end{equation*}
\citep[pg. 233][and references cited]{WOOD1999231}.
\end{enumerate}
\end{lemma}

\begin{lemma}[Dzhaparidze--van Zanten Bernstein inequality; 
{\citealp[Theorem~3.3 and Corollary~3.4(iii)]{dzhaparidze2001}}]
\label{lem:dzv-bernstein}
Let $M$ be a cadlag locally square-integrable martingale with $M_0=0$ and
predictable quadratic variation $\langle M\rangle_T$. Suppose that all jumps
of $M$ on $[0,T]$ are bounded in absolute value by $b$. Then, for all $x,v>0$,
\[
\Prob\left(M_T\ge x,\ \langle M\rangle_T\le v\right)
\le
\exp\left\{-\frac{x^2}{2(v+bx/3)}\right\}.
\]
The same bound applies to $-M_T$, and hence
\[
\Prob\left(|M_T|\ge x,\ \langle M\rangle_T\le v\right)
\le
2\exp\left\{-\frac{x^2}{2(v+bx/3)}\right\}.
\]
\end{lemma}

\begin{lemma}[Predictable truncation Dzhaparidze--van Zanten device]\label{lem:stopped-dzv}
Let $M_t=\sum_i\int_0^t H_{i,s}\,\dd M_{i,s}^0$ be a scalar martingale integral with respect to finitely many point-process martingales, where the $H_i$ are predictable. Let $H_i^{(R)}$ be predictable truncated integrands such that the martingale integral $M_t^{(R)}=\sum_i\int_0^t H_{i,s}^{(R)}\,\dd M_{i,s}^0$ has jumps bounded by $b$ almost surely. Furthermore, suppose that on an event $\calA_R$ one has $H_{i,s}^{(R)}=H_{i,s}$ for all $i$ and all $s\in[0,T]$. Then, for all $x,v>0$,
\begin{equation*}
\Prob\left(|M_T|>x\right)
\le
\Prob\left(\calA_R^c\right)
+
\Prob\left(\langle M\rangle_T>v\right)
+
2\exp\left\{-\frac{x^2}{2(v+bx/3)}\right\}.
\end{equation*}
\end{lemma}

\begin{proof}
On $\calA_R$, $M_T=M_T^{(R)}$. Hence
\[
\{|M_T|>x\}
\subset
\calA_R^c
\cup
\{\langle M\rangle_T>v\}
\cup
\{|M_T^{(R)}|>x,\ \langle M^{(R)}\rangle_T\le v\}.
\]
Apply Lemma~\ref{lem:dzv-bernstein} to $M^{(R)}$ and to $-M^{(R)}$.
\end{proof}

\section{Generic compact-window estimating-map and GMM proofs}\label{app:gmm}

In this section we prove the main results of Section \ref{sec:results}. After demonstrating that our estimators are well defined, we demonstrate high probability uniform concentration results for the empirical risk function, and its Hessian, around their expectations in Proposition \ref{prop:generic-controls} which is done by repeated applications of Theorem \ref{thm:local-concentration} and Lemma \ref{lem:stopped-dzv}. From there pathwise Taylor-expansion arguments, and the martingale central limit theorem, yield the desired results. Before turning to two-step weighting, we also record the least-squares contrast consequence used in Section~\ref{sec:simulation}. The results for the two-step estimator then follow fairly easily from the deterministically weighted estimator. In particular, after proving the plug-in estimate of the covariance $\Omega_H$ is eventually uniformly invertible with high probability, in Lemma \ref{lem:plugin-inverse-polynomial}, the previous arguments follow easily.

\begin{proposition}[Existence and measurable selection of GMM minimizers]\label{prop:gmm-measurable-selection}
Under Assumptions~\ref{ass:parameter}, \ref{ass:kernels}, and~\ref{ass:weights}, and for any fixed deterministic symmetric positive-definite matrix $W$, define
\[
    Q_{T,W}^H(\vartheta)=\norm{\frac{1}{T}\Psi_T^H(\vartheta)}_W^2 .
\]
On the event $N([-A,T])<\infty$, the map $\vartheta\mapsto Q_{T,W}^H(\vartheta)$ is continuous on $\Theta$.  Hence the argmin set in \eqref{eq:gmm-estimator} is nonempty and compact, and it admits a measurable minimizer.
\end{proposition}

\begin{proof}[Proof of Proposition~\ref{prop:gmm-measurable-selection}]
On $N([-A,T])<\infty$, every shifted window contains finitely many points.  Assumption~\ref{ass:weights} gives pathwise continuity of $\vartheta\mapsto H(t;\vartheta)$ and the derivative-envelope bound gives a polynomial local-count dominating function, e.g.\ \eqref{eq:H-envelope}. Hence, by the dominated convergence theorem for the compensator term and by finiteness of the event sum for the $\dd N$ term, $\vartheta\mapsto\Psi_T^H(\vartheta)$ is continuous.  Therefore $Q_{T,W}^H$ is continuous on the compact set $\Theta$ and has a nonempty compact argmin set.

The map from the observed configuration to $Q_{T,W}^H(\cdot)$ is measurable as a map into $C(\Theta)$ with the supremum norm; this follows from joint measurability in the configuration and continuity in $\vartheta$, first on a countable dense subset of $\Theta$ and then by uniform continuity on compact sets. Thus the argmin correspondence is measurable. 
\end{proof}

\begin{proposition}[Uniform controls for compact-window estimating maps]\label{prop:generic-controls}
Suppose Assumptions~\ref{ass:parameter}--\ref{ass:stability} and
Assumption~\ref{ass:weights} hold, and let $W$ be a fixed deterministic
symmetric positive-definite $q\times q$ matrix. For every $c>0$ there are constants $C,K<\infty$ and an integer $q_0\ge1$ such that, for all sufficiently large $T$,
\begin{align}
\sup_{\theta^\star\in\Theta_0}
\Prob_{\theta^\star}
\left(
    \norm{\frac{1}{T}\Psi_T^H(\theta^\star)}_W
    >K\sqrt{\frac{\log(T)}{T}}
\right)
&\le CT^{-c},
\label{eq:generic-truth-bound}\\
\sup_{\theta^\star\in\Theta_0}
\Prob_{\theta^\star}
\left(
    \sup_{\vartheta\in\Theta}
    \norm{\frac{1}{T}\Psi_T^H(\vartheta)-g_H(\vartheta,\theta^\star)}_W
    >K\frac{\log^{q_0}(T)}{\sqrt T}
\right)
&\le CT^{-c}.
\label{eq:generic-ulln}
\end{align}
Moreover, for every fixed $r>0$ such that
\[
\Theta_0^{+r}:=\{\vartheta\in\R^p:\dist(\vartheta,\Theta_0)\le r\}
\subset K_\Theta,
\]
the local Jacobian obeys 
\begin{equation}\label{eq:generic-jacobian-control}
\sup_{\theta^\star\in\Theta_0}
\Prob_{\theta^\star}
\left(
    \sup_{\vartheta\in B(\theta^\star,r)}
    \norm{\partial_\vartheta\left\{\frac{1}{T}\Psi_T^H(\vartheta)\right\}+B_H(\vartheta,\theta^\star)}_{\op}
    >K\frac{\log^{q_0}(T)}{\sqrt T}
\right)
\le CT^{-c}.
\end{equation}
\end{proposition}

 \begin{proof}[Proof of Proposition~\ref{prop:generic-controls}]
Write $m_T^H(\vartheta)=\frac{1}{T}\Psi_T^H(\vartheta)$.  By \eqref{eq:generic-decomposition}, $m_T^H(\vartheta)-g_H(\vartheta,\theta^\star)$ is the sum of a martingale average and a centred drift average. Specifically,
\[m_T^H(\vartheta)-g_H(\vartheta,\theta^\star)=\left(\frac 1T\sum_{i=1}^D\int_0^T H_{\cdot i}(t;\vartheta)\dd M_i^{\theta^\star}(t)\right)+\left(\frac 1T\int_0^T H(t;\vartheta)\{\lambda(t;\theta^\star)-\lambda(t;\vartheta)\}\dd t-g_H(\vartheta,\theta^\star)\right). \]

Throughout the proof, all vector and matrix statements are obtained coordinatewise and then converted to Euclidean or operator norms using the fixed dimensions $q$, $p$, and $D$. Moreover, the envelope assumptions are stable under the finite sums, products, and matrix multiplications that appear  below.  Thus, by Assumption~\ref{ass:weights} and Lemma~\ref{lem:derivative-envelope}, every indexed local-window functional that appears in drift terms, predictable variations, or as a derivative of such a term has both a polynomial local-count envelope and a polynomial local-count Lipschitz envelope in the sense of Theorem~\ref{thm:local-concentration}. In particular, multiplication by $\lambda_i(\cdot;\theta_0)$ is harmless because $\lambda_i(t;\theta_0)\le C(1+\mathfrak C_t)$ uniformly on $K_\Theta$. We prove this in full twice to demonstrate the basic argument, and omit it for brevity in the other cases.

Moreover, pathwise differentiability is justifiable as the $\dd N$ terms on $\{N([-A,T])<\infty\}$, are finite sums at event times; differentiation of the compensator terms is justified by the same polynomial envelopes and the uniform exponential moment \eqref{eq:count-mgf} .

First consider the drift term, where the shift-covariant local-window functional is
\[Z_{a,\vartheta,\theta_0}(N)
    =
    e_a^\top H(0;\vartheta)\{\lambda(0;\theta_0)-\lambda(0;\vartheta)\},
    \qquad a=1,\ldots,q,
\]
where $e_a$ is the $q$-vector with a one in the $a$th position and zeroes elsewhere. Moreover, $Z_{a,\vartheta,\theta_0}(N)$ has the required polynomial envelope and Lipschitz envelope, uniformly over $a=1,\ldots,q$ and $(\vartheta,\theta_0)\in\Theta\times\Theta_0$. To see this, use \eqref{eq:H-envelope} and the uniform bound $\lambda_i(t;\theta_0)\le C(1+\mathfrak C_t)$:
\begin{align*}
    \left|e_a^\top H(0;\vartheta)\{\lambda(0;\theta_0)-\lambda(0;\vartheta)\}\right|&\leq \norm{H(0;\vartheta)\lambda(0;\theta_0)}_2+ \norm{H(0;\vartheta)\lambda(0;\vartheta)}_2\\
    &\leq 2C(1+\mathfrak C_0^{m_0})(1+\mathfrak C_0)\leq B_0(1+\mathfrak C_0^{m_0+1})
\end{align*}
where $B_0$ and $m_0$ can be selected uniformly over $\Theta \times \Theta_0$. Moreover, 
\begin{align*}
    &\left|e_a^\top \left(H(0;\vartheta)\{\lambda(0;\theta_0)-\lambda(0;\vartheta)\}-H(0;\eta)\{\lambda(0;\theta_0)-\lambda(0;\eta)\}\right)\right|\\
    &\leq \norm{H(0;\vartheta)\{\lambda(0;\theta_0)-\lambda(0;\vartheta)\}-H(0;\eta)\{\lambda(0;\theta_0)-\lambda(0;\eta)\}}_2\\
    &\leq \norm{H(0;\vartheta)}_{\rm op}\norm{\lambda(0;\vartheta)-\lambda(0;\eta)}_2+\left(\norm{\lambda(0;\theta_0)}_2+\norm{\lambda(0;\eta)}_{2}\right)\norm{H(0;\vartheta)-H(0;\eta)}_{\rm op}\\
    &\leq B_{0,1}(1+\mathfrak C_0^{m_{0,1}})\norm{\vartheta-\eta}_2
\end{align*}
for some finite constants $B_{0,1}$ and $m_{0,1}$ which are independent of $\theta_0,\vartheta,\eta$. Hence, Theorem~\ref{thm:local-concentration}, applied coordinatewise, gives the required $K\log^{q_0}(T)/\sqrt T$ bound for the centred drift average uniformly over $\vartheta$ and $\theta^\star$.

Now consider the mean-zero martingale term. To do so, fix a coordinate $a\in\{1,\ldots,q\}$ and a parameter value $\vartheta$. The scalar martingale
\[
    M_{a,\vartheta}(T)
    =
    \sum_{i=1}^D\int_0^T H_{ai}(t;\vartheta)\,\dd M_i^{\theta^\star}(t)
\]
has the predictable bracket
\[
    V_{a,\vartheta}(T)
    =
    \sum_{i=1}^D\int_0^T H_{ai}(t;\vartheta)^2\lambda_i(t;\theta^\star)\,\dd t .
\]
Moreover, the bracket integrand
\[
    Z^V_{a,\vartheta,\theta_0}(N)
    =
    \sum_{i=1}^D H_{ai}(0;\vartheta)^2\lambda_i(0;\theta^\star)
\]
is an indexed local-window functional with polynomial envelope and polynomial
Lipschitz envelope over $\Theta\times\Theta_0$.  Indeed, by
\eqref{eq:H-envelope} and $\lambda_i(0;\theta^\star)\le C(1+\mathfrak C_0)$,
\[
    |Z^V_{a,\vartheta,\theta^\star}(N)|
    \le B(1+\mathfrak C_0^{m})
\]
uniformly in $(\vartheta,\theta^\star)$.  The product rule, the Lipschitz envelope
for $H$, and the same intensity bound give
\[
    |Z^V_{a,\vartheta,\theta_0}(N)-Z^V_{a,\eta,\theta_0}(N)|
    \le B(1+\mathfrak C_0^{m})\|\vartheta-\eta\|_2 .
\]
Theorem~\ref{thm:local-concentration} therefore gives uniform concentration of
$\frac{1}{T}V_{a,\vartheta}(T)$ around its expectation.  Stationarity and the
fixed-window moment bound yield
\[
    \sup_{\vartheta\in\Theta,\theta^\star\in\Theta_0}
    \E_{\theta^\star}\left[V_{a,\vartheta}(T)\right]
    \le CT .
\]
Consequently,
\begin{equation}\label{eq:generic-martingale-bracket-event}
    \sup_{\vartheta\in\Theta}V_{a,\vartheta}(T)\le C'T
\end{equation}
with probability at least $1-CT^{-c-p-2}$, uniformly in
$\theta^\star\in\Theta_0$, after potentially increasing constants.

We next prove that $\frac{1}{T}M_{a,\vartheta}(T)$ concentrates uniformly around zero.  Let
\[
    \calA_T=\left\{\sup_{0\le t\le T}\mathfrak C_t\le K_0\log(T),
    \quad \frac{1}{T}N([0,T])\le K_0\log(T)\right\},
\]
whose complement has probability at most $CT^{-c-p-2}$ by Lemma~\ref{lem:count-truncation}.  On $\calA_T$, the predictable truncation
$H_{ai}(t;\eta)\mathbf 1\{\mathfrak C_t\le K_0\log(T)\}$ agrees with the original integrand and has jumps bounded by $K\log^{q_0}(T)$.

Let $\calN_T$ be a deterministic $T^{-1/2}$-net of $\Theta$, so that $|\calN_T|\lesssim T^{p/2}$.  For each fixed $\eta\in\calN_T$, Lemma~\ref{lem:stopped-dzv}, together with \eqref{eq:generic-martingale-bracket-event}, gives
\begin{align*}
\Prob_{\theta^\star}\left(
    \left|\frac{1}{T}M_{a,\eta}(T)\right|
    > K\frac{\log^{q_0}(T)}{\sqrt T}
\right)&\leq \Prob_{\theta^\star}(\calA_T^c)+CT^{-c-p-2}+2\exp\left\{\frac{-K^2T\log^{2q_0}(T)}{2(CT+K'\sqrt{T}\log^{2q_0}(T))}\right\}\\
&\le CT^{-c-p-2}
\end{align*}
for $K$ and $q_0$ large enough.  A union bound over the fixed-dimensional net gives
\[
\Prob_{\theta^\star}\left(
\max_{\eta\in\calN_T}\left|\frac{1}{T}M_{a,\eta}(T)\right|
    > K\frac{\log^{q_0}(T)}{\sqrt T}
\right)
\leq |\calN_T| CT^{-c-p-2}\leq CT^{-c}.
\]

It remains to extend this net bound to all $\vartheta\in\Theta$, which we do by stochastic equicontinuity. On the event $\calA_T$, if $\eta\in\calN_T$ and $\norm{\vartheta-\eta}_2\le T^{-1/2}$, the pathwise Lipschitz envelope in Assumption~\ref{ass:weights} gives
\[
    \sup_{0\le t\le T}\max_{i=1}^D
    \abs{H_{ai}(t;\vartheta)-H_{ai}(t;\eta)}
    \le K'\frac{ \log^{m}(T)}{\sqrt T},
\]
then using that $\dd M_i^{\theta^\star}=\dd N_i-\lambda_i(t;\theta^\star)\dd t$ and $\lambda_i(t;\theta^\star)\le C(1+\mathfrak C_t)$ for some $C$ uniform over $\Theta_0$,
\[
\begin{aligned}
\abs{\frac{1}{T}\{M_{a,\vartheta}(T)-M_{a,\eta}(T)\}}
&\le K\frac{\log^{m'}(T)}{\sqrt T}
\left\{\frac 1TN([0,T])+\frac 1T\sum_{i=1}^D\int_0^T\lambda_i(t;\theta^\star)\dd t\right\} \le K'\frac{\log^{m'}(T)}{\sqrt T},
\end{aligned}
\]
for some $m'\geq m$. Therefore, using a standard netting argument
\begin{align}\label{eq:epsilonnetextension}
   \Prob_{\theta^\star}\left( \sup_{\vartheta\in\Theta}\abs{\frac 1T M_{a,\vartheta}(T)}
    > K\frac{\log^{q_0}(T)}{\sqrt T}\right)&\leq\Prob_{\theta^\star}(\mathcal A_T^c)+\\ &+\Prob_{\theta^\star}\left( \max_{\eta\in\calN_T}\abs{\frac 1T M_{a,\eta}(T)}
    +2K'\frac{\log^{m'}(T)}{\sqrt T}> K\frac{\log^{q_0}(T)}{\sqrt T}\right)\\ \nonumber
&\leq \1\left\{2K'\frac{\log^{m'}(T)}{\sqrt T}>K\frac{\log^{q_0}(T)}{2\sqrt T}\right\}+CT^{-c}
\end{align}
for all sufficiently large $T$ where we have selected $q_0>m$ and inflated $K$ slightly further. The proof of equation \eqref{eq:generic-ulln} concludes via a union bound over the (fixed) $q$ coordinates.

We now move to proving equation \eqref{eq:generic-truth-bound}, which follows the same route as the proof of equation \eqref{eq:generic-ulln}. At truth, $\vartheta=\theta^\star$, the drift term in \eqref{eq:generic-decomposition} vanishes. For each coordinate, the scalar martingale
\[
    M^0_a(T)
    =
    \sum_{i=1}^D\int_0^T H_{ai}(t;\theta^\star)\,\dd M_i^{\theta^\star}(t)
\]
has the predictable variation
\[
    \langle M^0_a\rangle_T
    =
    \sum_{i=1}^D\int_0^T H_{ai}(t;\theta^\star)^2\lambda_i(t;\theta^\star)\,\dd t,
\]
by Lemma~\ref{lem:predictable-jumps}. The same bracket concentration and predictable-truncation argument imply that with probability no less than $1-CT^{-c}$, $ \langle M^0_a\rangle_T\leq C'T$ for some constant $C$ which can be taken to be uniform over $\Theta_0$. Thus, using the same truncation event $\calA_T$ and Lemma~\ref{lem:stopped-dzv} implies that 
\begin{align*}
    \Prob_{\theta^\star}\left(|M_a^0(T)|>K\sqrt{T\log(T)}\right)\leq CT^{-c}+2\exp\left\{-\frac{K^2T\log(T)}{2(C'T+K_0'K\sqrt{T}\log^{q_0+1/2}T/3)}\right\}
\end{align*}
which can be made less than $2CT^{-c}$ for all $T$ sufficiently large for sufficiently large (fixed) $K$. A union bound over $a=1,\ldots,q$ then completes the proof of equation \eqref{eq:generic-truth-bound}.

To prove equation \eqref{eq:generic-jacobian-control}, we first fix $r>0$ such that $\Theta_0^{+r}\subset K_\Theta$ and differentiate \eqref{eq:generic-decomposition} on the local ball.  The preceding pathwise and dominated-differentiation justification gives, for each $\vartheta\in B(\theta^\star,r)$,
\begin{equation}\label{eq:generic-jacobian-decomposition}
\begin{aligned}
\partial_\vartheta m_T^H(\vartheta)
&=
\frac 1T\sum_{i=1}^D\int_0^T \partial_\vartheta H_{\cdot i}(t;\vartheta)\,\dd M_i^{\theta^\star}(t) \\
&\quad+
\frac 1T\int_0^T
\partial_\vartheta H(t;\vartheta)
\{\lambda(t;\theta^\star)-\lambda(t;\vartheta)\}\,\dd t
-
\frac 1T\int_0^T H(t;\vartheta)D_\theta(t;\vartheta)\,\dd t .
\end{aligned}
\end{equation}
The expectation of the two drift averages is $-B_H(\vartheta,\theta^\star)$ by \eqref{eq:BH-two-parameter}.  We now control the three terms in \eqref{eq:generic-jacobian-decomposition} entry by entry, uniformly over the local ball.  

Fix an entry $(a,b)\in\{1,\ldots,q\}\times\{1,\ldots,p\}$.  The two centred drift integrands are the local-window functionals
\[
Z^{(1)}_{ab,\vartheta,\theta^\star}(N)
=
\sum_{i=1}^D
\partial_{\vartheta_b}H_{ai}(0;\vartheta)
\{\lambda_i(0;\theta^\star)-\lambda_i(0;\vartheta)\},
\]
and
\[
Z^{(2)}_{ab,\vartheta}(N)
=
\sum_{i=1}^D
H_{ai}(0;\vartheta)\partial_{\vartheta_b}\lambda_i(0;\vartheta).
\]
The index set $\Theta_0^{+r}\times\Theta_0$ is compact and finite dimensional. Verification of the polynomial envelope is immediate from Assumption~\ref{ass:weights}, Lemma~\ref{lem:derivative-envelope}, and the lower bound on the intensities. The Lipschitz envelopes follow from \eqref{eq:H-lipschitz-envelope}, the derivative envelopes for $\lambda$, and the product rule; multiplication by $\lambda_i(\cdot;\theta_0)$ is harmless since $\lambda_i(t;\theta_0)\le C(1+\mathfrak C_t)$. Explicitly,
\begin{align*}
    &\left|\sum_{i=1}^D
H_{ai}(0;\vartheta)\partial_{\vartheta_b}\lambda_i(0;\vartheta)-\sum_{i=1}^D
H_{ai}(0;\eta)\partial_{\vartheta_b}\lambda_i(0;\eta)\right|\\
&\leq \sum_{i=1}^D |H_{ai}(0;\vartheta)|\left|\partial_{\vartheta_b}\lambda_i(0;\eta)-\partial_{\vartheta_b}\lambda_i(0;\vartheta)\right|+|H_{ai}(0;\vartheta)-H_{ai}(0;\eta)|\left|\partial_{\eta_b}\lambda_i(0;\eta)\right|\\
&\leq D B(1+\mathfrak C_0^{m_0})\norm{\eta-\vartheta}_2
\end{align*}
where the final inequality follows from the uniform second derivative envelope, the mean value theorem, and the uniform logarithmic envelopes on $\left|\partial_{\eta_b}\lambda_i(0;\eta)\right|$ and $|H_{ai}(0;\vartheta)|$. For $Z^{(1)}_{ab,\vartheta,\theta_0}(N)$
\begin{align*}
    \left|Z^{(1)}_{ab,\vartheta,\theta_0}(N)-Z^{(1)}_{ab,\eta,\theta_0}(N)\right|&\leq \sum_{i=1}^D\left(\lambda_i(0;\theta_0)+\lambda_i(0;\eta)\right)\left|\partial_{\vartheta_b}H_{ai}(0;\vartheta)-\partial_{\vartheta_b}H_{ai}(0;\eta)\right|+\\
    &\quad+\left|\partial_{\vartheta_b}H_{ai}(0;\vartheta)\right|\left|\lambda_i(0;\vartheta)-\lambda_i(0;\eta)\right|\leq B(1+\mathfrak C_0^{m})\norm{\vartheta-\eta}_2
\end{align*}
with $B,m$ being uniform over the parameter space. Therefore, Theorem~\ref{thm:local-concentration} gives, with probability at least $1-CT^{-c}$ after a union bound over the fixed entries,
\begin{equation}\label{eq:generic-jacobian-drift-control}
\sup_{\vartheta\in B(\theta^\star,r)}
\left|
\frac{1}{T}\int_0^T
\{Z^{(1)}_{ab,\vartheta,\theta^\star}\circ G_t
-Z^{(2)}_{ab,\vartheta}\circ G_t\}\,\dd t
-
\E_{\theta^\star}
\{Z^{(1)}_{ab,\vartheta,\theta^\star}-Z^{(2)}_{ab,\vartheta}\}
\right|
\le K\frac{\log^{q_0}(T)}{\sqrt T}
\end{equation}
for all sufficiently large $T$.  The expectation in \eqref{eq:generic-jacobian-drift-control} is the $(a,b)$ entry of $-B_H(\vartheta,\theta^\star)$.

It remains to control the martingale derivative term.  For fixed $(a,b,\vartheta)$ set
\[
M_{ab,\vartheta}^{J}(T)
=
\sum_{i=1}^D\int_0^T
\partial_{\vartheta_b}H_{ai}(t;\vartheta)\,\dd M_i^{\theta^\star}(t).
\]
Its predictable variation is
\begin{equation*}
V_{ab,\vartheta}^{J}(T)
=
\sum_{i=1}^D\int_0^T
\{\partial_{\vartheta_b}H_{ai}(t;\vartheta)\}^2
\lambda_i(t;\theta^\star)\,\dd t .
\end{equation*}
The predictable variation's integrand
\[
Z^{J,V}_{ab,\vartheta,\theta^\star}(N)
=
\sum_{i=1}^D
\{\partial_{\vartheta_b}H_{ai}(0;\vartheta)\}^2
\lambda_i(0;\theta^\star)
\]
has the same polynomial envelope and Lipschitz properties, which should be clear by repeating the argument in the proof of equation \eqref{eq:generic-ulln}. Theorem~\ref{thm:local-concentration}, stationarity and the exponential moment \eqref{eq:count-mgf} imply that, with probability at least $1-CT^{-c-p-2}$,
\begin{equation}\label{eq:generic-jacobian-bracket-event}
\sup_{\vartheta\in B(\theta^\star,r)}V_{ab,\vartheta}^{J}(T)
\le C_0T
\end{equation}
uniformly in $\theta^\star\in\Theta_0$, after increasing constants and taking $T$ sufficiently large. This follows from the same argument used to bound the martingale term in equation \eqref{eq:generic-ulln}.

Intersect the event that \eqref{eq:generic-jacobian-bracket-event} holds on with the sliding-window event
\[
    \calA_T=\left\{\sup_{0\le t\le T}\mathfrak C_t\le K_0\log(T),
    \quad \frac{1}{T}N([0,T])\le K_0\log(T)\right\}.
\]
On $\calA_T$, the predictable truncation
\[
\partial_{\vartheta_b}H_{ai}(t;\vartheta)\1\{\mathfrak C_t\le K_0\log(T)\}
\]
agrees with the original integrand and its jumps are bounded by $K\log^{q_0}(T)$.  Let $\calN_T(\theta^\star)$ be a deterministic $T^{-1/2}$-net of $B(\theta^\star,r)$ with cardinality at most $CT^{p/2}$.  The net $\calN_T(\theta^\star)$ is chosen deterministically after fixing $\theta^\star$.  The probability bound below is uniform in $\theta^\star$, and the subsequent union bound uses only the cardinality $\abs{\calN_T(\theta^\star)}\le CT^{p/2}$. For each fixed $\eta\in\calN_T(\theta^\star)$, Lemma~\ref{lem:stopped-dzv} with \eqref{eq:generic-jacobian-bracket-event} gives
\[
\Prob_{\theta^\star}
\left(
\abs{M_{ab,\eta}^{J}(T)}>K\sqrt{T}\log^{q_0}(T)
\right)
\le CT^{-c-p-2}
\]
for a larger $K$.  A union bound over the net $\calN_T(\theta^\star)$ which is of size at most $CT^{p/2}$ gives
\begin{equation}\label{eq:generic-jacobian-net-bound}
\max_{\eta\in\calN_T(\theta^\star)}
\abs{\frac{1}{T}M_{ab,\eta}^{J}(T)}
\le K\frac{\log^{q_0}(T)}{\sqrt T}
\end{equation}
with probability at least $1-CT^{-c}$, uniformly in $\theta^\star$.

The net bound extends to the full local ball by the second-derivative envelope.  If $\norm{\vartheta-\eta}_2\le T^{-1/2}$, then on $\calA_T$,
\[
\sup_{0\le t\le T}\max_{i=1}^D
\abs{\partial_{\vartheta_b}H_{ai}(t;\vartheta)-\partial_{\vartheta_b}H_{ai}(t;\eta)}
\le K\frac{\log^{m}(T)}{\sqrt{T}}.
\]
Using $\dd M_i^{\theta^\star}=\dd N_i-\lambda_i(t;\theta^\star)\dd t$ and $\lambda_i(t;\theta^\star)\le C(1+\mathfrak C_t)$, the difference between the martingale averages at $\vartheta$ and $\eta$ is at most $K\log^{q_0}(T)/\sqrt T$. To be precise,
\begin{align}\label{eq:MGLipschitzBound}
    \frac{1}{T}\left|M_{ab,\eta}^J-M_{ab,\vartheta}^J\right|\leq K\frac{\log^{m}(T)}{\sqrt{T}}\left\{\frac{1}{T}N([0,T])+\frac{1}{T}\sum_{i=1}^D\int_0^T\lambda_i(t;\theta^\star)\,\dd t\right\}\leq K\frac{\log^{m'}(T)}{\sqrt{T}}
\end{align}
Hence, \eqref{eq:generic-jacobian-net-bound} holds with the maximum replaced by the supremum over $B(\theta^\star,r)$, see e.g.\ equation \eqref{eq:epsilonnetextension}, for $q_0$ sufficiently large. We conclude the proof through a union bound over the fixed entries $(a,b)$ and \eqref{eq:generic-jacobian-drift-control} concludes the proof of \eqref{eq:generic-jacobian-control}.
\end{proof}

Prior to proving Theorem \ref{thm:generic-rate}, we need to prove one last local curvature lemma.

\begin{lemma}[Deterministic local moment curvature]\label{lem:generic-local-curvature}
Under Assumptions~\ref{ass:parameter}--\ref{ass:stability} and Assumptions~\ref{ass:weights} and~\ref{ass:moment-identification}, for every fixed deterministic symmetric positive-definite matrix $W$, there are constants $r_H>0$ and $b_H>0$, with $\Theta_0^{+r_H}\subset K_\Theta$, such that
\begin{equation*}
    \norm{g_H(\theta^\star+h,\theta^\star)}_W
    \ge b_H\norm{h}_2
\end{equation*}
for every $\theta^\star\in\Theta_0$ and every $h$ with $\norm{h}_2\le r_H$ and $\theta^\star+h\in\Theta$, where $g_H$ is defined in equation \eqref{eq:gH-def}.
\end{lemma}

\begin{proof}[Proof of Lemma~\ref{lem:generic-local-curvature}]
Assumption~\ref{ass:weights} and Lemma~\ref{lem:derivative-envelope} ensure that the integrand of $B_H(\vartheta,\theta^\star)$, which is
\[H(0;\vartheta)D_\theta(0;\vartheta)-\dot{H}(0;\vartheta)\left\{\lambda(0;\theta^\star)-\lambda(0;\vartheta)\right\},\]
satisfies the assumptions of Proposition~\ref{prop:finite-window-continuity}. In turn, this proposition implies the continuity of $(\vartheta,\theta)\mapsto B_H(\vartheta,\theta)$ on a neighbourhood of the diagonal in $\Theta\times\Theta_0$. Because $B_H(\theta,\theta)=A_H(\theta)$, as per definitions \eqref{eq:AH-def} and \eqref{eq:BH-two-parameter}, and since $W$ is fixed and positive definite, \eqref{eq:H-rank} implies
\[
    \inf_{\theta\in\Theta_0}
    \lambda_{\min}\{A_H(\theta)^\top W A_H(\theta)\}
    \ge
    \lambda_{\min}(W)\kappa_H
    =:\kappa_{H,W}>0 .
\]
Moreover, as $(\vartheta,\theta)\mapsto B_H(\vartheta,\theta)$ is continuous, there is
$r>0$ so small that
\[
    \lambda_{\min}\left\{
    B_H(\vartheta,\theta)^\top W B_H(\vartheta,\theta)
    \right\}
    \ge \frac{\kappa_{H,W}}{2}
\]
whenever $\theta\in\Theta_0$ and $\norm{\vartheta-\theta}_2\le r$, with $\Theta_0^{+r}\subset K_\Theta$. Taylor expanding the first argument of $g_H$ at $\theta$ gives
\[
    g_H(\theta+h,\theta)
    =
    -A_H(\theta)h+R(h,\theta),
    \qquad
    \norm{R(h,\theta)}_W\leq C\norm{h}_2^2,
\]
uniformly over $\theta\in\Theta_0$ as $g_H(\cdot,\theta)\in C^2(\Theta)$. Decreasing $r$ once more, if necessary, gives
\[
    \norm{g_H(\theta+h,\theta)}_W
    \geq
    \frac{\norm{h}_2}{2}\sqrt{\lambda_{\min}\left\{A_H(\theta)^\top W A_H(\theta)\right\}},
\]
which proves the claim with $r_H=r$ and $b_H=\frac{\sqrt{\kappa_{H,W}}}{2\sqrt{2}}$.
\end{proof}

\subsection{Proof of Theorem~\ref{thm:generic-rate} (GMM high-probability rate)}\label{app:proof:thm:generic-rate}
We now have the necessary tools to prove the first main result, Theorem \ref{thm:generic-rate}.

\begin{proof}[Proof of Theorem~\ref{thm:generic-rate}]
Fix $\epsilon>0$.  By \eqref{eq:H-global-separation} and the positive
definiteness of the fixed matrix $W$, the corresponding $W$-norm
separation at distance $\epsilon$ is positive:
\[
    \delta_{H,W,\epsilon}
    :=
    \inf_{\theta^\star\in\Theta_0}
    \inf_{\substack{\vartheta\in\Theta:\\
        \norm{\vartheta-\theta^\star}_2\ge\epsilon}}
    \norm{g_H(\vartheta,\theta^\star)}_W
    >0 .
\]
Indeed,
\[
    \norm{g_H(\vartheta,\theta^\star)}_W
    \ge
    \lambda_{\min}(W)^{1/2}
    \norm{g_H(\vartheta,\theta^\star)}_2 ,
\]
and the last Euclidean norm is uniformly separated from zero by
\eqref{eq:H-global-separation}.

The uniform law of large numbers \eqref{eq:generic-ulln} implies that, for all
sufficiently large $T$,
\[
\sup_{\theta^\star\in\Theta_0}
\Prob_{\theta^\star}
\left(
    \sup_{\vartheta\in\Theta}
    \norm{\frac 1T\Psi_T^H(\vartheta)-g_H(\vartheta,\theta^\star)}_W
    >\delta_{H,W,\epsilon}/4
\right)
\le CT^{-c}.
\]
On the complementary event, since
$g_H(\theta^\star,\theta^\star)=0$, we have
\[
    \norm{\frac 1T\Psi_T^H(\theta^\star)}_W
    \le \delta_{H,W,\epsilon}/4 .
\]
Moreover, for every
$\vartheta\in\Theta$ with
$\norm{\vartheta-\theta^\star}_2\ge\epsilon$,
\[
    \norm{\frac 1T\Psi_T^H(\vartheta)}_W
    \ge
    \norm{g_H(\vartheta,\theta^\star)}_W
    -
    \norm{\frac 1T\Psi_T^H(\vartheta)-g_H(\vartheta,\theta^\star)}_W
    \ge
    3\delta_{H,W,\epsilon}/4 .
\]
Thus no minimizer of
$Q_{T,W}^H(\vartheta)=\norm{\frac{1}{T}\Psi_T^H(\vartheta)}_W^2$
can lie outside $B(\theta^\star,\epsilon)$ on this event. Consequently,
\begin{equation}\label{eq:consistencygeneralH}
    \sup_{\theta^\star\in\Theta_0}
    \Prob_{\theta^\star}\left(\norm{\hat\theta_H-\theta^\star}_2>\epsilon\right)
    \le CT^{-c}
\end{equation}
for all sufficiently large $T$.

For the following let $r_H$ and $b_H$ be as in Lemma~\ref{lem:generic-local-curvature}. We next intersect the high-probability events from equations \eqref{eq:generic-truth-bound}, \eqref{eq:generic-jacobian-control} and \eqref{eq:consistencygeneralH}. Explicitly, with $\epsilon = r_H/2$,
\begin{align*}
\sup_{\theta^\star\in\Theta_0}&
\Prob_{\theta^\star}
\left(
    \norm{\frac 1T\Psi_T^H(\theta^\star)}_W
    \leq K_0\sqrt{\frac{\log(T)}{T}}
\right)
\geq 1- C_0T^{-c_0}\\
\sup_{\theta^\star\in\Theta_0}&
\Prob_{\theta^\star}
\left(
    \sup_{\vartheta\in B(\theta^\star,r_H)}
    \norm{\partial_\vartheta\left(\frac 1T\Psi_T^H(\vartheta)\right)+B_H(\vartheta,\theta^\star)}_{\op}
    \leq K_1\frac{\log^{q_0}(T)}{\sqrt T}
\right)
\geq 1-C_1T^{-c_1}\\
     \sup_{\theta^\star\in\Theta_0}&
    \Prob_{\theta^\star}\left(\norm{\hat\theta_H-\theta^\star}_2\leq \epsilon\right)
    \geq 1- C_2T^{-c_2}.
\end{align*}
We next compare any $\vartheta=\theta^\star+h$ with $\norm{h}_2\le r_H/2$ to the truth.  By the fundamental theorem of calculus
\begin{align*}
    \frac{1}{T}\Psi_T^H(\theta^\star+h)-\frac{1}{T}\Psi_T^H(\theta^\star)&=\int_0^1 \partial_\theta \left(\frac{1}{T}\Psi_T^H(\theta^\star+sh)\right)h\,\dd s\\
    &=g_H(\theta^\star+h,\theta^\star)+\int_0^1 \left\{\partial_\theta \left(\frac{1}{T}\Psi_T^H(\theta^\star+sh)\right)+B_H(\theta^\star+sh,\theta^\star)\right\}h\,\dd s\\
    &=g_H(\theta^\star+h,\theta^\star)+R_T(h)
\end{align*}
To control the remaining integral
\begin{align*}
    \norm{R_T(h)}_W\leq \int_0^1 \norm{\left\{\partial_\theta \left(\frac{1}{T}\Psi_T^H(\theta^\star+sh)\right)+B_H(\theta^\star+sh,\theta^\star)\right\}}_{\rm op}\norm{h}_2\, \dd s \leq K_1\frac{\log^{q_0}(T)}{\sqrt{T}}\norm{h}_2.
\end{align*}
Thus,
\[
    \frac{1}{T}\Psi_T^H(\theta^\star+h)
    =
     \frac{1}{T}\Psi_T^H(\theta^\star)
    +g_H(\theta^\star+h,\theta^\star)
    +R_T(h),
    \qquad
    \norm{R_T(h)}_W\leq K_1\frac{\log^{q_0}(T)}{\sqrt{T}}\norm{h}_2 .
\]
For all large $T$, $K_1\frac{\log^{q_0}(T)}{\sqrt{T}}\leq b_H/4$.  Therefore, by Lemma~\ref{lem:generic-local-curvature},
\[
    \norm{ \frac{1}{T}\Psi_T^H(\theta^\star+h)}_W
    \ge
    \frac{3b_H}{4}\norm{h}_2-K_0\sqrt{\frac{\log(T)}{T}} .
\]
If $\norm{h}_2>\frac{4K_0}{b_H}\sqrt{\frac{\log(T)}{T}}$, the right-hand side is strictly greater than $2K_0\sqrt{\frac{\log(T)}{T}}$, for all sufficiently large $T$, whereas
\[
    \norm{ \frac{1}{T}\Psi_T^H(\theta^\star)}_W\leq K_0\sqrt{\frac{\log(T)}{T}} .
\]
Thus no minimizer inside $B(\theta^\star,r_H/2)$ can lie outside a ball of radius $K\sqrt{\frac{\log(T)}{T}}$. 

This concludes the proof under a stationary start. Under a non-stationary start, that satisfies the finite pre-history condition, the result can easily be transferred to the non-stationary start via Lemma \ref{lem:tv-transfer}.

In particular, let $\psi_T:N_\theta^\eta|_{(b_T,\infty)}\mapsto \norm{\hat \theta_H-\theta^\star}_2$ and $E_T=\left\{\omega \ : \ \psi_T(\omega)>K\sqrt{\frac{\log(T)}{T}} \right\}$. Thus the proof concludes. 
\end{proof}

In order to demonstrate the central limit theorem, we first demonstrate that the estimating equation $\calL_{\theta^\star}\left\{T^{-1/2}\Psi_T^H(\theta^\star)\right\}$ is asymptotically normal, which we can leverage in a classical Taylor expansion argument.

\begin{proposition}[Martingale CLT for the estimating map]\label{prop:generic-map-clt}
Suppose Assumptions~\ref{ass:parameter}--\ref{ass:stability} and Assumption~\ref{ass:weights} hold.  Then
\begin{equation}
\label{eq:uniform-general-clt}
\sup_{\theta^\star \in\Theta_0}
 d_{\mathrm{BL}}
\left(
\mathcal L_{\theta^\star}\{T^{-1/2}\Psi_T^H(\theta^\star)\},
N_q(0,\Omega_H(\theta^\star))
\right)
\to0 .
\end{equation}
Here $d_{\mathrm{BL}}$ is the bounded-Lipschitz metric on probability laws on $\R^q$.
\end{proposition}

\begin{proof}[Proof of Proposition~\ref{prop:generic-map-clt}]
Suppose that \eqref{eq:uniform-general-clt} fails.  Then there are $\epsilon>0$, $T_n\to\infty$, and $\theta_n\in\Theta_0$ such that
\[
 d_{\mathrm{BL}}
\left(
\mathcal L_{\theta_n}\{T_n^{-1/2}\Psi_{T_n}^H(\theta_n)\},
N_q(0,\Omega_H(\theta_n))
\right)>
\epsilon .
\]
By compactness, pass to a subsequence, still denoted $\theta_n$, such that $\theta_n\to\theta_\infty\in\Theta_0$.

Fix $a\in\R^q$.  Under $\Prob_{\theta_n}$,
\[
X_n^a
=
 a^\top T_n^{-1/2}\Psi_{T_n}^H(\theta_n)
=
\sum_{i=1}^D\int_0^{T_n}
T_n^{-1/2}a^\top H_{\cdot i}(t;\theta_n)
\,\dd M_i^{\theta_n}(t)
\]
is a scalar martingale endpoint, the following argument follows Lemma 3.13 of \cite{clinet2017}. The predictable variation of the scalar martingale is
\[
V_n^a
=
\frac1{T_n}\sum_{i=1}^D\int_0^{T_n}
\{a^\top H_{\cdot i}(t;\theta_n)\}^2\lambda_i(t;\theta_n)\,\dd t .
\]
The integrand is a shift-covariant compact-window functional with a polynomial local-count envelope and a polynomial Lipschitz envelope, by Assumption~\ref{ass:weights} and Lemma~\ref{lem:derivative-envelope}.  Theorem~\ref{thm:local-concentration}, applied with the parameter index fixed at $\theta_n$, gives
\[
V_n^a-a^\top\Omega_H(\theta_n)a\to0
\qquad\text{in }\Prob_{\theta_n}\text{-probability}.
\]
Proposition~\ref{prop:finite-window-continuity} gives $\Omega_H(\theta_n)\to\Omega_H(\theta_\infty)$, and hence
\[
V_n^a\to a^\top\Omega_H(\theta_\infty)a
\qquad\text{in }\Prob_{\theta_n}\text{-probability}.
\]

The Lindeberg condition follows from the jump formula in Lemma~\ref{lem:predictable-jumps}.  For every $\delta>0$, using $x^2\1\{|x|>\delta\}\le |x|^3/\delta$,
\[
\begin{aligned}
&\E_{\theta_n}\left[
\sum_{0<t\le T_n}
(\Delta X_n^a(t))^2
\1\{|\Delta X_n^a(t)|>\delta\}
\right] \\
&\qquad\le
\frac1{\delta T_n^{3/2}}
\sum_{i=1}^D\int_0^{T_n}
\E_{\theta_n}\left[
|a^\top H_{\cdot i}(t;\theta_n)|^3\lambda_i(t;\theta_n)
\right]\dd t
\le
\frac{C_a}{\sqrt{T_n}},
\end{aligned}
\]
where the last bound uses stationarity, the polynomial envelopes, and the fixed-window exponential moment bound.  The martingale central limit theorem of Rebolledo, in the form of \citet[Theorem~VIII.3.24]{jacodshiryaev2003}, yields
\[
 a^\top T_n^{-1/2}\Psi_{T_n}^H(\theta_n)
 \Rightarrow
 N\{0,a^\top\Omega_H(\theta_\infty)a\} .
\]
By Cram\'er--Wold,
\[
 T_n^{-1/2}\Psi_{T_n}^H(\theta_n)
 \Rightarrow
 N_q(0,\Omega_H(\theta_\infty)).
\]
Also $N_q(0,\Omega_H(\theta_n))\Rightarrow N_q(0,\Omega_H(\theta_\infty))$.  Since $d_{\mathrm{BL}}$ metrizes weak convergence on $\R^q$, the displayed subsequential convergence contradicts the choice of $(T_n,\theta_n)$.  Therefore \eqref{eq:uniform-general-clt} holds.
\end{proof}

\subsection{Proof of Theorem~\ref{thm:generic-an} (GMM asymptotic normality)}\label{app:proof:thm:generic-an}

We now prove Theorem~\ref{thm:generic-an} via a standard Taylor expansion argument, after the estimate is contained within the interior of $\Theta_0$.
\begin{proof}[Proof of Theorem~\ref{thm:generic-an}]
 By Theorem~\ref{thm:generic-rate},
\[
    \hat\theta_H-\theta^\star=O_{\Prob_{\theta^\star}}\left(\sqrt{\frac{\log(T)}{T}}\right).
\]
Since $\Theta_0\Subset\Theta^\circ$, there is $\rho>0$ such that $B(\theta,\rho)\subset\Theta^\circ$ for every $\theta\in\Theta_0$.  The preceding rate implies
\[
\Prob_{\theta^\star}\{\hat\theta_H\in B(\theta^\star,\rho)\}\to1,
\]
so $\hat\theta_H$ is an interior minimizer with probability tending to one.  On this event the first-order condition is
\[
    \partial_\vartheta \left(\frac{1}{T^2} \Psi_T^H(\hat\theta_H)^\top W \Psi_T^H(\hat\theta_H)\right)=0.
\]
Moreover, optimality of the estimator $\hat \theta_H$ yields that
\[
    \left\|\frac{1}{T}\Psi_T^H(\hat\theta_H)\right\|_W
    \le
    \left\|\frac{1}{T}\Psi_T^H(\theta^\star)\right\|_W .
\]
By Proposition~\ref{prop:generic-map-clt}, $\frac{1}{T}\Psi_T^H(\theta^\star)=O_{\Prob_{\theta^\star}}(T^{-1/2})$, and the fixed positive-definite matrix $W$ has deterministic eigenvalue bounds.  Hence
\begin{equation}\label{eq:gmm-minimizer-moment-small}
    \frac{1}{T}\Psi_T^H(\hat\theta_H)=O_{\Prob_{\theta^\star}}(T^{-1/2}).
\end{equation}

Lemma~\ref{lem:generic-stochastic-differentiability}, with $h_T=\hat\theta_H-\theta^\star$, gives
\[
    \frac 1T \Psi_T^H(\hat\theta_H)
    =
    \frac 1T \Psi_T^H(\theta^\star)-A_H(\theta^\star)(\hat\theta_H-\theta^\star)+o_{\Prob_{\theta^\star}}(T^{-1/2})
\]
and
\[
    \partial_\vartheta \frac 1T \Psi_T^H(\hat\theta_H)
    =
    -A_H(\theta^\star)+o_{\Prob_{\theta^\star}}(1).
\]
Substituting these two expansions into the first-order condition and using \eqref{eq:gmm-minimizer-moment-small} yields
\[
    A_H^\top W A_H\sqrt T(\hat\theta_H-\theta^\star)
    =\frac{1}{\sqrt{T}}A_H^\top W \Psi_T^H(\theta^\star)+o_{\Prob_{\theta^\star}}(1),
\]
where $A_H$ is evaluated at $\theta^\star$.  Since $A_H^\top W A_H$ is nonsingular by Assumption~\ref{ass:moment-identification}, Proposition~\ref{prop:generic-map-clt} and Slutsky's theorem give \eqref{eq:generic-sandwich}; the oracle optimally weighted formula follows by setting $W=\Omega_H(\theta^\star)^{-1}$.

Transfer of the central limit theorem to non-stationary start follows from equation \eqref{eq:BurninDistributionalLimit}, and using that $\sqrt{T-b_T}\sim \sqrt{T}.$

\end{proof}

\subsection{Least-squares contrast minimizer}\label{app:ls-contrast-estimator}

In this section of the supplemental material we apply our theoretical results to the least squares estimator introduced in example \ref{ex:weight-families-body}. We do so since, like the maximum likelihood estimate, it is a member of our derivative weight class via its first order condition, i.e.\ that it is the zero of a derivative estimator. To ensure the maximum of the equation \eqref{eq:ls-contrast} lies inside $\Theta_0$ we must first localize $\hat \theta_{\rm LS}$ to be contained within $\Theta_0^\circ$ with high probability.

\begin{proposition}[Least-squares contrast minimizer]\label{prop:supp-ls-contrast-estimator}
Let \(\hat\theta_{\rm LS}\) be a measurable global minimizer of the least-squares contrast \(\gamma_T\) in \eqref{eq:ls-contrast} over \(\Theta\), and define
\[
    A_{\rm LS}(\theta)=\E_\theta\{D_\theta(0;\theta)^\top D_\theta(0;\theta)\},
    \qquad
    \Omega_{\rm LS}(\theta)=\E_\theta\{D_\theta(0;\theta)^\top\Lambda(0;\theta)D_\theta(0;\theta)\}.
\]
Under Assumptions~\ref{ass:parameter}--\ref{ass:identifiability}, for every \(c>0\) there are \(C,K<\infty\) such that, for all sufficiently large \(T\),
\[
    \sup_{\theta^\star\in\Theta_0}
    \Prob_{\theta^\star}\left(
    \norm{\hat\theta_{\rm LS}-\theta^\star}_2>K\sqrt{\frac{\log(T)}{T}}
    \right)
    \le CT^{-c}.
\]
For each fixed \(\theta^\star\in\Theta_0\),
\[
    \sqrt T(\hat\theta_{\rm LS}-\theta^\star)
    \Rightarrow
    N_p\{0,A_{\rm LS}(\theta^\star)^{-1}\Omega_{\rm LS}(\theta^\star)A_{\rm LS}(\theta^\star)^{-1}\}.
\]
\end{proposition}

\begin{proof}
Set \(H_D(t;\theta)=D_\theta(t;\theta)^\top\) and \(m_T^D(\theta)=T^{-1}\Psi_T^{H_D}(\theta)\). By Lemma~\ref{lem:derivative-envelope}, \(H_D\) satisfies the envelope requirements of Assumption~\ref{ass:weights}. Let \(\underline\lambda>0\) be the uniform lower bound on the intensities over \(K_\Theta\). Then
\[
    A_{\rm LS}(\theta)
    =\E_\theta\{D_\theta(0;\theta)^\top D_\theta(0;\theta)\}
    \succeq \underline\lambda\,I(\theta),
\]
so \(A_{\rm LS}\) is uniformly nonsingular on \(\Theta_0\) by Assumption~\ref{ass:identifiability}. Similarly,
\[
    \Omega_{\rm LS}(\theta)
    =\E_\theta\{D_\theta(0;\theta)^\top\Lambda(0;\theta)D_\theta(0;\theta)\}
    \succeq \underline\lambda\,A_{\rm LS}(\theta)
\]
is uniformly nonsingular on \(\Theta_0\). On \(N([-A,T])<\infty\), the map \(\theta\mapsto\gamma_T(\theta)\) is continuous on compact \(\Theta\), by the same finite-event and dominated-convergence argument as Proposition~\ref{prop:gmm-measurable-selection}; hence a measurable global minimizer exists.

Define the population contrast
\[
    \Gamma(\vartheta,\theta^\star)
    =
    \E_{\theta^\star}
    \sum_{i=1}^D
    \{\lambda_i(0;\vartheta)^2-2\lambda_i(0;\vartheta)\lambda_i(0;\theta^\star)\}.
\]
The same compact-window concentration, martingale-bracket, and net argument used in Proposition~\ref{prop:generic-controls} gives, for every \(c>0\), constants \(C,K<\infty\) and \(q_0\ge1\) such that
\[
    \sup_{\theta^\star\in\Theta_0}
    \Prob_{\theta^\star}
    \left(
    \sup_{\vartheta\in\Theta}
    \left|\gamma_T(\vartheta)-\Gamma(\vartheta,\theta^\star)\right|
    >K\frac{\log^{q_0}T}{\sqrt T}
    \right)
    \le CT^{-c}.
\]
Indeed, writing \(\dd N_i(t)=\lambda_i(t;\theta^\star)\dd t+\dd M_i^{\theta^\star}(t)\), the deterministic average has integrand \(\lambda_i(0;\vartheta)^2-2\lambda_i(0;\vartheta)\lambda_i(0;\theta^\star)\), while the martingale average has integrand \(\lambda_i(0;\vartheta)\). Both are compact-window functionals with polynomial local-count and Lipschitz envelopes uniformly over \((\vartheta,\theta^\star)\in\Theta\times\Theta_0\).

Moreover,
\[
    \Gamma(\vartheta,\theta^\star)-\Gamma(\theta^\star,\theta^\star)
    =
    \E_{\theta^\star}\bigl\|\lambda(0;\vartheta)-\lambda(0;\theta^\star)\bigr\|_2^2 .
\]
If this quantity is zero, stationarity gives equality of the two intensity paths for Lebesgue-a.e. time under \(\Prob_{\theta^\star}\), and Assumption~\ref{ass:identifiability} gives \(\vartheta=\theta^\star\). Proposition~\ref{prop:finite-window-continuity} gives continuity of \((\vartheta,\theta^\star)\mapsto\Gamma(\vartheta,\theta^\star)\), so compactness yields, for every \(\epsilon>0\),
\[
    \delta_\epsilon
    :=
    \inf_{\theta^\star\in\Theta_0}
    \inf_{\norm{\vartheta-\theta^\star}_2\ge\epsilon}
    \{\Gamma(\vartheta,\theta^\star)-\Gamma(\theta^\star,\theta^\star)\}
    >0.
\]
Combining this separation with the preceding uniform law gives
\[
    \sup_{\theta^\star\in\Theta_0}
    \Prob_{\theta^\star}\left(
    \norm{\hat\theta_{\rm LS}-\theta^\star}_2>\epsilon
    \right)
    \le CT^{-c}
\]
for all sufficiently large \(T\).

It remains to sharpen consistency to the stated logarithmic rate. Since \(\Theta_0\Subset\Theta^\circ\), the preceding consistency places \(\hat\theta_{\rm LS}\) in an interior neighborhood of \(\theta^\star\) with probability at least \(1-CT^{-c}\). On this event the first-order condition holds, and \(m_T^D(\hat\theta_{\rm LS})=0\), because \(\nabla_\theta\gamma_T(\theta)=-(2/T)\Psi_T^{H_D}(\theta)\). The local-curvature argument used in Lemma~\ref{lem:generic-local-curvature}, applied here to \(H_D\), gives constants \(r_D,b_D>0\) such that
\[
    \|g_{H_D}(\theta^\star+h,\theta^\star)\|_2
    \ge b_D\|h\|_2
\]
whenever \(\theta^\star\in\Theta_0\), \(\|h\|_2\le r_D\), and \(\theta^\star+h\in\Theta\). This local statement uses only the rank bound just verified, not global derivative-moment separation.

On the high-probability event from Proposition~\ref{prop:generic-controls}, the Taylor expansion
\[
    m_T^D(\theta^\star+h)
    =
    m_T^D(\theta^\star)+g_{H_D}(\theta^\star+h,\theta^\star)+R_T(h)
\]
has \(\|R_T(h)\|_2\le K\log^{q_0}(T)T^{-1/2}\|h\|_2\), and for large \(T\) this last coefficient is at most \(b_D/2\). Since \(m_T^D(\hat\theta_{\rm LS})=0\) and \(\|m_T^D(\theta^\star)\|_2\le K\sqrt{\log(T)/T}\), we obtain
\[
    \|\hat\theta_{\rm LS}-\theta^\star\|_2
    \le K\sqrt{\frac{\log(T)}{T}},
\]
uniformly with polynomially high probability.

For the limit distribution, fix \(\theta^\star\in\Theta_0\). The rate just proved and Lemma~\ref{lem:generic-stochastic-differentiability}, again with \(H=H_D\), yield
\[
    0
    =
    m_T^D(\hat\theta_{\rm LS})
    =
    m_T^D(\theta^\star)
    -
    A_{\rm LS}(\theta^\star)(\hat\theta_{\rm LS}-\theta^\star)
    +
    o_{\Prob_{\theta^\star}}(T^{-1/2}).
\]
Hence
\[
    \sqrt T(\hat\theta_{\rm LS}-\theta^\star)
    =
    A_{\rm LS}(\theta^\star)^{-1}T^{-1/2}\Psi_T^{H_D}(\theta^\star)
    +
    o_{\Prob_{\theta^\star}}(1).
\]
Proposition~\ref{prop:generic-map-clt} gives \(T^{-1/2}\Psi_T^{H_D}(\theta^\star)\Rightarrow N_p\{0,\Omega_{\rm LS}(\theta^\star)\}\), which proves the asserted sandwich limit.
\end{proof}

The remainder of this section of the supplement aims to transfer the main results to the two-step estimator. It does so by demonstrating that the plug-in optimal weight matrix concentrates uniformly around the population optimal weight matrix, and then applying standard M-estimation arguments. For later use, define the two-parameter population covariance
\begin{equation*}
\bar\Omega_H(\vartheta,\theta^\star)
=
\E_{\theta^\star}\{H(0;\vartheta)\Lambda(0;\vartheta)H(0;\vartheta)^\top\}.
\end{equation*}
Thus $\Omega_H(\theta)=\bar\Omega_H(\theta,\theta)$.

\begin{lemma}[Polynomial-tail well-conditioning of the plug-in covariance]
\label{lem:plugin-inverse-polynomial}
Suppose Assumptions~\ref{ass:parameter}--\ref{ass:stability} and
Assumptions~\ref{ass:weights}--~\ref{ass:moment-identification} hold.

Let $\tilde\theta_T$ be a preliminary estimator satisfying the same
polynomial-tail logarithmic rate as in Theorem~\ref{thm:generic-rate}; that is,
for every $c>0$ there are $C,K<\infty$ such that, for all
large $T$,
\[
    \sup_{\theta^\star\in\Theta_0}
    \Prob_{\theta^\star}\left(
    \norm{\tilde\theta_T-\theta^\star}_2
    >
    K\sqrt{\frac{\log(T)}{T}}
    \right)
    \le CT^{-c}.
\]
Then, for every $c>0$, there are $C<\infty$ and deterministic constants
$0<\omega_-<\omega_+<\infty$ such that, for all sufficiently large $T$,
\begin{equation}\label{eq:plugin-invert}
\sup_{\theta^\star\in\Theta_0} \Prob_{\theta^\star}\left(\lambda_{\min}\{\hat\Omega_{H,T}(\tilde\theta_T)\}<\omega_-\  \vee \lambda_{\max}\{\hat\Omega_{H,T}(\tilde\theta_T)\}>\omega_+\right) \le CT^{-c}.
\end{equation}
Equivalently, on an event of probability at least $1-CT^{-c}$,
\[
    \omega_-I_q
    \preceq
    \hat\Omega_{H,T}(\tilde\theta_T)
    \preceq
    \omega_+ I_q
\]
and hence
\[
    \omega_+^{-1}I_q
    \preceq
    \hat W_T
    \preceq
    \omega_-^{-1}I_q,
    \qquad
    \hat W_T:=\hat\Omega_{H,T}(\tilde\theta_T)^{-1}.
\]
\end{lemma}

\begin{proof}
First, for each entry of the matrix
$H(0;\vartheta)\Lambda(0;\vartheta)H(0;\vartheta)^\top$, Assumption~\ref{ass:weights}
and Lemma~\ref{lem:derivative-envelope} give a polynomial local-count envelope, uniformly over
$\vartheta\in\Theta$, in particular, there exist uniform $B,m>0$ such that 
\[\left|\left(H(0;\vartheta)\Lambda(0;\vartheta)H(0;\vartheta)^\top\right)_{ij}\right|\leq C(1+\mathfrak C_0^m)\]
for each $ij$ pair. This follows from the fact that $p,q$ and $D$ are fixed so the operator norm bounds for $H$ imply entrywise bounds. This holds true also for the uniform polynomial local-count Lipschitz envelope. Therefore, entrywise concentration from Theorem~\ref{thm:local-concentration}, which reads for all $c>0$ there exists a $K>0$, $C>0$ and $q_0>0$ so that 
\[\sup_{\theta^\star \in \Theta_0}\Prob_{\theta^\star}\left(\sup_{\vartheta \in \Theta}\abs{\hat\Omega_{H,T}(\vartheta)_{ij}
-
\bar\Omega_H(\vartheta,\theta^\star)_{ij}}>K\frac{\log^{q_0}(T)}{\sqrt{T}}\right)\leq CT^{-c}\]
and the fixed dimension converts the
entrywise bound to the operator norm, for some new constant $K$ proving
\begin{equation}\label{eq:Omega-plugin-uniform}
\sup_{\theta^\star \in \Theta_0}\Prob_{\theta^\star}\left(\sup_{\vartheta \in \Theta}\norm{\hat\Omega_{H,T}(\vartheta)
-
\bar\Omega_H(\vartheta,\theta^\star)}_{\op}>K\frac{\log^{q_0}(T)}{\sqrt{T}}\right)\leq CT^{-c}
\end{equation}

Next, note that the population plug-in covariance is locally Lipschitz in its
first argument, uniformly over the true parameter.  More precisely, there are
$r_\Omega>0$ and $L_\Omega<\infty$ such that, whenever
$\theta^\star\in\Theta_0$ and
$\|\vartheta-\theta^\star\|_2\le r_\Omega$,
\[
    \|\bar\Omega_H(\vartheta,\theta^\star)
      -\Omega_H(\theta^\star)\|_{\op}
    \le
    L_\Omega\|\vartheta-\theta^\star\|_2 .
\]
This can be seen by considering 
\[
    F(\vartheta,N)
    =
    H(0;\vartheta)\Lambda(0;\vartheta)H(0;\vartheta)^\top ,
\]
the derivative $\partial_\vartheta F(\vartheta,N)$ is a finite sum of terms
of the form
\[
    \partial_\vartheta H(0;\vartheta)\Lambda(0;\vartheta)H(0;\vartheta)^\top,
    \qquad
    H(0;\vartheta)\partial_\vartheta\Lambda(0;\vartheta)H(0;\vartheta)^\top,
\]
and their transposes.  By Assumption~\ref{ass:weights} and
Lemma~\ref{lem:derivative-envelope}, these derivatives have a polynomial
local-count envelope uniformly on a fixed neighbourhood of $\Theta_0$.
The fixed-window moment bound then gives an integrable uniform envelope under
$\Prob_{\theta^\star}$, uniformly in $\theta^\star\in\Theta_0$.  The mean-value
theorem therefore gives the displayed Lipschitz bound.

Let
\[
    a_T=\sqrt{\frac{\log(T)}{T}},
    \qquad
    b_T=\frac{\log^{q_0}(T)}{\sqrt T},
\]
where $q_0$ is the exponent in equation \eqref{eq:Omega-plugin-uniform}. Next, define the events
\[
    \mathcal E_{\Omega,T}
    =
    \left\{
    \sup_{\vartheta\in\Theta}
    \|\hat\Omega_{H,T}(\vartheta)
      -\bar\Omega_H(\vartheta,\theta^\star)\|_{\op}
    \le K_\Omega b_T
    \right\}
\]
and
\[
    \mathcal E_{\theta,T}
    =
    \left\{
    \|\tilde\theta_T-\theta^\star\|_2
    \le K_\theta a_T
    \right\}.
\]
By equation \eqref{eq:Omega-plugin-uniform} and by the assumed polynomial-tail
rate of $\tilde\theta_T$, after increasing constants,
\[
    \sup_{\theta^\star\in\Theta_0}
    \Prob_{\theta^\star}\{(\mathcal E_{\Omega,T}\cap\mathcal E_{\theta,T})^c\}
    \le CT^{-c}.
\]
For all sufficiently large $T$, $K_\theta a_T\le r_\Omega$.  On
$\mathcal E_{\Omega,T}\cap\mathcal E_{\theta,T}$,
\[
\begin{aligned}
&\|\hat\Omega_{H,T}(\tilde\theta_T)
      -\Omega_H(\theta^\star)\|_{\op}  \\
&\qquad\le
\|\hat\Omega_{H,T}(\tilde\theta_T)
      -\bar\Omega_H(\tilde\theta_T,\theta^\star)\|_{\op}
+
\|\bar\Omega_H(\tilde\theta_T,\theta^\star)
      -\Omega_H(\theta^\star)\|_{\op}  \\
&\qquad\le
K_\Omega b_T+L_\Omega K_\theta a_T .
\end{aligned}
\]
Since $a_T\to0$ and $b_T\to0$, the last display is at most
$\omega_0/2$ for all sufficiently large $T$.  Weyl's eigenvalue
inequality gives
\[
    \lambda_{\min}\{\hat\Omega_{H,T}(\tilde\theta_T)\}
    \ge
    \lambda_{\min}\{\Omega_H(\theta^\star)\}
    -
    \|\hat\Omega_{H,T}(\tilde\theta_T)
      -\Omega_H(\theta^\star)\|_{\op}
    \geq\omega_0/2.
\]
Thus $\hat\Omega_{H,T}(\tilde\theta_T)$ is invertible on this event.

It remains only to obtain the upper eigenvalue bound. By continuity of
$\Omega_H(\theta)$ on the compact set $\Theta_0$,
\[
    \bar\omega
    :=
    \sup_{\theta\in\Theta_0}\lambda_{\max}\{\Omega_H(\theta)\}<\infty .
\]
Again by Weyl's inequality, on the same event,
\[
    \lambda_{\max}\{\hat\Omega_{H,T}(\tilde\theta_T)\}
    \le
    \bar\omega+\omega_0/2
    =\vcentcolon\omega_+<\infty .
\]
Therefore
\[
    \frac{\omega_0}{2}I_q
    \preceq
    \hat\Omega_{H,T}(\tilde\theta_T)
    \preceq
    \omega_+ I_q.
\]
Inverting reverses the Loewner order and gives
\[
    \omega_+^{-1}I_q
    \preceq
    \hat W_T
    \preceq
    2\omega_0^{-1}I_q .
\]
With $\omega_-=\omega_0/2$, the bounds $\lambda_{\min}\{\hat\Omega_{H,T}(\tilde\theta_T)\}\ge\omega_-$ and $\lambda_{\max}\{\hat\Omega_{H,T}(\tilde\theta_T)\}\le\omega_+$ are exactly the complement of the event in \eqref{eq:plugin-invert}, which proves the claim.
\end{proof}

\begin{lemma}[Local stochastic differentiability of compact-window GMM]\label{lem:generic-stochastic-differentiability}
Fix $\theta^\star\in\Theta_0$.  Under Assumptions~\ref{ass:parameter}--\ref{ass:stability} and Assumption~\ref{ass:weights}, if $h_T=O_{\Prob_{\theta^\star}}(\sqrt{\log(T)/T})$ and $\theta^\star+h_T\in\Theta$ with probability tending to one, then
\begin{align}
    \frac 1T \Psi_T^H(\theta^\star+h_T)
    &=\frac 1T \Psi_T^H(\theta^\star)-A_H(\theta^\star)h_T+o_{\Prob_{\theta^\star}}(T^{-1/2}),
    \label{eq:generic-stoch-diff-map}\\
    \partial_\vartheta \frac 1T \Psi_T^H(\theta^\star+h_T)
    &=-A_H(\theta^\star)+o_{\Prob_{\theta^\star}}(1).
    \label{eq:generic-stoch-diff-jac}
\end{align}
\end{lemma}

\begin{proof}[Proof of Lemma~\ref{lem:generic-stochastic-differentiability}]
Choose $r>0$ such that $B(\theta^\star,r)\subset K_\Theta$.  Since
$h_T=o_{\Prob_{\theta^\star}}(1)$, the whole segment
$\{\theta^\star+s h_T:0\le s\le1\}$ lies in $B(\theta^\star,r)$ with
probability tending to one.  By Proposition~\ref{prop:generic-controls}, for
some integer $q_0$,
\[
    \sup_{\vartheta\in B(\theta^\star,r)}
    \norm{\partial_\vartheta\frac 1T \Psi_T^H(\vartheta)+B_H(\vartheta,\theta^\star)}_{\op}
    =O_{\Prob_{\theta^\star}}\left(\frac{\log^{q_0}T}{\sqrt T}\right).
\]
Continuity of $B_H(\vartheta,\theta^\star)$ at $\vartheta=\theta^\star$, together with $B_H(\theta^\star,\theta^\star)=A_H(\theta^\star)$, gives
\[
    B_H(\theta^\star+h_T,\theta^\star)=A_H(\theta^\star)+o_{\Prob_{\theta^\star}}(1),
\]
which proves \eqref{eq:generic-stoch-diff-jac}.

For \eqref{eq:generic-stoch-diff-map}, use the fundamental theorem of calculus:
\[
\begin{aligned}
   \frac 1T \Psi_T^H(\theta^\star+h_T)-\frac 1T \Psi_T^H(\theta^\star)
    &=-\int_0^1 B_H(\theta^\star+s h_T,\theta^\star)h_T\,\dd s \\
    &\quad+
    \int_0^1\left\{\partial_\vartheta \frac 1T \Psi_T^H(\theta^\star+s h_T)+B_H(\theta^\star+s h_T,\theta^\star)\right\}h_T\,\dd s .
\end{aligned}
\]
The second integral is
\[
    O_{\Prob_{\theta^\star}}\left(\frac{\log^{q_0}T}{\sqrt T}\right)
    O_{\Prob_{\theta^\star}}\left(\sqrt{\frac{\log(T)}{T}}\right)
    =o_{\Prob_{\theta^\star}}(T^{-1/2}).
\]
It remains to replace $B_H(\theta^\star+s h_T,\theta^\star)$ by $A_H(\theta^\star)$.  The exact two-parameter derivative is
\begin{equation*}
\begin{aligned}
B_H(\vartheta,\theta^\star)
&=
\E_{\theta^\star}\{H(0;\vartheta)D_\theta(0;\vartheta)\}  \\
&\quad-
\E_{\theta^\star}\left[
\partial_\vartheta H(0;\vartheta)
\{\lambda(0;\theta^\star)-\lambda(0;\vartheta)\}
\right],
\end{aligned}
\end{equation*}
where the second line is interpreted as the corresponding contraction of the derivative tensor of $H$ with the intensity difference.  Hence, for $\vartheta=\theta^\star+s h_T$,
\begin{align*}
\norm{B_H(\vartheta,\theta^\star)-A_H(\theta^\star)}_{\op} &\le
\norm{\E_{\theta^\star}\{H(0;\vartheta)D_\theta(0;\vartheta)-H(0;\theta^\star)D_\theta(0;\theta^\star)\}}_{\op}\\
&+
\norm{\E_{\theta^\star}\left[
\partial_\vartheta H(0;\vartheta)
\{\lambda(0;\theta^\star)-\lambda(0;\vartheta)\}
\right]}_{\op} .
\end{align*}
The first term is $\mathcal O(\norm{\vartheta-\theta^\star}_2)$ by the Lipschitz envelopes for $H$, Lemma \ref{lem:derivative-envelope} for the envelope of $D_\theta$, the mean-value theorem, and the fixed-window moment bound. The second term is also $\mathcal O(\norm{\vartheta-\theta^\star}_2)$, because $\nabla_\vartheta H$ has a polynomial envelope by assumption \ref{ass:weights}, and 
\[\norm{\lambda(0;\vartheta)-\lambda(0;\theta^\star)}_2\leq C(1+\mathfrak C_0^m)\norm{\vartheta-\theta^\star}_2\] 
by the polynomial local-count envelope of $\nabla \lambda(0;\theta)$ and the mean-value theorem. Therefore,
\[
    \sup_{0\le s\le1}
    \norm{B_H(\theta^\star+s h_T,\theta^\star)-A_H(\theta^\star)}_{\op}
    \le C\norm{h_T}_2
\]
with the displayed deterministic bound holding on the event $\theta^\star+h_T\in\Theta$.  Consequently,
\[
    \norm{\int_0^1\{B_H(\theta^\star+s h_T,\theta^\star)-A_H(\theta^\star)\}h_T\,\dd s}_2
    \le C\norm{h_T}_2^2
    =O_{\Prob_{\theta^\star}}\left(\frac{\log(T)}{T}\right)
    =o_{\Prob_{\theta^\star}}(T^{-1/2}).
\]
Combining the preceding displays proves \eqref{eq:generic-stoch-diff-map}.
\end{proof}

\subsection{Proof of Corollary~\ref{cor:two-step-gmm} (feasible optimal GMM)}\label{app:proof:cor:two-step-gmm}
The necessary regularity conditions have now been obtained to prove the asymptotic results for the two-step estimator. We first use Lemma \ref{lem:plugin-inverse-polynomial} for local invertibility which immediately implies the proof of Theorem \ref{thm:generic-rate} holds. The central limit theorem follows from the same argument as Theorem \ref{thm:generic-an}.

\begin{proof}[Proof of Corollary~\ref{cor:two-step-gmm}]
By Lemma~\ref{lem:plugin-inverse-polynomial}, there are finite positive constants, $\omega_-,\omega_+$ such that 
\[
    \omega_+^{-1}I_q
    \preceq
    \hat W_T
    \preceq
    2\omega_0^{-1}I_q,
    \qquad
    \hat W_T:=\hat\Omega_{H,T}(\tilde\theta_T)^{-1}.
\]
On this event, the consistency and local-rate proof of
Theorem~\ref{thm:generic-rate} applies with $\hat W_T$ in place of the
fixed weighting matrix. The lower bound on the smallest eigenvalue follows from $\omega_+^{-1}I_q\preceq\hat W_T$ and the full-column-rank condition on $A_H$.  Moreover, $\|x\|_{\hat W_T}\ge \omega_+^{-1/2}\|x\|_2$ on this event, so the global separation in Assumption~\ref{ass:moment-identification} transfers to the random weighted norm. In particular, this implies
\begin{align*}
   \sup_{\theta^\star\in \Theta_0}\Prob_{\theta^\star}\left(\norm{\hat\theta_H^{(2)}-\theta^\star}_2
   >K\sqrt{\frac{\log(T)}{T}}\right)\leq  CT^{-c}
\end{align*}
by equation \eqref{eq:plugin-invert}. The proof of Theorem~\ref{thm:generic-rate} can be applied exactly to the first summand above, using the lower bounding constant $\omega_+^{-1}$ in place of $\lambda_{\min}(W)$. Therefore, equation \eqref{eq:pluginHPrate} holds under a stationary or non-stationary start (after burning in the data).

To prove equation \eqref{eq:plugin-CLT}, first note that $\Theta_0\Subset\Theta^\circ$ so that the second-step estimator is an interior minimizer with probability tending to one with first-order condition
\[
    \partial_\vartheta\left( \frac{1}{T^2}\Psi_T^H(\hat\theta_H^{(2)})^\top
    \hat W_T
    \Psi_T^H(\hat\theta_H^{(2)})\right)=0\]
Lemma~\ref{lem:generic-stochastic-differentiability} yields
\[
    \frac{1}{T}\Psi_T^H(\hat\theta_H^{(2)})
    =
     \frac{1}{T}\Psi_T^H(\theta^\star)
    -
    A_H(\hat\theta_H^{(2)}-\theta^\star)
    +
    o_{\Prob_{\theta^\star}}(T^{-1/2}),
\]
and
\[
    \partial_\vartheta \left( \frac{1}{T}\Psi_T^H(\hat\theta_H^{(2)})\right)
    =
    -A_H+o_{\Prob_{\theta^\star}}(1).
\]

Moreover, on the same event, optimality gives
\[
     \frac{1}{T^2}\Psi_T^H\left(\hat\theta_H^{(2)}\right)^\top
    \hat W_T
  \Psi_T^H\left(\hat\theta_H^{(2)}\right)
    \leq
     \frac{1}{T^2}\Psi_T^H(\theta^\star)^\top
    \hat W_T
     \Psi_T^H(\theta^\star).
\]
Since  $\omega_+^{-1}I_q\preceq\hat W_T\preceq 2\omega_0^{-1}I_q$ and, by
Proposition~\ref{prop:generic-map-clt},
$\frac{1}{T}\Psi_T^H(\theta^\star)
=O_{\Prob_{\theta^\star}}(T^{-1/2})$, it follows that
\[
    \frac{1}{T}\Psi_T^H\left(\hat\theta_H^{(2)}\right)
    =
    O_{\Prob_{\theta^\star}}(T^{-1/2}).
\]

Substituting the previous displays into the first-order condition yields
\[
    A_H^\top\hat W_TA_H
    \sqrt T(\hat\theta_H^{(2)}-\theta^\star)
    =
    A_H^\top\hat W_T
   \frac{1}{\sqrt{T}}\Psi_T^H(\theta^\star)
    +
    o_{\Prob_{\theta^\star}}(1).
\]

If $\hat W_T\to\Omega_H^{-1}(\theta^\star)$, Slutsky's theorem, and non-singularity of $A_H^\top\Omega_H^{-1}A_H$ gives the displayed asymptotic linear expansion.  The limiting normal distribution follows from Proposition~\ref{prop:generic-map-clt}. Moreover, in the case of a non-stationary start that satisfies the local integrability assumptions, the justification at the end of the proof of Theorem \ref{thm:generic-an} via the same total variation bound implies the same limit after a suitable burn-in period and rescaling.

It remains to show $\hat W_T\xrightarrow[]{p} \Omega_H(\theta^\star)^{-1}$.
By \eqref{eq:Omega-plugin-uniform},
\[\sup_{\vartheta\in\Theta}\norm{\hat\Omega_{H,T}(\vartheta)-\bar\Omega_H(\vartheta,\theta^\star)}_{\op}=o_{\Prob_{\theta^\star}}(1).\]
Since $\tilde\theta_T\xrightarrow[]{p}\theta^\star$ and $\vartheta\mapsto\bar\Omega_H(\vartheta,\theta^\star)$ is continuous at $\theta^\star$,
\[\bar\Omega_H(\tilde\theta_T,\theta^\star)\xrightarrow[]{p}\bar\Omega_H(\theta^\star,\theta^\star)=\Omega_H(\theta^\star).\]
Therefore,
\[\hat\Omega_{H,T}(\tilde\theta_T)\xrightarrow[]{p}\Omega_H(\theta^\star).\]
By the positive-definiteness of $\Omega_H(\theta^\star)$ and the continuous
mapping theorem,
\[\hat W_T\xrightarrow[]{p}\Omega_H(\theta^\star)^{-1}.\]

\end{proof}

\subsection{Proof of Corollary~\ref{cor:almost-sure-rate}}
In order to demonstrate the almost sure convergence results of Corollary~\ref{cor:almost-sure-rate} we must first prove a lemma about the almost sure regularity of the estimating equations. In doing so, we can ensure the estimating equations are regular for all sufficiently large $T$ almost surely. This must be done since if one was to na\"ively apply the first Borel-Cantelli lemma to equation \eqref{eq:generic-rate} then the result would only hold almost surely for all large $T_n$ in a \emph{countable} sequence of times $\{T_n\}_{n\in \N}$. After proving the regularity lemma, we can apply the same pathwise arguments as used Theorem \ref{thm:generic-rate} and Corollary \ref{cor:two-step-gmm}.

\begin{lemma}[Pathwise interpolation over unit horizon intervals]
\label{lem:pathwise-unit-interpolation}
Fix $\theta^\star\in\Theta_0$ and let
\[
    m_T(\vartheta)=\frac{1}{T}\Psi_T^H(\vartheta).
\]
Furthermore, for $n\geq2$ define the event
\[
    \mathcal C_n
    =
    \left\{
    \sup_{0\le t\le n+1}\mathfrak C_t\le L\log(n)
    \right\}.
\]
Under
Assumptions~\ref{ass:parameter}--\ref{ass:stability} and
Assumption~\ref{ass:weights}, there exist constants
$C<\infty$ and $q_\star<\infty$ such that the following holds.

On $\mathcal C_n$, for all sufficiently large $n$,
\begin{align}
\sup_{T\in[n,n+1]}\sup_{\vartheta\in\Theta}
\left\|m_T(\vartheta)-m_n(\vartheta)\right\|_2
&\le
C\frac{\log^{q_\star}(n)}{n},
\label{eq:pathwise-interpolation-map}\\
\sup_{T\in[n,n+1]}
\sup_{\vartheta\in B(\theta^\star,r)}
\left\|
\partial_\vartheta m_T(\vartheta)
-
\partial_\vartheta m_n(\vartheta)
\right\|_{\op}
&\le
C\frac{\log^{q_\star}(n)}{n}
\label{eq:pathwise-interpolation-jac}\\
\sup_{T\in[n,n+1]}\sup_{\vartheta\in\Theta}
\left\|
\hat\Omega_{H,T}(\vartheta)
-
\hat\Omega_{H,n}(\vartheta)
\right\|_{\op}
&\le
C\frac{\log^{q_\star}(n)}{n} 
\label{eq:pathwise-interpolation-cov}
\end{align}
for every fixed $r>0$ such that $B(\theta^\star,r)\subset K_\Theta$.  Moreover, there exists a sufficiently large fixed $L$, such that the event $\mathcal C_n$ holds eventually almost surely under $\Prob_{\theta^\star}$.

\end{lemma}

\begin{proof}
We first prove the interpolation bounds on the deterministic event
$\mathcal C_n$.  By Assumption~\ref{ass:weights},
Lemma~\ref{lem:derivative-envelope}, and the product rule, the integrands
appearing in $\Psi_T^H(\vartheta)$ and in
$\partial_\vartheta\Psi_T^H(\vartheta)$ have polynomial local-count
envelopes.  Hence, there are deterministic constants $C,m<\infty$ such that,
for $j=0,1$,
\[
    \sup_{\vartheta\in\Theta}
    \left\|\partial_\vartheta^j H(t;\vartheta)\right\|_{\op}
    +
    \sup_{\vartheta\in\Theta}
    \left\|\partial_\vartheta^j \lambda(t;\vartheta)\right\|_{\op}
    +
    \sup_{\vartheta\in\Theta}
    \left\|
    \partial_\vartheta^j
    \{H(t;\vartheta)\lambda(t;\vartheta)\}
    \right\|_{\op}
    \le
    C\{1+\mathfrak C_t^m\}.
\]
The same polynomial envelope bound holds for the integrand $H(t;\vartheta)\Lambda(t;\vartheta)H(t;\vartheta)^\top$ for the plug-in covariance, after increasing $C$ and $m$ as justified in the proof of Lemma \ref{lem:plugin-inverse-polynomial}.

On the event $\mathcal C_n$,
\[
    \sup_{0\le t\le n+1}\mathfrak C_t\le L\log(n) .
\]
Since $A>0$ is fixed, the interval $(n,n+1]$ can be covered by a
deterministic number of windows of length at most $A$ where each endpoint carries no events almost surely.  Therefore,
\[
    N((n,n+1])
    \le
    C\sup_{0\le t\le n+1}\mathfrak C_t
    \le
    C\log(n)
\]
on $\mathcal C_n$.  Similarly, $[-A,n+1]$ can be covered by $O(n)$
windows of length at most $A$, so that $N([-A,n+1])\le Cn\log(n)$
on $\mathcal C_n$.

Consequently, for $j=0,1$,
\[
\begin{aligned}
&\sup_{T\in[n,n+1]}\sup_{\vartheta}
\left\|
\partial_\vartheta^j\Psi_T^H(\vartheta)
-
\partial_\vartheta^j\Psi_n^H(\vartheta)
\right\|  \\
&\qquad\le
C\sup_{T\in[n,n+1]}
\left[
    \int_n^T \{1+\mathfrak C_t^m\}\,\dd N(t)
    +
    \int_n^T \{1+\mathfrak C_t^m\}\,\dd t
\right]  \\
&\qquad\le
C\log^{m+1}(n).
\end{aligned}
\]
Also,
\[
    \sup_{\vartheta\in\Theta}
    \left\|
    \partial_\vartheta^j\Psi_n^H(\vartheta)
    \right\|
    \le
    Cn\log^{m+1}(n),
    \qquad j=0,1.
\]
For $T\in[n,n+1]$, we thus obtain
\[
    m_T(\vartheta)-m_n(\vartheta)
    =
    \frac1T
    \{\Psi_T^H(\vartheta)-\Psi_n^H(\vartheta)\}
    +
    \left(\frac1T-\frac 1n\right)\Psi_n^H(\vartheta).
\]
Therefore,
\[
\begin{aligned}
\sup_{T\in[n,n+1]}\sup_{\vartheta\in\Theta}
\norm{m_T(\vartheta)-m_n(\vartheta)}_2
&\le
C\frac{\log^{m+1}(n)}{n}
+
C\frac{n\log^{m+1}(n)}{n^2} \\
&\le
C\frac{\log^{m+1}(n)}{n}.
\end{aligned}
\]
This proves \eqref{eq:pathwise-interpolation-map}, after renaming the exponent
as $q_\star$.  The proof of \eqref{eq:pathwise-interpolation-jac} is
identical, with $\Psi_T^H$ replaced by
$\partial_\vartheta\Psi_T^H$ since both functions satisfy the same polynomial local count envelopes. 

For the covariance interpolation, let
\[
    F(t;\vartheta)
    =
    H(t;\vartheta)\Lambda(t;\vartheta)H(t;\vartheta)^\top .
\]
The polynomial local-count envelope gives
\[\sup_{T\in[n,n+1]}\sup_{\vartheta\in\Theta}\norm{\int_n^T F(t;\vartheta)\,\dd t }_{\op} \leq C\log^{q_\star}(n)\]
and
\[\sup_{\vartheta\in\Theta}\norm{\int_0^n F(t;\vartheta)\,\dd t}_{\op}\leq Cn\log^{q_\star}(n) .
\]
Then since
\[
    \hat\Omega_{H,T}(\vartheta)
    =
    \frac{1}{T}\int_0^T F(t;\vartheta)\,\dd t,
\]
the same normalization argument gives
\[
\sup_{T\in[n,n+1]}\sup_{\vartheta\in\Theta}
\left\|
\hat\Omega_{H,T}(\vartheta)-\hat\Omega_{H,n}(\vartheta)
\right\|_{\op}
\le
C\frac{\log^{q_\star}(n)}{n}.
\]
so that equation \eqref{eq:pathwise-interpolation-cov} holds on $\calC_n$. It remains to prove eventual occurrence of $\mathcal C_n$.  By
Lemma~\ref{lem:count-truncation}, for every $c>0$, after taking $L$ large enough,
\[
    \Prob_{\theta^\star}(\mathcal C_n^c)
    \le
    Cn^{-c}
\]
for all sufficiently large $n$.  Choose $c>2$.  Then
\[
    \sum_{n=2}^\infty \Prob_{\theta^\star}(\mathcal C_n^c)<\infty.
\]
The first Borel--Cantelli lemma implies that $\mathcal C_n^c$ occurs only finitely often with probability 1.  Hence $\mathcal C_n$ eventually holds almost surely. 
\end{proof}

We now apply the previous lemmas to prove Corollary \ref{cor:almost-sure-rate}.

\begin{proof}
Choose $c>2$.  By Proposition~\ref{prop:generic-controls},
Theorem~\ref{thm:generic-rate}, and the first Borel--Cantelli lemma, there
exist deterministic constants $K_0,K_1,K_2,q_0<\infty$ and an event
$\mathcal E$ with $\Prob_{\theta^\star}(\mathcal E)=1$ such that, on
$\mathcal E$, all sufficiently large integers $n$ satisfy
\begin{align}
    \norm{m_n(\theta^\star)}_W
    &\le
    K_0\sqrt{\frac{\log(n)}{n}},
    \label{eq:as-integer-truth}\\
    \sup_{\vartheta\in\Theta}
    \norm{m_n(\vartheta)-g_H(\vartheta,\theta^\star)}_W
    &\le
    K_1\frac{\log^{q_0}(n)}{\sqrt n},
    \label{eq:as-integer-ulln}\\
    \sup_{\vartheta\in B(\theta^\star,r)}
    \norm{
    \partial_\vartheta m_n(\vartheta)
    +
    B_H(\vartheta,\theta^\star)}_{\op}
    &\le
    K_2\frac{\log^{q_0}(n)}{\sqrt n}
    \label{eq:as-integer-jac}
\end{align}
for every fixed $r>0$ such that
$B(\theta^\star,r)\subset K_\Theta$.

Intersect $\mathcal E$ with the almost-sure event from
Lemma~\ref{lem:pathwise-unit-interpolation}.  On this intersection, all
statements below hold eventually with probability 1.

First extend the integer controls to all $T\in[n,n+1]$ using Lemma~\ref{lem:pathwise-unit-interpolation} and
\eqref{eq:as-integer-truth}--\eqref{eq:as-integer-jac} imply that, for all
large $n$,
\begin{align}
    \sup_{T\in[n,n+1]}
    \norm{m_T(\theta^\star)}_W
    &\le
    K_0'\sqrt{\frac{\log(n)}{n}},
    \label{eq:as-real-truth}\\
    \sup_{T\in[n,n+1]}\sup_{\vartheta\in\Theta}
    \norm{m_T(\vartheta)-g_H(\vartheta,\theta^\star)}_W
    &\le
    K_1'\frac{\log^{q_0}(n)}{\sqrt n},
    \label{eq:as-real-ulln}\\
    \sup_{T\in[n,n+1]}
    \sup_{\vartheta\in B(\theta^\star,r)}
    \norm{\partial_\vartheta m_T(\vartheta)+B_H(\vartheta,\theta^\star)}_{\op}
    &\le
    K_2'\frac{\log^{q_0}(n)}{\sqrt n}.
    \label{eq:as-real-jac}
\end{align}

The same deterministic ball argument from Theorem~\ref{thm:generic-rate}, applied uniformly over $T\in[n,n+1]$ gives that no global minimizer of $\vartheta\mapsto \norm{m_T(\vartheta)}_W^2$ can
lie outside $B(\theta^\star,r_H/2)$, for all sufficiently large $T$, where $r_H$ and $b_H$ are the deterministic constants from
Lemma~\ref{lem:generic-local-curvature}.  

Now fix $T\in[n,n+1]$ and take
$h$ such that $\norm{h}_2\leq r_H/2$ and
$\theta^\star+h\in\Theta$.  By the fundamental theorem of calculus,
\[
\begin{aligned}
    m_T(\theta^\star+h)-m_T(\theta^\star)
    &=
    \int_0^1
    \partial_\vartheta m_T(\theta^\star+sh)h\,\dd s =
    g_H(\theta^\star+h,\theta^\star)
    +
    R_T(h),
\end{aligned}
\]
where
\[R_T(h)
    =
    \int_0^1
    \left\{
    \partial_\vartheta m_T(\theta^\star+sh)
    +
    B_H(\theta^\star+sh,\theta^\star)
    \right\}h\,\dd s .\]
By \eqref{eq:as-real-jac},
\[\norm{R_T(h)}_W\leq K_3\frac{\log^{q_0}(n)}{\sqrt n}\|h\|_2\le \frac{b_H}{4}\norm{h}_2\]

for all sufficiently large $n$.
Therefore, using Lemma~\ref{lem:generic-local-curvature} and equation
\eqref{eq:as-real-truth},
\[
\begin{aligned}
    \norm{m_T(\theta^\star+h)}_W
    &\ge
    \norm{g_H(\theta^\star+h,\theta^\star)}_W
    -
    \norm{R_T(h)}_W
    -
    \norm{m_T(\theta^\star)}_W \\
    &\ge
    \frac{3b_H}{4}\norm{h}_2
    -
    K_0'\sqrt{\frac{\log(n)}{n}} .
\end{aligned}
\]
At the truth,
\[ \norm{m_T(\theta^\star)}_W \leq K_0'\sqrt{\frac{\log(n)}{n}}.\]
Hence, if
\[\norm{h}_2>\frac{4K_0'}{b_H}\sqrt{\frac{\log(n)}{n}},\]
then
\[\norm{m_T(\theta^\star+h)}_W>\norm{m_T(\theta^\star)}_W.\]
Such a point cannot be a minimizer.  Since every minimizer has already been
localized inside $B(\theta^\star,r_H/2)$, it follows that, for all large
$n$ and all $T\in[n,n+1]$,
\[
    \norm{\hat\theta_{H,T}-\theta^\star}_2
    \le
    K\sqrt{\frac{\log(n)}{n}}.
\]
Thus, after increasing $K$, we obtain
\[
    \|\hat\theta_{H,T}-\theta^\star\|_2
    \le
    K\sqrt{\frac{\log(T)}{T}}
\]
for all sufficiently large real $T$, almost surely.  This proves equation \eqref{eq:as-rate-gmm}.

We now move on to proving equation \eqref{eq:as-rate-two-step}. By Lemma~\ref{lem:plugin-inverse-polynomial}, choosing a summable polynomial tail exponent and applying the first Borel--Cantelli lemma gives
\[\sup_{\vartheta\in\Theta}\norm{ \hat\Omega_{H,n}(\vartheta) -\bar\Omega_H(\vartheta,\theta^\star)}_{\op}\leq K_\Omega\frac{\log^{q_0}(n)}{\sqrt n}\]
for all sufficiently large integers $n$, almost surely.  By
Lemma~\ref{lem:pathwise-unit-interpolation}, this extends to
\[\sup_{T\in[n,n+1]}\sup_{\vartheta\in\Theta}\norm{ \hat\Omega_{H,T}(\vartheta) -\bar\Omega_H(\vartheta,\theta^\star)}_{\op} \leq K_\Omega'\frac{\log^{q_0}n}{\sqrt n}\]
for all sufficiently large $n$, almost surely.

The assumed almost-sure preliminary rate and the local Lipschitz continuity of
$\vartheta\mapsto\bar\Omega_H(\vartheta,\theta^\star)$ yield
\[\sup_{T\in[n,n+1]}\norm{\bar\Omega_H(\tilde\theta_T,\theta^\star)
    -\Omega_H(\theta^\star)}_{\op}\leq C\sqrt{\frac{\log(n)}{n}}\implies \sup_{T\in[n,n+1]}\norm{\hat\Omega_{H,T}(\tilde\theta_T)- \Omega_H(\theta^\star)}_{\op}\longrightarrow 0\]
almost surely.

Since
\[
    \inf_{\theta\in\Theta_0}
    \lambda_{\min}\{\Omega_H(\theta)\}
    =
    \omega_0>0,\qquad
    \sup_{\theta\in\Theta_0}
    \lambda_{\max}\{\Omega_H(\theta)\}<\infty
\]
by compactness and continuity, Weyl's inequality implies that, eventually
almost surely and uniformly over $T\in[n,n+1]$,
\[
    \frac{\omega_0}{2}I_q
    \preceq
    \hat\Omega_{H,T}(\tilde\theta_T)
    \preceq
    \omega_+ I_q
\]
for some deterministic finite $\omega_+<\infty$.  Hence the inverse exists
eventually almost surely, and
\[
    \omega_+^{-1}I_q
    \preceq
    \hat W_T
    \preceq
    2\omega_0^{-1}I_q,
    \qquad
    \hat W_T
    =
    \hat\Omega_{H,T}(\tilde\theta_T)^{-1}.
\]

On this eventual event, the proof of equation \eqref{eq:as-rate-gmm} applies with the deterministic norm-equivalence constants $\omega_+^{-1}$ and $2\omega_0^{-1}$ in place of the fixed matrix $W$.  
Therefore, for some fixed $K<\infty$,
\[
    \norm{\hat\theta_{H,T}^{(2)}-\theta^\star}_2
    \le
    K\sqrt{\frac{\log T}{T}}
\]
for all sufficiently large $T>0$, almost surely.

Finally, under slight modifications and an additional interpolation lemma, the previous arguments hold under a non-stationary start after a suitable burn in. In particular, consider the event
\[\calC_{n}^{\rm burn}=\left\{\sup_{L\log(n)\leq t\leq n+1}\mathfrak C_t\leq L_1\log(n)\right\}.\]
By Lemma \ref{lem:tv-transfer}
\begin{align}\label{eq:burninevent1}
    \Prob_\theta^\calH\left(\left(\calC_{n}^{\rm burn}\right)^c\right)&\leq Cn^{-c}+ \Prob_\theta^{\rm stat}\left(\left(\calC_{n}^{\rm burn}\right)^c\right)\\ \label{eq:burninevent2}
    &\leq Cn^{-c}+\Prob_\theta^{\rm stat}\left(\sup_{0\leq t\leq n+1}\mathfrak C_t> L_1\log(n)\right)\leq 2Cn^{-c}.
\end{align}
By picking $L,L_1$ such that $c=2$ then the first Borel-Cantelli lemma applies so that $\Prob_\theta^\calH\left(\calC_{n}^{\rm burn}\text{ all large $n$}\right)=1$. Applying the same pathwise logic ensures \ref{lem:pathwise-unit-interpolation} holds on $\calC_n^{\rm burn}$ with probability 1 for all large $n$, both conditionally and unconditionally on $\calH$. Similarly the same logic ensures that the same logic holds on the lower burn-in region, in particular, the following lemma holds

\begin{lemma}[Pathwise interpolation for burn-in statistics]
\label{lem:pathwise-burnin-interpolation}
Let $b_T=A+L\log(T)$ where $L$ is selected such that equations \eqref{eq:burninevent1} and \eqref{eq:burninevent2} hold for the specific choice of $L_1$. Let $\tau_T=T-b_T$ and define
\[
    m_{b,T}(\vartheta)
    =
    \frac{1}{\tau_T}
    \left\{
    \Psi_T^H(\vartheta)-\Psi_{b_T}^H(\vartheta)
    \right\}.
\]
Similarly, define the burn-in plug-in covariance by
\[
    \hat\Omega_{H,b,T}(\vartheta)
    =
    \frac{1}{\tau_T}
    \int_{b_T}^{T}
    H(t;\vartheta)\Lambda(t;\vartheta)H(t;\vartheta)^\top\,\dd t .
\]
Let $\omega_n=\sup_{T\in[n,n+1]}|b_T-b_n|$. Then, under Assumptions~\ref{ass:parameter}--\ref{ass:stability} and
Assumption~\ref{ass:weights}, there exist constants $C<\infty$ and
$q_\star<\infty$ such that, on $\mathcal C_n^{\rm burn}$, for all sufficiently large
$n$,
\begin{align}
\sup_{T\in[n,n+1]}\sup_{\vartheta\in\Theta}
\left\|
m_{b,T}(\vartheta)-m_{b,n}(\vartheta)
\right\|_2
&\le
C
\frac{(1+\omega_n)\log^{q_\star}(n)}{n},
\label{eq:pathwise-burnin-map}
\\
\sup_{T\in[n,n+1]}
\sup_{\vartheta\in B(\theta^\star,r)}
\left\|
\partial_\vartheta m_{b,T}(\vartheta)
-
\partial_\vartheta m_{b,n}(\vartheta)
\right\|_{\op}
&\le
C
\frac{(1+\omega_n)\log^{q_\star}(n)}{n},
\label{eq:pathwise-burnin-jac}
\\
\sup_{T\in[n,n+1]}\sup_{\vartheta\in\Theta}
\left\|
\hat\Omega_{H,b,T}(\vartheta)
-
\hat\Omega_{H,b,n}(\vartheta)
\right\|_{\op}
&\le
C
\frac{(1+\omega_n)\log^{q_\star}(n)}{n}.
\label{eq:pathwise-burnin-cov}
\end{align}
\end{lemma}

\begin{proof}
    First we prove equations \eqref{eq:pathwise-burnin-map} and \eqref{eq:pathwise-burnin-jac}. To do so, by the same polynomial argument as in Lemma \ref{lem:pathwise-unit-interpolation} for $j=0,1$
    \begin{equation}\label{eq:upperinterp}
    \sup_{T\in[n,n+1]}\sup_{\vartheta\in\Theta}
\left\|\partial_\theta^j \Psi_T^H(\vartheta)-\partial_\theta^j\Psi_n^H(\vartheta)\right\|_2\le
C\frac{\log^{q_\star}(n)}{n},\end{equation}
and 
\begin{equation}\label{eq:lowerinter}\sup_{T\in[n,n+1]}\sup_{\vartheta\in\Theta}
\left\|\partial_\theta^j \Psi_{b_T}^H(\vartheta)-\partial_\theta^j\Psi_{b_n}^H(\vartheta)\right\|_2\le
C(1+\omega_n)\log^{q_\star}(n).\end{equation}
Moreover, 
\begin{equation}\label{eq:OverallInterp}\sup_{\theta\in \Theta}\norm{\partial_\theta^j \Psi_{n}^H(\vartheta)-\partial_\theta^j\Psi_{b_n}^H(\vartheta)}\leq C\tau_n\log^{q_\star}(n).
\end{equation}
where equation \eqref{eq:OverallInterp} holds since the left hand side integrates over $[b_n,n]$. Then combining that 
\[\left|\frac{1}{\tau_T}-\frac{1}{\tau_n}\right|\leq C\frac{1+\omega_n}{n^2}\]
with equations \eqref{eq:upperinterp}, \eqref{eq:lowerinter} and \eqref{eq:OverallInterp} implies equations \eqref{eq:pathwise-burnin-map} and \eqref{eq:pathwise-burnin-jac}.

To prove equation \eqref{eq:pathwise-burnin-cov} set $F(t;\theta)=H(t;\theta)\Lambda(t;\theta)H(t;\theta)^\top$. Then the same polynomial count argument gives
\[\norm{\int_n^TF(t;\theta)\dd t-\int_{b_n}^{b_T}F(t;\theta)\dd t}_{\op}\leq C(1+\omega_n)\log^{q_\star}(n)\]
on $\calC_n^{\rm burn}$ for all $n$ sufficiently large almost surely. As the right hand side of the preceding display is independent of $\theta$ and $T$ equation \eqref{eq:pathwise-burnin-cov} follows using that $\tau_T\asymp n$.
\end{proof}

By Lemma \ref{lem:tv-transfer}, the results of Proposition~\ref{prop:generic-controls} hold after a finite burn in of length $A+L\log(n)$, for $L>0$ selected as in Lemma \ref{lem:pathwise-burnin-interpolation}. Therefore,  equations \eqref{eq:as-integer-truth}-\eqref{eq:as-integer-ulln}-\eqref{eq:as-integer-jac} hold with $\partial_\theta^j m_n(\theta)$ replaced with $\partial_\theta^j m_{n,b}(\theta)$ for $j=0,1$. Lemma \ref{lem:pathwise-burnin-interpolation}, and the same argument proving equation \eqref{eq:as-rate-gmm} holds in this case.

Similarly, for the two-step estimator, applying Lemma \ref{lem:tv-transfer} to Lemma \ref{lem:plugin-inverse-polynomial} and then Lemma \ref{lem:pathwise-burnin-interpolation} ensures that the burnt in two step covariance estimator exists and for all large $n$ almost surely. The remainder follows the same argument as for the proof of equation \eqref{eq:as-rate-two-step} under a stationary start. 

Therefore, equations \eqref{eq:as-rate-gmm} and \eqref{eq:as-rate-two-step} hold under any integrable finite start to the process.
\end{proof}

\section{Likelihood localization for the MLE upper bound}\label{app:mle-upper}

This appendix proves Theorem~\ref{thm:main-rate} and Corollary~\ref{cor:mle-asymptotic-normality}. The local score, local curvature, and score CLT inputs are consequences of the generic compact-window estimating-map theory applied to the score weight $H_{\rm score}=D_\theta^\top\Lambda^{-1}$, as referenced in equations \eqref{eq:likelihoodweights} and \eqref{eq:D-Lambda-def}. However, this does not imply that the maximum likelihood estimator agrees with the zero of the score, since our optimisation objective only considers a bounded parameter set. Therefore, the likelihood-specific work is the global localization step: a uniform likelihood law and population KL separation show that the global maximizer is, with high probability, inside the local ball where the score equation and generic curvature expansion apply. In particular, this section of the supplement ensures that the maximum likelihood estimate agrees with the zero of the score in the constrained problem where $\Theta$ is compact.

First, write
\[
    s_i(t;\vartheta)=\nabla_\vartheta\log\lambda_i(t;\vartheta),
    \qquad
    H_i^\ell(t;\vartheta)=\nabla_\vartheta^2\log\lambda_i(t;\vartheta).
\]
For $\vartheta\in\Theta$ and $\theta^\star\in\Theta_0$, define the population criteria
\begin{equation*}
\mathbb L(\vartheta,\theta^\star)
=
\E_{\theta^\star}\left[
\sum_{i=1}^D
\{\lambda_i(0;\theta^\star)\log\lambda_i(0;\vartheta)-\lambda_i(0;\vartheta)\}
\right]
\end{equation*}
and
\begin{equation*}
J(\vartheta,\theta^\star)
=
\E_{\theta^\star}\left[
\sum_{i=1}^D
\left\{
\frac{\nabla\lambda_i(0;\vartheta)\nabla\lambda_i(0;\vartheta)^\top}{\lambda_i(0;\vartheta)}
-
\{\lambda_i(0;\theta^\star)-\lambda_i(0;\vartheta)\}H_i^\ell(0;\vartheta)
\right\}
\right].
\end{equation*}
Then $J(\theta,\theta)=I(\theta)$.

\begin{proposition}[Uniform likelihood law]\label{prop:uniform-likelihood-law}
Under Assumptions~\ref{ass:parameter}--\ref{ass:stability}, for every $c>0$ there are constants $C,K<\infty$ and an integer $q_0\ge1$ such that, for all sufficiently large $T$,
\begin{equation}\label{eq:uniform-likelihood-law}
\sup_{\theta^\star\in\Theta_0}
\Prob_{\theta^\star}\left(
\sup_{\vartheta\in\Theta}
\left|\frac{1}{T}\ell_T(\vartheta)-\mathbb L(\vartheta,\theta^\star)\right|
>
K\frac{\log^{q_0}T}{\sqrt T}
\right)
\le CT^{-c} .
\end{equation}
\end{proposition}

\begin{proof}
Under $\Prob_{\theta^\star}$, decompose the centred likelihood using the Doob-Meyer decomposition so that
\[\frac{1}{T}\ell_T(\vartheta)-\mathbb L(\vartheta,\theta^\star)=M_T(\vartheta)+D_T(\vartheta),\]
where
\[
    M_T(\vartheta)
    =\frac1T\sum_{i=1}^D\int_0^T\log\lambda_i(t;\vartheta)\,\dd M_i^{\theta^\star}(t)
\]
is a mean zero martingale and
\[
\begin{aligned}
    D_T(\vartheta)
    &=\frac1T\int_0^T
    \sum_{i=1}^D
    \{\lambda_i(t;\theta^\star)\log\left(\lambda_i(t;\vartheta)\right)-\lambda_i(t;\vartheta)\}\,\dd t
    -\mathbb L(\vartheta,\theta^\star)
\end{aligned}
\]
is the centred drift term. By Lemma~\ref{lem:derivative-envelope}, the drift integrand and its derivative in $\vartheta$ are compact-window functionals with polynomial local-count envelopes, and the assumptions on the conditional intensity. Theorem~\ref{thm:local-concentration}, applied to a fixed-dimensional parameter index, gives
\begin{equation}\label{eq:likelihood-drift-ulln}
    \sup_{\vartheta\in\Theta}|D_T(\vartheta)|
    \le K\frac{\log^{q_0}(T)}{\sqrt T}
\end{equation}
with probability at least $1-CT^{-c}$, uniformly in $\theta^\star$.

It remains to control the martingale term uniformly.  Let $\Omega_{\rm slide}$ be the event from Lemma~\ref{lem:count-truncation}. On this event, Lemma~\ref{lem:derivative-envelope} gives
\[
    \sup_{t\in[0,T],\vartheta\in\Theta,i}|\log\lambda_i(t;\vartheta)|
    \le C\log^{m_1}(T),
\]
and Lemma~\ref{lem:predictable-jumps} gives jump sizes of $ M_T(\vartheta)$ bounded by $C\frac{\log^{m_1}(T)}{T}$. The predictable variation of $M_T(\vartheta)$ is
\[
    \langle M_T(\vartheta)\rangle =\frac{1}{T^2}\sum_{i=1}^D\int_0^T\{\log\lambda_i(t;\vartheta)\}^2\lambda_i(t;\theta^\star)\,\dd t.
\]
The integrand of $ \langle M_T(\vartheta)\rangle$ is again a compact-window functional with polynomial envelope, so Theorem~\ref{thm:local-concentration} and the fixed-window moment bound imply that, uniformly over $\vartheta$, this bracket is at most $C\log^{m_2}(T)/T$ with probability at least $1-CT^{-c}$.  Applying the Lemma~\ref{lem:stopped-dzv} for each fixed $\vartheta$ gives
\[\Prob_{\theta^\star}\left(|M_T(\vartheta)|>K\frac{\log^{q_1}(T)}{\sqrt T}\right)\leq C T^{-c-p-5}\]
for $K,q_1$ and $T$ all sufficiently large.

Next, fix a $T^{-1/2}$-net $\calN_T$ of $\Theta$, with $|\calN_T|\le C T^{p/2}$, and union bound the preceding display over $\calN_T$.  To pass from the net to all of $\Theta$, write $\dd M_i^{\theta^\star}=\dd N_i-\lambda_i(t;\theta^\star)\dd t$ and use the mean-value theorem:
\[
\begin{aligned}
    |M_T(\vartheta)-M_T(\eta)|
    &\le
    \norm{\vartheta-\eta}_2
    \sup_{\zeta\in\Theta}
    \frac1T\sum_{i=1}^D\int_0^T
    \norm{\nabla_\zeta\log\lambda_i(t;\zeta)}_2\,\dd N_i(t) \\
    &\quad+
    \norm{\vartheta-\eta}_2
    \sup_{\zeta\in\Theta}
    \frac1T\sum_{i=1}^D\int_0^T
    \norm{\nabla_\zeta\log\lambda_i(t;\zeta)}_2
    \lambda_i(t;\theta^\star)\,\dd t .
\end{aligned}
\]
 On $\Omega_{\rm slide}$ and the total-count event \eqref{eq:total-count-truncation}, Lemma~\ref{lem:derivative-envelope} implies that 
\[|M_T(\vartheta)-M_T(\eta)|\leq K\frac{\log^m(T)}{\sqrt{T}}\]
in a similar fashion to equation \eqref{eq:MGLipschitzBound}. Thus,
\begin{equation}\label{eq:likelihood-martingale-ulln}
    \sup_{\vartheta\in\Theta}|M_T(\vartheta)|
    \le K\frac{\log^{q_2}T}{\sqrt T}
\end{equation}
with polynomially high probability, uniformly in $\theta^\star$, by the same logic as equation \eqref{eq:epsilonnetextension}. Combining \eqref{eq:likelihood-drift-ulln} and \eqref{eq:likelihood-martingale-ulln} proves \eqref{eq:uniform-likelihood-law}.
\end{proof}

\begin{proposition}[Score-weight controls and likelihood localization]\label{prop:mle-localization-inputs}
Under Assumptions~\ref{ass:parameter}--\ref{ass:identifiability}, the following hold for all sufficiently large $T$
\begin{enumerate}[label=(\roman*)]
\item For every $c>0$, there exists positive constants $C,K<\infty$
\[
\sup_{\theta^\star\in\Theta_0}
\Prob_{\theta^\star}\left(
\norm{\frac{1}{T}\nabla\ell_T(\theta^\star)}_2
>
K\sqrt{\frac{\log(T)}{T}}
\right)
\le CT^{-c}.
\]
\item There exist $r_0>0$, $C>0$ and $\kappa>0$ such that, for every $c>0$,
\[
\sup_{\theta^\star\in\Theta_0}
\Prob_{\theta^\star}\left(
\inf_{\vartheta\in B(\theta^\star,r_0)}
\lambda_{\min}\left\{-\frac{1}{T}\nabla^2\ell_T(\vartheta)\right\}
<\kappa/4
\right)
\le CT^{-c}.
\]
\item For every $\epsilon>0$ and every $c>0$, there exists a finite $C>0$ such that
\[
\sup_{\theta^\star\in\Theta_0}
\Prob_{\theta^\star}\left(
\norm{\hat\theta_T-\theta^\star}_2>\epsilon
\right)
\le CT^{-c}
\]
for all sufficiently large $T$.
\end{enumerate}
\end{proposition}

\begin{proof}
For the score weight $S(t;\vartheta)=D_\theta(t;\vartheta)^\top\Lambda(t;\vartheta)^{-1}$, Lemma~\ref{lem:derivative-envelope} and the lower bound on intensities show admissibility in the sense of Assumption~\ref{ass:weights}. Moreover,
\[
    \Psi_T^S(\vartheta)=\nabla\ell_T(\vartheta),
    \qquad
    \partial_\vartheta\left\{\frac{1}{T}\Psi_T^S(\vartheta)\right\}=\frac{1}{T}\nabla^2\ell_T(\vartheta),
\]
and the corresponding population Jacobian satisfies $B_S(\vartheta,\theta^\star)=J(\vartheta,\theta^\star)$. Part (i) follows directly from Proposition~\ref{prop:generic-controls}, specifically, equation \eqref{eq:generic-truth-bound}.

For part (ii), Proposition~\ref{prop:generic-controls}, specifically equation \eqref{eq:generic-jacobian-control} gives, uniformly over a fixed local ball,
\[
\sup_{\vartheta\in B(\theta^\star,r)}
\left\|
\frac{1}{T}\nabla^2\ell_T(\vartheta)+J(\vartheta,\theta^\star)
\right\|_{\op}
=
O_{\Prob_{\theta^\star}}\left(\frac{\log^{q_0}T}{\sqrt T}\right),
\]
with polynomially high probability uniformly over $\theta^\star\in\Theta_0$.  The map $(\vartheta,\theta)\mapsto J(\vartheta,\theta)$ is continuous by Proposition~\ref{prop:finite-window-continuity} and the derivative envelopes.  Since $J(\theta,\theta)=I(\theta)$ and the Fisher information is uniformly nonsingular on $\Theta_0$, there is a $\kappa>0$ such that $I(\theta)\succeq \kappa I_p$ on $\Theta_0$.  Compactness gives $r_0>0$ such that $J(\vartheta,\theta)\succeq(\kappa/2)I_p$ whenever $\theta\in\Theta_0$ and $\|\vartheta-\theta\|_2\le r_0$.  The stochastic bound then implies part (ii) .

For part (iii), Proposition~\ref{prop:uniform-likelihood-law} gives a uniform law of large numbers for the likelihood.  The contrast is uniquely maximized at $\theta^\star$, since
\[
\mathbb L(\vartheta,\theta^\star)-\mathbb L(\theta^\star,\theta^\star)
=
-\sum_{i=1}^D\E_{\theta^\star}\left[
\lambda_i(0;\theta^\star)
\left\{
\frac{\lambda_i(0;\vartheta)}{\lambda_i(0;\theta^\star)}-1
-
\log\left(\frac{\lambda_i(0;\vartheta)}{\lambda_i(0;\theta^\star)}
\right)\right\}
\right].
\]
The expectation is non-negative and only vanishes when $\vartheta=\theta^\star$ by Assumption~\ref{ass:identifiability}.  Proposition~\ref{prop:finite-window-continuity} gives continuity of $\mathbb L$, and compactness gives the uniform separation
\[
\inf_{\theta^\star\in\Theta_0}
\inf_{\|\vartheta-\theta^\star\|_2\ge\epsilon}
\{\mathbb L(\theta^\star,\theta^\star)-\mathbb L(\vartheta,\theta^\star)\}>0 .
\]
The uniform likelihood law then implies consistency with polynomially high probability, by exactly the same argument preceding equation \eqref{eq:consistencygeneralH}, which yields exactly (iii).
\end{proof}

\subsection{Proof of Theorem~\ref{thm:main-rate} (MLE high-probability rate)}\label{app:proof:thm:main-rate}

\begin{proof}[Proof of Theorem~\ref{thm:main-rate}]
For any $c>0$ let $K_0$ be selected large enough so that  Proposition~\ref{prop:mle-localization-inputs}(i) applies with $K\mapsto K_0$. Therefore, for all sufficiently large $T$, $8 K_0\sqrt{\frac{\log(T)}{\kappa^2T}}<r_0/2$. Next, intersect the high-probability events from Proposition~\ref{prop:mle-localization-inputs} on which
\[
    \norm{\frac{1}{T}\nabla\ell_T(\theta^\star)}_2\le K_0\sqrt{\frac{\log(T)}{T}},
    \qquad
    \inf_{\vartheta\in B(\theta^\star,r_0)}
    \lambda_{\min}\left\{-\frac{1}{T}\nabla^2\ell_T(\vartheta)\right\}\ge\kappa/4,
\]
and $\hat\theta_T\in B(\theta^\star,r_0/2)$.  For any $h\in \R^p$ with $\norm{h}_2\le r_0/2$, Taylor's formula gives
\[
\frac{1}{T}\{\ell_T(\theta^\star+h)-\ell_T(\theta^\star)\}
\le
    K_0\sqrt{\frac{\log(T)}{T}}\norm{h}_2-\frac{\kappa}{8}\norm{h}_2^2 .
\]
Therefore every $h$ satisfying $\norm{h}_2>8 K_0\sqrt{\frac{\log(T)}{\kappa^2T}}$ and $\norm{h}_2\leq r_0/2$ has likelihood strictly smaller than that of $\theta^\star$.  Since the global maximizer is already localized inside $B(\theta^\star,r_0/2)$,
\[
    \norm{\hat\theta_T-\theta^\star}_2
    \le
    \frac{8K_0}{\kappa}\sqrt{\frac{\log(T)}{T}}
\]
on the same event.  This proves the theorem under a stationary start after enlarging constants.  The bound transfers to a non-stationary start satisfying the finite pre-history condition by the same burn-in argument given at the end of the proof of Theorem~\ref{thm:generic-rate}.
\end{proof}

\subsection{Proof of Corollary~\ref{cor:mle-asymptotic-normality} (MLE asymptotic normality)}\label{app:proof:cor:mle-asymptotic-normality}

\begin{proof}[Proof of Corollary~\ref{cor:mle-asymptotic-normality}]
Let
\[
S(t;\vartheta)=D_\theta(t;\vartheta)^\top\Lambda(t;\vartheta)^{-1}
\]
be the score weight.  Lemma~\ref{lem:derivative-envelope} and the uniform lower bound on intensities show that $S$ satisfies Assumption~\ref{ass:weights}.  Since $\Psi_T^S(\theta^\star)=\nabla\ell_T(\theta^\star)$ and $\Omega_S(\theta^\star)=I(\theta^\star)$, Proposition~\ref{prop:generic-map-clt} gives
\begin{equation}\label{eq:mle-score-clt-supp}
T^{-1/2}\nabla\ell_T(\theta^\star)
\Rightarrow
N_p\{0,I(\theta^\star)\} .
\end{equation}

By Theorem~\ref{thm:main-rate}, $\hat\theta_T-\theta^\star=o_{\Prob_{\theta^\star}}(1)$, and since $\Theta_0\Subset\Theta^\circ$, the maximizer is an interior point with probability tending to one.  Thus $\nabla\ell_T(\hat\theta_T)=0$ with probability tending to one.  Let $\rho_T=K\log(T)/\sqrt T$. Theorem \ref{thm:main-rate} gives $\|\hat\theta_T-\theta^\star\|_2\le\rho_T$ with probability tending to one.  On the ball $B(\theta^\star,\rho_T)$, Proposition~\ref{prop:generic-controls} applied to the score weight gives
\[
\sup_{\vartheta\in B(\theta^\star,\rho_T)}
\left\|\frac{1}{T}\nabla^2\ell_T(\vartheta)+J(\vartheta,\theta^\star)\right\|_{\op}
=o_{\Prob_{\theta^\star}}(1).
\]
The continuity of $(\vartheta,\theta)\mapsto J(\vartheta,\theta)$, proved through Proposition~\ref{prop:finite-window-continuity}, and $J(\theta^\star,\theta^\star)=I(\theta^\star)$, imply
\begin{equation}\label{eq:mle-hessian-local-convergence-supp}
\sup_{\vartheta\in B(\theta^\star,\rho_T)}
\left\|\frac{1}{T}\nabla^2\ell_T(\vartheta)+I(\theta^\star)\right\|_{\op}
=o_{\Prob_{\theta^\star}}(1).
\end{equation}

A Taylor expansion of the score around $\theta^\star$ gives, on the event that $\hat\theta_T$ is interior,
\[0=T^{-1/2}\nabla\ell_T(\theta^\star)+\left[\int_0^1\frac1T\nabla^2\ell_T\{\theta^\star+s(\hat\theta_T-\theta^\star)\}\,\dd s\right]\sqrt T(\hat\theta_T-\theta^\star).\]
The whole segment lies in $B(\theta^\star,\rho_T)$ with probability tending
to one, and \eqref{eq:mle-hessian-local-convergence-supp} gives
\[\int_0^1\frac1T\nabla^2\ell_T\{\theta^\star+s(\hat\theta_T-\theta^\star)\}\,\dd s=-I(\theta^\star)+o_{\Prob_{\theta^\star}}(1).\]
Since $I(\theta^\star)$ is nonsingular, inversion and \eqref{eq:mle-score-clt-supp} yields
\[\sqrt T(\hat\theta_T-\theta^\star)=I(\theta^\star)^{-1}T^{-1/2}\nabla\ell_T(\theta^\star)+o_{\Prob_{\theta^\star}}(1),\]
and the displayed normal limit follows by Slutsky's theorem and the same burn-in argument as given at the end of the proof of Theorem \ref{thm:generic-an}.
\end{proof}

\section{Likelihood lower bound}\label{app:lan-lower}

We conclude the general theoretical discussion of the supplement by demonstrating the $T^{-1/2}$ rate cannot be uniformly improved by any (measurable) estimator. To do so, we first prove a technical lemma about the Kullback-Leibler (KL) divergence between two ``well-behaved'' conditional intensity functions, conditional on a locally finite history. Using this result, we are able to apply Le Cam's two-point lemma \citep[e.g.][Chapter 15]{wainwright2019} and thus prove Theorem \ref{thm:lower-bound}. 

\begin{lemma}[Conditional KL on the likelihood interval]\label{lem:conditional-kl}
Let $K\Subset\Thetao$.  For a locally finite history $\calH$ on $[-A,0)$, let $\Prob_{\theta,\calH}^{[0,T]}$ be the law on $[0,T]$ of the multivariate point process whose predictable intensity, given $\calH$ and its own past $\omega|_{[0,t)}$, is the general compact-memory intensity
\begin{align*}
\lambda_{i}^{\calH,\omega}(t;\theta)
&=
\phi_{\theta,i}\left(
\nu_{\theta,i}
+
\sum_{j=1}^D
\int_{[t-A,t)}
h_{\theta,ij}(t-s)\,d(\calH_j+\omega_j)(s)
\right),
\qquad 0\le t\le T .
\end{align*}
Here $\calH_j$ is understood as a measure on $[-A,0)$ and $\omega_j$ as a measure on $[0,T]$.  This family is a regular conditional version of the stationary future law: if $\calH=N|_{[-A,0)}$ under $\Prob_\theta$, then
\[
\Prob_\theta\big(N|_{[0,T]}\in\cdot\mid \calH\big)=\Prob_{\theta,\calH}^{[0,T]}(\cdot)
\qquad \Prob_\theta\text{-a.s.}
\]
There is a constant $C_K<\infty$ such that, for all $\theta,\theta'\in K$ and all $T\ge1$, if $\calH=N|_{[-A,0)}$ has law $\Prob_\theta^{[-A,0)}$, then
\begin{equation*}
\E_{\theta}\left[
\KL\left(\Prob_{\theta,\calH}^{[0,T]}\,\middle\|\,\Prob_{\theta',\calH}^{[0,T]}\right)
\right]
\le
C_K T\norm{\theta-\theta'}_2^2.
\end{equation*}
\end{lemma}

\subsection{Proof of Lemma~\ref{lem:conditional-kl} (Conditional KL on the likelihood interval)}\label{app:proof:lem:conditional-kl}

\begin{proof}
The existence of a regular conditional version follows from the compactly supported finite-memory property and the existence of regular conditional distributions on the standard Borel configuration spaces; see \citet[Lemma 1.16 (i)]{kallenberg2017}.

For a fixed history $\calH$, the standard likelihood and KL formula for point processes with strictly positive predictable intensities gives, conditional on $\calH$,
\[
\KL\!\left(\Prob_{\theta,\calH}^{[0,T]}\middle\|\Prob_{\theta',\calH}^{[0,T]}\right)
=
\E_{\theta,\calH}\left[
\sum_{i=1}^D\int_0^T
\left\{
\lambda_{i}^{\calH,N}(t;\theta)
\log\left(\frac{\lambda_{i}^{\calH,N}(t;\theta)}{\lambda_{i}^{\calH,N}(t;\theta')}\right)
-\lambda_{i}^{\calH,N}(t;\theta)
+\lambda_{i}^{\calH,N}(t;\theta')
\right\}\,\dd t\right],
\]
see \citet[Vol.~I, Secs.~7.2--7.3]{daley2003}. Let $\underline\lambda>0$ be the uniform lower bound on the intensities over $K_\Theta$.  Then
\[
x\log(x/y)-x+y \le \underline\lambda^{-1}(x-y)^2,
\qquad x,y\ge\underline\lambda .
\]
By the mean-value theorem on $K_\Theta$ and Lemma~\ref{lem:derivative-envelope}, for the concatenated path under $\Prob_{\theta,\calH}^{[0,T]}$,
\[
\big|\lambda_{i}^{\calH,N}(t;\theta)
-\lambda_{i}^{\calH,N}(t;\theta')\big|
\le
C_K\norm{\theta-\theta'}_2\{1+N([t-A,t))^m\}.
\]
Averaging first over $\calH\sim \Prob_\theta^{[-A,0)}$ and then over $\Prob_{\theta,\calH}^{[0,T]}$ is the same as averaging the concatenated path under the stationary law $\Prob_\theta$ on $[-A,T]$ by the tower property.  Stationarity and the fixed-window exponential count moment \eqref{eq:count-mgf} therefore give
\[
\E_\theta\left[\KL\!\left(\Prob_{\theta,\calH}^{[0,T]}\middle\|\Prob_{\theta',\calH}^{[0,T]}\right)\right]
\leq C_K\norm{\theta-\theta'}_2^2
\E_\theta\left[
\int_0^T \{1+N([t-A,t))^{2m}\}\,\dd t\right]
\le C_K T \norm{\theta-\theta'}_2^2,
\]
which is as required.
\end{proof}

\subsection{Proof of Theorem~\ref{thm:lower-bound}}

We are now ready to prove the rate-level minimax lower bound.

\begin{proof}
Since $B(\theta_0,r)\subset\Theta_0$, choose
\[
K:=\overline{B(\theta_0,r/2)}\subset\Theta_0,\qquad K\Subset\Theta^\circ,
\]
and define the empty pre-sample event
\[
E_0:=\{N|_{[-A,0)}=\emptyset\}.
\]
By the finite-memory property, and the uniform fixed-window moment bounds, there exists $p_0>0$ such that
\[
\inf_{\theta\in K}\Prob_\theta(E_0)\ge p_0.
\]
Indeed, conditional on the past before $-A$, the probability of no accepted point in $[-A,0)$ is $e^{-\Lambda^0_\theta}$, where $\Lambda^0_\theta=\int_{-A}^0\lambda(t;\theta)dt$ is the conditional total hazard with no points accepted in $[-A,0)$. Polynomial local-count envelopes and compact memory give
\[\Lambda^0_\theta\le C_E\{1+N([-2A,-A))^{m_0}\}\]
for some $m_0<\infty$, uniformly over $\theta\in K$.  Hence, using Jensen's inequality and the fixed-window moment bound,
\[
\Prob_\theta(E_0)=\E_\theta e^{-\Lambda^0_\theta}
\ge e^{-\E_\theta\Lambda^0_\theta}
\ge p_0>0
\]
uniformly over $\theta\in K$.

Let $C_K$ be the constant in Lemma~\ref{lem:conditional-kl} for this compact set $K$, and set $C_\varnothing=C_K/p_0$.  Choose $a_0>0$ so small that
\[
    a_0\le r/4,
    \qquad
    4C_\varnothing a_0^2\le \frac12 .
\]
Fix $0<a\le a_0$ and a unit vector $v\in\mathbb R^p$, and define
\[
\theta_1=\theta_0+\frac{2a}{\sqrt T}v .
\]
Then $\theta_1\in K$ for every $T\ge1$, since $\|\theta_1-\theta_0\|_2\le2a\le r/2$.

Define the conditional future laws from the empty history by
\[
Q_0:=\Prob_{\theta_0,\varnothing}^{[0,T]},
\qquad
Q_1:=\Prob_{\theta_1,\varnothing}^{[0,T]}.
\]
Since $E_0$ is the atom $\{\calH=\varnothing\}$ of the history distribution and has probability at least $p_0$ under $\Prob_{\theta_0}$, Lemma~\ref{lem:conditional-kl} gives
\[
p_0\KL(Q_0\|Q_1)
\le
\E_{\theta_0}\left[
\KL\left(
\Prob_{\theta_0,\calH}^{[0,T]}
\middle\|
\Prob_{\theta_1,\calH}^{[0,T]}
\right)
\right]
\le
C_KT\|\theta_1-\theta_0\|_2^2 .
\]
Therefore
\[
\KL(Q_0\|Q_1)
\le
4C_\varnothing a^2
\le \frac12 .
\]

Let $\tilde\theta_T$ be any estimator measurable with respect to $N|_{[-A,T]}$.  On $E_0$, the history is fixed and empty, so the estimator is a measurable function of the future path only.  Set $s_T=a/\sqrt T$.  Since $\|\theta_1-\theta_0\|_2=2s_T$, the two-point testing reduction \citep[e.g.][Chapter 15]{wainwright2019} gives
\[
\max_{\ell=0,1}
Q_\ell\left(
\|\tilde\theta_T(\varnothing,N)-\theta_\ell\|_2\ge s_T
\right)
\ge
\frac12\{1-d_{\rm TV}(Q_0,Q_1)\}.
\]
By Pinsker's inequality,
\[
d_{\rm TV}(Q_0,Q_1)
\le
\sqrt{\frac12\KL(Q_0\|Q_1)}
\le
\frac12,
\]
and hence the preceding maximum is at least $1/4$.

For $\ell=0,1$, conditioning on the empty pre-sample event gives
\[
\Prob_{\theta_\ell}\left(
\|\tilde\theta_T-\theta_\ell\|_2\ge s_T
\right)
\ge
\Prob_{\theta_\ell}(E_0)
Q_\ell\left(
\|\tilde\theta_T(\varnothing,N)-\theta_\ell\|_2\ge s_T
\right).
\]
Since $\theta_0,\theta_1\in K$, $\Prob_{\theta_\ell}(E_0)\ge p_0$.  Thus
\[
\sup_{\theta\in\Theta_0}
\Prob_\theta\left(
\|\tilde\theta_T-\theta\|_2\ge \frac{a}{\sqrt T}
\right)
\ge
\frac{p_0}{4}.
\]
Taking the infimum over all estimators proves \eqref{eq:lower-bound} with $\epsilon_0=p_0/4$ and, for instance, $T_0=1$.
\end{proof}

\section{Identifiability and moment-identification examples}
\label{app:trunc-exp-info}

In this section of the supplement we verify that the examples given in Section \ref{sec:examples} satisfy our standing assumptions.

\begin{proposition}[Identifiability and information for positive-active truncated exponentials]
\label{prop:trunc-exp-ident-info}
Consider the active scalar truncated-exponential family of Example~\ref{ex:trunc-exp}, with each active amplitude restricted to a compact subinterval of $(0,\infty)$. Then the model is identifiable in the sense of
Assumption~\ref{ass:identifiability}. Moreover, for every
$\theta$ in this active compact interior model, the Fisher information matrix
$I(\theta)$ is positive definite. Consequently, for every compact
$\Theta_0\Subset\Thetao$ contained in this model,
\[
\inf_{\theta\in\Theta_0}\lambda_{\min}\{I(\theta)\}>0 .
\]
\end{proposition}

\begin{proof}

The proof identifies parameters by matching the jump sizes of the intensities $\lambda_i(\cdot;\theta)$, so we begin by recording why those jumps are well defined and why an intensity difference that vanishes for Lebesgue-almost every $t$ must in fact vanish at every time, jumps included.

Each intensity, and each of its parameter derivatives used below, is an affine function of finite-memory convolutions of the form $Y_f(t)=\int_{[t-A,t)} f(t-s)\,\dd N_j(s)$ with $f$ a smooth kernel, as in Lemma~\ref{lem:predictable-jumps}(i). As $t$ increases, the events contributing to $Y_f(t)$ change only when an event of $N_j$ at some time $s$ enters the window $[t-A,t)$, at $t=s$, or leaves it, at $t=s+A$; call these the entry and exit times. Between two consecutive entry/exit times $Y_f$ is a fixed finite sum of smooth terms $f(t-s_\ell)$, hence $C^1$, and its only discontinuities are the one-sided jumps at those times.

We work on the full-probability event of Lemma~\ref{lem:predictable-jumps}(ii): on every compact interval, distinct components never jump simultaneously, no event falls on a prescribed deterministic time, and no two events are separated by exactly $A$. On this event all entry and exit times are distinct and isolated, so any finite linear combination of such convolutions is $C^1$ on each open interval between consecutive entry/exit times. Being continuous there, it is identically zero on each such interval as soon as it is zero for Lebesgue-almost every $t$; and since its one-sided limits at every entry or exit time are then limits of zero, each entry and exit jump vanishes as well. These vanishing jumps are the identities exploited below.

First we prove identifiability. Suppose that, under $\Prob_\theta$,
\[
\lambda_i(t;\theta)=\lambda_i(t;\theta')
\quad\text{for all }i\text{ and for Lebesgue-a.e. }t,
\quad \Prob_\theta\text{-a.s.}
\]
We also intersect the preceding path-regularity event with the event on which
this equality of intensities holds. Fix a receiving component $i$ and a
source component $j$. Since the stationary intensity of component $j$ is
positive, $N_j$ has infinitely many jumps on the real line almost surely.
Let $s$ be a jump time of $N_j$.

By the no-simultaneous-jump and no-exact-$A$-separation properties, no other
event enters a predictable window at time $s$, and no event exits a
predictable window at time $s$. Therefore the only right discontinuity of
$\lambda_i(\cdot;\theta)-\lambda_i(\cdot;\theta')$ at $s$ caused by this
source event is its entry into the $j$-convolution. Since the intensity
difference is zero on the continuity intervals immediately before and after
$s$, this right jump must vanish. Thus
\[
\alpha_{ij}c_{ij}(\beta_{ij})
-
\alpha'_{ij}c_{ij}(\beta'_{ij})
=
0 .
\]
Similarly, at time $s+A$, the same source event is the only event exiting
the $j$-convolution, and no other event enters or exits at that time. Hence
the right jump at $s+A$ must also vanish, which gives
\[
\alpha_{ij}c_{ij}(\beta_{ij})e^{-\beta_{ij}A}
-
\alpha'_{ij}c_{ij}(\beta'_{ij})e^{-\beta'_{ij}A}
=
0 .
\]
Both factors $\alpha_{ij}c_{ij}(\beta_{ij})$ and
$\alpha'_{ij}c_{ij}(\beta'_{ij})$ are nonzero by the active-amplitude restriction. Dividing the second
display by the first yields
\[
e^{-\beta_{ij}A}=e^{-\beta'_{ij}A},
\]
and hence $\beta_{ij}=\beta'_{ij}$. Substituting back into the first display
gives $\alpha_{ij}=\alpha'_{ij}$. Since $i$ and $j$ were arbitrary, all
active amplitude and decay coordinates agree. The equality of intensities then
reduces to $\mu_i=\mu_i'$ for every $i$. Thus $\theta=\theta'$, proving
identifiability.

We now prove positive definiteness of $I(\theta)$. Fix
$\theta$ in the active compact interior model. It is enough to show that,
for every $v\in\mathbb R^p$,
\[
v^\top I(\theta)v=0
\quad\Longrightarrow\quad
v=0 .
\]
For each receiving component $i$, define the directional derivative of the
intensity
\[
D_i^v(t)
=
v_{\mu_i}
+
\sum_{j=1}^D
v_{\alpha_{ij}}
\int_{[t-A,t)}
k_{ij}(t-s;\beta_{ij})\,\dd N_j(s)
+
\sum_{j=1}^D
\alpha_{ij}v_{\beta_{ij}}
\int_{[t-A,t)}
\partial_\beta k_{ij}(t-s;\beta_{ij})\,\dd N_j(s).
\]
Then $D_i^v(t)=v^\top\nabla_\theta\lambda_i(t;\theta)$. Hence, if
$v^\top I(\theta)v=0$, then
\[
0
=
\E_\theta\left[
\sum_{i=1}^D
\frac{\{D_i^v(0)\}^2}{\lambda_i(0;\theta)}
\right].
\]
Since $\lambda_i(0;\theta)>0$, it follows that
\[
D_i^v(0)=0
\qquad\text{almost surely for every }i .
\]
By stationarity, $D_i^v(t)\stackrel{d}{=}D_i^v(0)$ for
every fixed $t$. Hence, for every integer $m$,
\[\E_\theta\left[\int_m^{m+1}\1\{D_i^v(t)\ne0\}\diff t\right]=\int_m^{m+1}\Prob_\theta\{D_i^v(t)\ne0\}\diff t=0.
\]
Thus, after taking a countable intersection over $m\in\mathbb Z$,
\[
D_i^v(t)=0
\quad\text{for Lebesgue-a.e. }t\in\mathbb R,
\qquad \Prob_\theta\text{-a.s.}
\]

We now use the path regularity of the compactly supported convolutions. On every
compact interval, the discontinuity times of $D_i^v$ are contained in the
finite set of event times and event times plus $A$. Between these times
$D_i^v$ is continuously differentiable. Therefore, since $D_i^v$ vanishes
for Lebesgue-a.e. time, it must vanish identically on each open continuity
interval. At any entry or exit time $r$, left-continuity gives
$D_i^v(r)=0$, while the right-hand continuity interval gives
$D_i^v(r+)=0$. Hence every right jump
\[
\Delta_+D_i^v(r):=D_i^v(r+)-D_i^v(r)
\]
vanishes.

Fix $i,j$, and let $s$ be a jump time of $N_j$. By the same
path-regularity properties used in the identifiability argument, the only
right jump of $D_i^v$ at $s$ caused by this source event is its entry into
the $j$-convolution. Therefore
\[
v_{\alpha_{ij}}k_{ij}(0+;\beta_{ij})
+
\alpha_{ij}v_{\beta_{ij}}
\partial_\beta k_{ij}(0+;\beta_{ij})
=
0 .
\]
At $s+A$, the only right jump caused by the same source event is its exit
from the $j$-convolution. The exit discontinuity is determined by the
left limit of the convolution contribution as the lag approaches $A$, not by
the chosen pointwise value $k(A;\beta)$. This right jump is the negative of
\[
v_{\alpha_{ij}}k_{ij}(A-;\beta_{ij})
+
\alpha_{ij}v_{\beta_{ij}}
\partial_\beta k_{ij}(A-;\beta_{ij}),
\]
and it must vanish. Hence
\[
v_{\alpha_{ij}}k_{ij}(A-;\beta_{ij})
+
\alpha_{ij}v_{\beta_{ij}}
\partial_\beta k_{ij}(A-;\beta_{ij})
=
0 .
\]
The two equations may be written as
\[
\begin{pmatrix}
k_{ij}(0+;\beta_{ij}) &
\alpha_{ij}\partial_\beta k_{ij}(0+;\beta_{ij})\\
k_{ij}(A-;\beta_{ij}) &
\alpha_{ij}\partial_\beta k_{ij}(A-;\beta_{ij})
\end{pmatrix}
\begin{pmatrix}
v_{\alpha_{ij}}\\
v_{\beta_{ij}}
\end{pmatrix}
=
0 .
\]
It remains only to check that this $2\times2$ matrix is nonsingular. Since
\[
k_{ij}(0+;\beta)=c_{ij}(\beta),
\qquad
k_{ij}(A-;\beta)=c_{ij}(\beta)e^{-\beta A},
\]
and
\[
\partial_\beta k_{ij}(0+;\beta)=c'_{ij}(\beta),
\qquad
\partial_\beta k_{ij}(A-;\beta)
=
\{c'_{ij}(\beta)-Ac_{ij}(\beta)\}e^{-\beta A},
\]
its determinant is
\[
\alpha_{ij}
\left[
k_{ij}(0+;\beta_{ij})
\partial_\beta k_{ij}(A-;\beta_{ij})
-
k_{ij}(A-;\beta_{ij})
\partial_\beta k_{ij}(0+;\beta_{ij})
\right]
=
-\alpha_{ij}A c_{ij}(\beta_{ij})^2e^{-\beta_{ij}A}.
\]
This is nonzero because the active amplitude $\alpha_{ij}$ is bounded away from zero, $A>0$, and
$c_{ij}(\beta_{ij})>0$. Hence
\[
v_{\alpha_{ij}}=v_{\beta_{ij}}=0 .
\]
Since $i,j$ were arbitrary, all amplitude and decay coordinates of $v$
vanish. The identity $D_i^v(t)=0$ then reduces to
\[
D_i^v(t)=v_{\mu_i}=0
\]
for every $i$. Therefore $v=0$, proving that $I(\theta)$ is positive
definite.

Finally, $\theta\mapsto I(\theta)$ is continuous on compact interior
parameter sets by Proposition~\ref{prop:finite-window-continuity}, applied to
the local-window integrand in \eqref{eq:fisher}, together with the derivative
envelope Lemma~\ref{lem:derivative-envelope}. Since
\[
(\theta,u)\mapsto u^\top I(\theta)u
\]
is continuous and strictly positive on the compact set
$\Theta_0\times S^{p-1}$, it has a positive minimum there. Equivalently,
\[
\inf_{\theta\in\Theta_0}\lambda_{\min}\{I(\theta)\}>0 .
\]
\end{proof}

\begin{proposition}[Verification for the positive-active smooth-ReLU truncated-exponential example]
\label{prop:nl-smooth-relu-trunc-exp}
The model in Example~\ref{ex:nl-smooth-relu-trunc-exp}, formulated on compact
positive active-amplitude coordinate sets satisfying the displayed spectral-radius
condition, satisfies Assumptions~\ref{ass:parameter}--\ref{ass:stability}.
Moreover, it is identifiable in the sense of Assumption~\ref{ass:identifiability},
and its Fisher information matrix is positive definite at every active
interior parameter value.  Consequently, for every compact
$\Theta_0\Subset\Theta^\circ$ contained in this active model,
\[
    \inf_{\theta\in\Theta_0}
    \lambda_{\min}\{I(\theta)\}>0 .
\]
\end{proposition}

\begin{proof}
For active edges, write
\[
    d_{ij}(\beta)=\frac{\beta}{1-e^{-\beta A}},
    \qquad
    k_{ij}(u;\beta)=d_{ij}(\beta)e^{-\beta u}\mathbf 1\{0\le u\le A\}.
\]
On every compact interval $B_{ij}\Subset(0,\infty)$, the normalising map
$\beta\mapsto d_{ij}(\beta)$ is $C^\infty$, strictly positive, and has
uniformly bounded derivatives of all orders.  Hence
\[
    h_{\theta,ij}(u)=\gamma_{ij}k_{ij}(u;\beta_{ij})
\]
is compactly supported, nonnegative, strictly positive on $[0,A]$ for active
edges, and has uniformly bounded parameter derivatives up to order four on
the active compact parameter set.  Inactive edges are fixed at
$h_{\theta,ij}\equiv0$ and are not part of the active coordinate system.

The link
\[
    \Phi_i(x)
    =
    \epsilon_i+
    \frac{a_i}{b_i}\log\{1+\exp(b_i(x-c_i))\}
\]
is $C^\infty$, positive, non-affine, and globally Lipschitz.  Its first
derivative is
\[
    \Phi_i'(x)
    =
    \frac{a_i}{1+\exp\{-b_i(x-c_i)\}},
\]
so
\[
    0<
    \frac{a_i}{1+\exp\{-b_i(\underline\nu_i-c_i)\}}
    \le
    \Phi_i'(x)
    \le a_i,
    \qquad x\ge\underline\nu_i .
\]
All higher derivatives of $\Phi_i$ are bounded globally.  Since
$\nu_i\in[\underline\nu_i,\bar\nu_i]\Subset(0,\infty)$, the baseline link
values $\Phi_i(\nu_i)$ are uniformly bounded above and below.  Thus the
positivity, smoothness, and link-regularity requirements in
Assumption~\ref{ass:parameter} hold, with $\phi_{\theta,i}=\Phi_i$.

Let
\[Y_i(t;\theta)=\sum_{j:(i,j)\in\mathcal E}\int_{[t-A,t)}h_{\theta,ij}(t-s)\,\dd N_j(s).\]
Because $Y_i(t;\theta)\ge0$ and $\Phi_i$ is $a_i$-Lipschitz,
\[\lambda_i(t;\theta)=\Phi_i\{\nu_i+Y_i(t;\theta)\}\le\Phi_i(\nu_i)+a_iY_i(t;\theta).\]
Thus the nonlinear intensity is dominated by the linear Hawkes
intensity with baseline $\eta_{\theta,i}=\Phi_i(\nu_i)$ and excitation
kernels $a_i h_{\theta,ij}$.  Since the truncated-exponential filters are
normalized,
\[
    \int_0^A h_{\theta,ij}(u)\,\dd u
    =
    \gamma_{ij}\mathbf 1\{(i,j)\in\mathcal E\}.
\]
The corresponding Lipschitz stability matrix is therefore
\[
    G(\theta)_{ij}
    =
    a_i\gamma_{ij}\mathbf 1\{(i,j)\in\mathcal E\}.
\]
The assumed uniform bound
\[
    \sup_{\theta\in\Theta}\rho\{G(\theta)\}\le1-\varepsilon
\]
gives the required nonlinear Lipschitz stability, stationary Poisson
embedding, and fixed-window exponential moment bounds.

The compact-window derivative structure is explicit.  For the latent inputs, with $X_i(t;\theta)=\nu_i+Y_i(t;\theta)$
\begin{align*}
    \partial_{\nu_r}X_i(t;\theta)&=\mathbf 1\{r=i\},\\
\partial_{\gamma_{ij}}X_i(t;\theta)&=\int_{[t-A,t)}k_{ij}(t-s;\beta_{ij})\,\dd N_j(s),\qquad (i,j)\in\mathcal E,\\
\partial_{\beta_{ij}}X_i(t;\theta)&=\gamma_{ij}\int_{[t-A,t)}\partial_\beta k_{ij}(t-s;\beta_{ij})\,\dd N_j(s),\qquad (i,j)\in\mathcal E.
\end{align*}
Higher parameter derivatives are finite sums of the same form with
$\partial_\beta^r k_{ij}$, $r\le4$.  Since these kernel derivatives are
uniformly bounded on $B_{ij}\times[0,A]$, all derivatives of $X_i$ up to
order four are bounded by a constant times $1+\mathfrak C_t$.  Combining this with
the bounded derivatives of $\Phi_i$ and the chain rule shows that the
intensity, score, Hessian, and higher derivatives are compact-window
functionals with polynomial local-count envelopes.  This verifies the
regularity requirements used by the likelihood and estimating-equation
arguments.

It remains to prove identifiability and positive definiteness of Fisher
information.  First suppose that, under $\Prob_\theta$,
\[
    \lambda_i(t;\theta)=\lambda_i(t;\theta')
    \quad\text{for all }i\text{ and Lebesgue-a.e. }t,
    \quad \Prob_\theta\text{-a.s.}
\]
Because each $\Phi_i$ is known and strictly increasing, this implies
\[
    X_i(t;\theta)=X_i(t;\theta')
    \quad\text{for all }i\text{ and Lebesgue-a.e. }t,
    \quad \Prob_\theta\text{-a.s.}
\]
Work on the full-probability path-regularity event from
Lemma~\ref{lem:predictable-jumps}: there are no simultaneous jumps, no jumps
at deterministic times, and no two event times are separated by exactly
$A$.  On this event, finite linear combinations of the truncated-exponential
convolutions are continuously differentiable between event times and event
times shifted by $A$.  Hence, if such a combination is zero for
Lebesgue-a.e. time, then its one-sided entry and exit jumps must vanish.

Fix an active edge $(i,j)\in\mathcal E$.  Since component $j$ has
intensity bounded below by a positive constant, $N_j$ has infinitely many jumps almost
surely on the real line and let $s$ be such a jump time of $N_j$. At time $s$,
the event enters the predictable memory window.  No other event enters or
exits at the same time on the regularity event. Therefore, the right jump of
$X_i(\cdot;\theta)-X_i(\cdot;\theta')$ at $s$ gives
\[\gamma_{ij}d_{ij}(\beta_{ij})-\gamma'_{ij}d_{ij}(\beta'_{ij})=0.\]
At time $s+A$, the same event exits the memory window.  Again, no other
event enters or exits at that time. Hence,
\[\gamma_{ij}d_{ij}(\beta_{ij})e^{-\beta_{ij}A}-\gamma'_{ij}d_{ij}(\beta'_{ij})e^{-\beta'_{ij}A}=0.\]
The first display has a strictly positive common value.  Dividing the second
display by the first gives
\[
    e^{-\beta_{ij}A}=e^{-\beta'_{ij}A},
\]
and hence $\beta_{ij}=\beta'_{ij}$.  Substitution into the first display
then gives $\gamma_{ij}=\gamma'_{ij}$.  Since the active edge was arbitrary,
all active excitation parameters agree.  The equality
$X_i(t;\theta)=X_i(t;\theta')$ then reduces to
\[
    \nu_i=\nu_i',
    \qquad i=1,\ldots,D.
\]
Thus $\theta=\theta'$, proving identifiability on the active coordinate
space.

We now prove positive definiteness of Fisher information.  Fix an active
interior parameter $\theta$, and suppose $v\in\mathbb R^p$ satisfies
\[
    v^\top I(\theta)v=0 .
\]
For each component $i$, define the directional derivative of the latent
input by
\[\dot X_i^v(t)=v_{\nu_i}+\sum_{j:(i,j)\in\mathcal E}v_{\gamma_{ij}}\int_{[t-A,t)}k_{ij}(t-s;\beta_{ij})\,\dd N_j(s) +\sum_{j:(i,j)\in\mathcal E}\gamma_{ij}v_{\beta_{ij}}\int_{[t-A,t)}\partial_\beta k_{ij}(t-s;\beta_{ij})\,\dd N_j(s).\]
Then
\[v^\top\nabla_\theta\lambda_i(t;\theta)=\Phi_i'\{X_i(t;\theta)\}\dot X_i^v(t).\]
Therefore,
\[0=v^\top I(\theta)v=\E_\theta\left[\sum_{i=1}^D\frac{\Phi_i'\{X_i(0;\theta)\}^2\{\dot X_i^v(0)\}^2}{\lambda_i(0;\theta)}\right].\]
Since $\lambda_i(0;\theta)>0$ and $\Phi_i'\{X_i(0;\theta)\}$ is bounded away from zero on the relevant state space, it follows that
\[\dot X_i^v(0)=0\qquad \Prob_\theta\text{-a.s. for every }i.\]
By stationarity and Fubini's theorem,
\[\dot X_i^v(t)=0\quad\text{for Lebesgue-a.e. }t\in\mathbb R,\qquad \Prob_\theta\text{-a.s. for every }i.\]
On the same path-regularity event used above, this implies that every
one-sided entry and exit jump of $\dot X_i^v$ vanishes.

Fix an active edge $(i,j)\in\mathcal E$, and let $s$ be a jump time of
$N_j$.  The vanishing entry jump at $s$ gives
\[v_{\gamma_{ij}}k_{ij}(0+;\beta_{ij})+\gamma_{ij}v_{\beta_{ij}}\partial_\beta k_{ij}(0+;\beta_{ij})=0.\]
The vanishing exit jump at $s+A$ gives
\[ v_{\gamma_{ij}}k_{ij}(A-;\beta_{ij}) + \gamma_{ij}v_{\beta_{ij}} \partial_\beta k_{ij}(A-;\beta_{ij})=0.\]
Because
\[ k_{ij}(0+;\beta)=d_{ij}(\beta), \qquad k_{ij}(A-;\beta)=d_{ij}(\beta)e^{-\beta A},\]
and
\[
    \partial_\beta k_{ij}(0+;\beta)=d_{ij}'(\beta),
    \qquad
    \partial_\beta k_{ij}(A-;\beta)
    =
    \{d_{ij}'(\beta)-Ad_{ij}(\beta)\}e^{-\beta A},
\]
the two equations can be written as
\[
    \begin{pmatrix}
    d_{ij}(\beta_{ij})
    &
    \gamma_{ij}d_{ij}'(\beta_{ij})
    \\
    d_{ij}(\beta_{ij})e^{-\beta_{ij}A}
    &
    \gamma_{ij}\{d_{ij}'(\beta_{ij})-Ad_{ij}(\beta_{ij})\}
    e^{-\beta_{ij}A}
    \end{pmatrix}
    \begin{pmatrix}
    v_{\gamma_{ij}}\\
    v_{\beta_{ij}}
    \end{pmatrix}
    =
    0 .
\]
The determinant of this matrix is
\[
    -\gamma_{ij}A d_{ij}(\beta_{ij})^2 e^{-\beta_{ij}A},
\]
which is nonzero because $\gamma_{ij}>0$, $A>0$, and
$d_{ij}(\beta_{ij})>0$. Hence, $v_{\gamma_{ij}}=v_{\beta_{ij}}=0$.
Since $(i,j)\in\mathcal E$ was arbitrary, all active excitation-coordinate
components of $v$ vanish.  The identity $\dot X_i^v(t)=0$ then reduces to
\[
    v_{\nu_i}=0,
    \qquad i=1,\ldots,D.
\]
Thus $v=0$.  Hence $I(\theta)$ is positive definite at every active
interior parameter value.

Finally, $\theta\mapsto I(\theta)$ is continuous on compact active interior
parameter sets.  Indeed, the Fisher integrand is a compact-window functional
with a polynomial local-count envelope, because the kernels and the
smooth-rectifier derivatives have the bounds verified above.  Proposition
\ref{prop:finite-window-continuity} therefore gives continuity.  Since
\[
    (\theta,u)\mapsto u^\top I(\theta)u
\]
is continuous and strictly positive on the compact set
$\Theta_0\times S^{p-1}$, it has a positive minimum there.  Equivalently,
\[
    \inf_{\theta\in\Theta_0}
    \lambda_{\min}\{I(\theta)\}>0 .
\]
\end{proof}

\subsection{A fully identified non-score GMM example}
\label{app:one-point-age-gmm}

We finally demonstrate a non-score GMM example in this subsection of the appendix.

\begin{proposition}[One-point age moments identify the univariate truncated-exponential model]
\label{prop:one-point-age-gmm}
Consider the univariate Hawkes model
\[
    \lambda(t;\theta)=\mu+\alpha X_\beta(t),
    \qquad
    X_\beta(t)=\int_{[t-A,t)} k_\beta(t-s)\,\dd N(s),
\]
where
\[
    k_\beta(u)=c(\beta)e^{-\beta u}\mathbf 1\{0\le u\le A\},
    \qquad
    \theta=(\mu,\alpha,\beta),
\]
where $c\in C^4([\underline\beta,\bar\beta])$ is strictly positive.
Let $\Theta_0$ be a compact interior parameter set with
$\mu>0$, $\alpha>0$, $\beta\in[\underline\beta,\bar\beta]\Subset(0,\infty)$,
and uniform subcriticality.  Let
\[
    \mathfrak C_t=N([t-A,t)),
    \qquad
    U_t=\int_{[t-A,t)}(t-s)\,\dd N(s),
\]
and define
\[
    H(t)=
    \begin{pmatrix}
        \mathbf 1\{\mathfrak C_t=0\}\\
        \mathbf 1\{\mathfrak C_t=1\}\\
        U_t\mathbf 1\{\mathfrak C_t=1\}
    \end{pmatrix}.
\]
Then $H$ is an admissible compact-window non-score weight.  Moreover, for
every $\theta^\star\in\Theta_0$, the population drift
\[
    g_H(\theta,\theta^\star)
    =
    \E_{\theta^\star}
    \left[
        H(0)\{\lambda(0;\theta^\star)-\lambda(0;\theta)\}
    \right]
\]
has the unique zero $\theta=\theta^\star$.  The uniform separation and local
rank conditions in Assumption~\ref{ass:moment-identification} hold.  
Finally,
\[\inf_{\theta\in\Theta_0}\lambda_{\min}\{\Omega_H(\theta)\}>0.\]
\end{proposition}

\begin{proof}
The weight is predictable, shift-covariant, bounded by $1+A$, and depends
only on $N|_{[t-A,t)}$.  Since it is independent of $\theta$, the
smoothness and envelope requirements in Assumption~\ref{ass:weights} are
immediate.

Fix $\theta^\star=(\mu^\star,\alpha^\star,\beta^\star)$.  Write
\[
    p_0=\Prob_{\theta^\star}(\mathfrak C_0=0),
    \qquad
    \nu(B)=\Prob_{\theta^\star}(\mathfrak C_0=1,\ U_0\in B),
    \qquad B\subset(0,A).
\]
By the Poisson embedding, positivity of the baseline intensity, and compact
memory, $p_0>0$, and $\nu$ gives positive mass to every nonempty open
subinterval of $(0,A)$.  Indeed, conditional on the history before $-A$,
the probability of no accepted Poisson marks in $[-A,0)$ is positive on
large-count truncation events, and those events have positive stationary
probability.  Similarly, for any nonempty open age interval $I\subset(0,A)$,
the event that exactly one accepted mark occurs in the corresponding time
interval $-I$ and no other accepted mark occurs in $[-A,0)$ has positive
probability.

Suppose $g_H(\theta,\theta^\star)=0$.  On $\{\mathfrak C_0=0\}$, both convolution
terms vanish, so the first coordinate gives
\[
    0=p_0(\mu^\star-\mu),
\]
and hence $\mu=\mu^\star$.  On $\{\mathfrak C_0=1\}$, the unique point has age
$U_0$, so
\[
    \lambda(0;\theta)=\mu+\alpha k_\beta(U_0).
\]
The second and third coordinates of $g_H=0$ therefore give
\[
    \int_0^A
    \{\alpha^\star k_{\beta^\star}(u)-\alpha k_\beta(u)\}\,\nu(\dd u)=0,
\]
and
\[
    \int_0^A
    u\{\alpha^\star k_{\beta^\star}(u)-\alpha k_\beta(u)\}\,\nu(\dd u)=0.
\]
Define
\[
    L_r(\beta)=\int_0^A u^r e^{-\beta u}\,\nu(\dd u),
    \qquad r=0,1.
\]
Since $k_\beta(u)=c(\beta)e^{-\beta u}$, the two equations imply
\[
    \alpha c(\beta)L_0(\beta)
    =
    \alpha^\star c(\beta^\star)L_0(\beta^\star),
    \qquad
    \alpha c(\beta)L_1(\beta)
    =
    \alpha^\star c(\beta^\star)L_1(\beta^\star).
\]
Dividing yields
\[
    m(\beta)=m(\beta^\star),
    \qquad
    m(\beta)=\frac{L_1(\beta)}{L_0(\beta)}.
\]
But
\[
    m'(\beta)
    =
    -\operatorname{Var}_{\nu_\beta}(U)<0,
    \qquad
    \nu_\beta(\dd u)=\frac{e^{-\beta u}\nu(\dd u)}{L_0(\beta)},
\]
because $\nu$ is not concentrated at one point.  Hence
$\beta=\beta^\star$, and then the first equation gives
$\alpha=\alpha^\star$.  Thus $g_H(\theta,\theta^\star)=0$ implies
$\theta=\theta^\star$.  Continuity of $g_H$ and compactness give the
uniform separation condition.

It remains to check local rank.  Under $\Prob_\theta$, let
\[
    \nu_\theta(B)=\Prob_\theta(\mathfrak C_0=1,\ U_0\in B),
    \qquad
    p_0(\theta)=\Prob_\theta(\mathfrak C_0=0).
\]
The derivative of the intensity is
\[
    D_\theta(0;\theta)
    =
    \partial_\theta\lambda(0;\theta)
    =
    \bigl(1,\ X_\beta(0),\ \alpha\,\partial_\beta X_\beta(0)\bigr).
\]
Thus
\[
    A_H(\theta)=E_\theta\{H(0)D_\theta(0;\theta)\}.
\]
The first row is $(p_0(\theta),0,0)$.  The lower $2\times2$ block in the
$(\alpha,\beta)$-columns is
\[
    \begin{pmatrix}
        M_0 & \alpha \dot M_0\\
        M_1 & \alpha \dot M_1
    \end{pmatrix},
\]
where
\[
    M_r=\int_0^A u^r k_\beta(u)\,\nu_\theta(\dd u),
    \qquad
    \dot M_r=\int_0^A u^r\partial_\beta k_\beta(u)\,\nu_\theta(\dd u).
\]
Since
\[
    \partial_\beta k_\beta(u)
    =
    k_\beta(u)\left\{\frac{c'(\beta)}{c(\beta)}-u\right\},
\]
we have
\[
    \dot M_r=\frac{c'(\beta)}{c(\beta)}M_r-M_{r+1}.
\]
Therefore
\[
    \det A_H(\theta)
    =
    p_0(\theta)\alpha\{M_1^2-M_0M_2\}.
\]
The bracketed term is strictly negative by the Cauchy Schwarz inequality, because the
probability measure proportional to $k_\beta(u)\nu_\theta(\dd u)$ is not
concentrated at a single point.  Hence $A_H(\theta)$ is nonsingular.
Continuity and compactness give the uniform local rank condition.

Next, $\Omega_H(\theta)=E_\theta\{H(0)\lambda(0;\theta)H(0)^\top\}$ is
nonsingular.  If $a^\top H(0)=0$ $\Prob_\theta$-a.s. on the support weighted by
$\lambda(0;\theta)>0$, then $\{\mathfrak C_0=0\}$ gives $a_1=0$, while
$\{\mathfrak C_0=1\}$ gives $a_2+a_3U_0=0$ $\nu_\theta$-a.e.  Since
$\nu_\theta$ is not concentrated at one point, $a_2=a_3=0$.

To conclude, the map $\theta\mapsto\Omega_H(\theta)$ is continuous by
Proposition~\ref{prop:finite-window-continuity}, since
$H(0)\lambda(0;\theta)H(0)^\top$ has a polynomial finite-window envelope.
Since $\Theta_0$ is compact and $\Omega_H(\theta)$ is positive definite
for every $\theta\in\Theta_0$, compactness gives
\[\inf_{\theta\in\Theta_0}\lambda_{\min}\{\Omega_H(\theta)\}>0.\]
\end{proof}

\section{Simulation implementation details}\label{app:simulation-implementation}

This appendix records the implementation details omitted from the main simulation section.  Paths are generated using the linear Hawkes cluster representation \citep{hawkesoakes1974}.  Immigrants are generated from time $-B_T-A$ to $T$, offspring are generated recursively using the truncated-exponential offspring density, and the retained observation is $N|_{[-A,T]}$.  In the reported design $B_T=450$, so the pre-sample interval $[-A,0)$ is available for likelihood and moment evaluation while the effect of truncating clusters born before the burn-in time is negligible relative to the Monte Carlo scale.

All criteria are optimized over the same compact parameter set, with positive baselines, positive excitation amplitudes, $\beta\in[0.25,4.50]$, and $\rho(\alpha)<0.98$.  The likelihood is fit by bounded L-BFGS-B.  The least-squares contrast \eqref{eq:ls-contrast}, reported as method \(J\), and the overidentified GMM criterion are optimized after a smooth transformation of the compact box using BFGS with a Nelder--Mead fallback.  The overidentified moment integrals are evaluated by adaptive Gauss--Legendre quadrature, split at event times and at event times plus $A$.  The plug-in covariance matrix is symmetrized and row-scaled before inversion, with a relative ridge $10^{-8}$ used only if the scaled matrix is numerically singular or too ill-conditioned.

Population Godambe matrices are estimated independently from $128$ long simulations with $T_{\rm pop}=250000$ and $250000$ random time-evaluation points per run.  These population matrices are used only for the horizontal target lines in Figure~\ref{fig:bivar-overall-efficiency} and the oracle diagnostic intervals in Appendix~\ref{app:simulation-diagnostics}.

\section{Additional simulation diagnostics}\label{app:simulation-diagnostics}

This appendix reports two supplementary diagnostics for the bivariate simulation study in Section~\ref{sec:simulation}. For each method $m\in\{{\rm M},{\rm J},{\rm O}\}$ and coordinate $k$, we form the population-target asymptotic normal interval
\[
    \hat\theta_{m,k}\pm 1.96\sqrt{\frac{(V_m)_{kk}}{T}},
\]
where \(V_{\rm M}=I(\theta_0)^{-1}\), \(V_{\rm J}=A_{\rm LS}(\theta_0)^{-1}\Omega_{\rm LS}(\theta_0)A_{\rm LS}(\theta_0)^{-1}\) is the local least-squares sandwich covariance, and \(V_{\rm O}\) is the optimally weighted overidentified GMM covariance. These intervals use the population covariance targets, not the sample plug-in covariance estimates; their purpose is to diagnose the asymptotic normal approximation and the Godambe width ordering. With $2000$ replications, the Monte Carlo standard error of a single-coordinate $0.95$ coverage estimate is about $0.005$, and the plotted average over coordinates is correspondingly more stable.

Figure~\ref{fig:bivar-ci-diagnostics} reports empirical marginal coverage and mean CI width for $T\ge1000$. Coverage is computed as the fraction of replication--coordinate pairs for which the marginal interval contains the true coordinate. The coverage panel shows that all three methods remain close to the nominal $0.95$ target. The width panel shows the expected \(T^{-1/2}\) shrinkage, with the least-squares intervals visibly wider than the MLE intervals, while the overidentified GMM intervals are almost indistinguishable from MLE.

\begin{figure}[t]
\centering
\includegraphics[width=\textwidth]{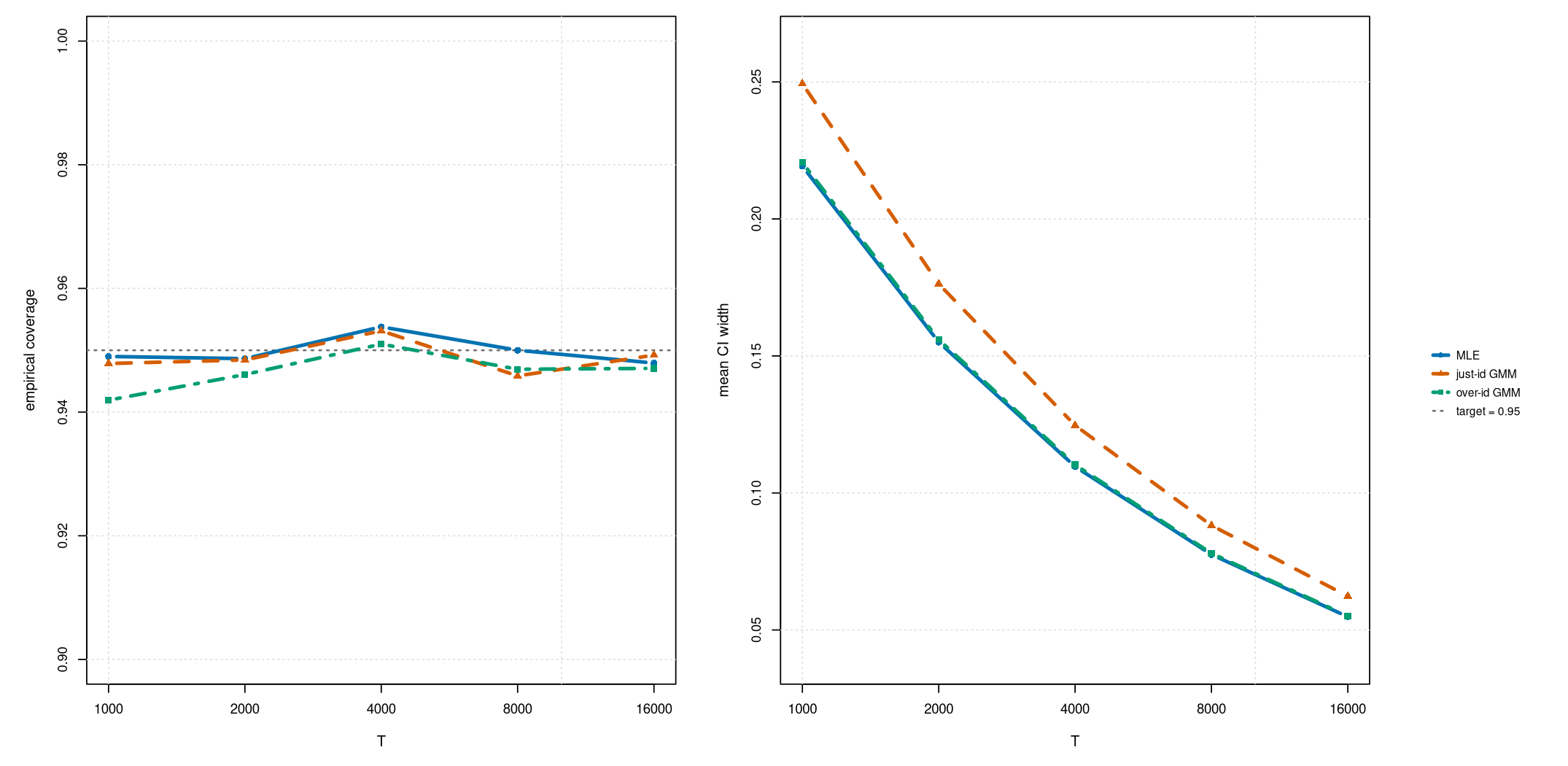}
\caption{Asymptotic normal confidence-interval diagnostics. Left: empirical marginal coverage of nominal $95\%$ intervals, averaged over coordinates and replications. Right: mean asymptotic CI width.}
\label{fig:bivar-ci-diagnostics}
\end{figure}

Figure~\ref{fig:bivar-eigen-inflation} compares the full covariance matrices through the ordered eigenvalues of
\[
    I(\theta_0)^{1/2}V_m I(\theta_0)^{1/2},
    \qquad m\in\{{\rm J},{\rm O}\}.
\]
The MLE benchmark is one in every eigendirection. The least-squares sandwich covariance is inflated in several directions, with the largest eigenvalue around \(1.55\), whereas the overidentified covariance remains close to one throughout the spectrum. This confirms that the efficiency ordering is not an artifact of marginal standard errors alone.

\begin{figure}[t]
\centering
\includegraphics[width=0.78\textwidth]{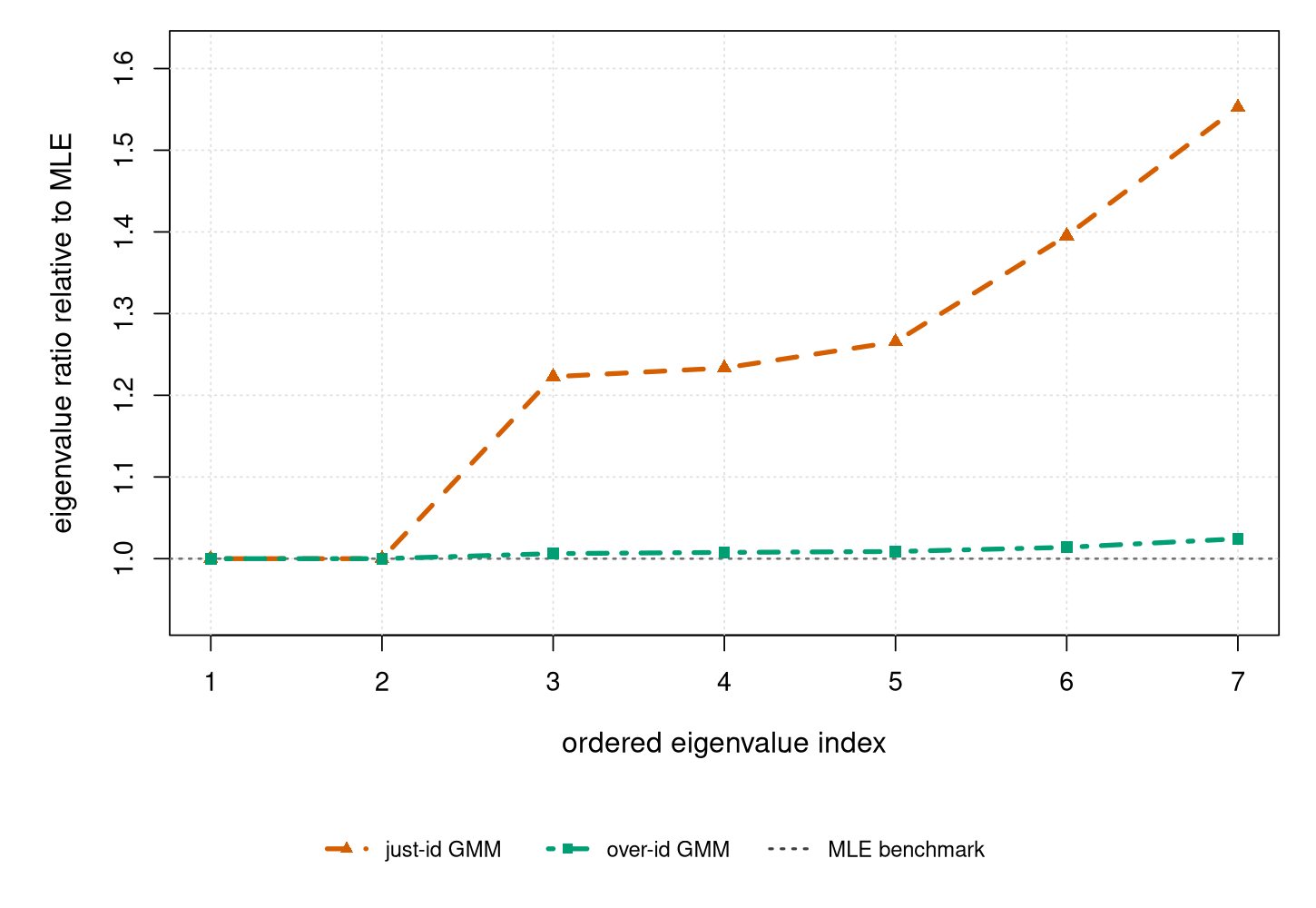}
\caption{Population covariance eigenvalue inflation relative to the MLE. The plotted values are the ordered eigenvalues of \(I(\theta_0)^{1/2}V_m I(\theta_0)^{1/2}\) for the least-squares and overidentified GMM covariance targets.}
\label{fig:bivar-eigen-inflation}
\end{figure}

Together, these diagnostics support the same conclusion as the main simulation figures: all three procedures exhibit the expected asymptotic calibration, while the efficiency ordering is governed by the Godambe covariance. The overidentified moment library nearly recovers the MLE benchmark in this design, whereas the least-squares derivative target retains a non-negligible efficiency loss.


\end{document}